\documentclass{birkjour}
\usepackage{latexsym}
\usepackage{amsfonts}
\usepackage{amsthm}
\usepackage{amsmath}
\usepackage{amssymb}
\usepackage{color}
\usepackage{mathrsfs, mathtools}
\usepackage[normalem]{ulem}
\usepackage[toc,page]{appendix}
\usepackage[colorlinks,citecolor=blue,urlcolor=blue]{hyperref}
\allowdisplaybreaks
\newcommand\dela[1]{}

\numberwithin{equation}{section}

\newcommand{\nn}{\nonumber}

\newcommand{\eps}{\varepsilon}

\newcommand\Eb{\mathbb{E}}

\newcommand\N{\mathbb{N}}

\def\v{\mathrm{v}}
\newcommand\hu{\widehat{u}}
\newcommand\invf{\mathcal{F}^{-1}}
\newcommand\R{\mathbb{R}}

\newcommand\dom{{\mathbb{R}^3}}

\def\rd{\mathbb{R}^d}

\def\E{\mathbb{E}}

\def\tp{\widetilde{\mathbb{P}}}
\def\hp{\widehat{\mathbb{P}}}
\def\tf{\widetilde{\mathcal{F}}}

\def\tom{\widetilde{\Omega}}

\def\rH{\mathrm{H}}
\def\rV{\mathrm{V}}
\def\rU{\mathrm{U}}
\def\rD{\mathrm{D}}

\def\lipzc{2}
\def\divv{\mathrm{div}\,}
\newcommand{\tunk}{{\widetilde{u}}_{{n}_{k}}}
\newcommand{\tu}{\widetilde{u}}
\newcommand{\bcal}{\mathcal{B}}
\newcommand{\ccal}{\mathcal{C}}

\newcommand{\fcal}{\mathcal{F}}
\newcommand{\lcal}{\mathcal{L}}
\newcommand{\ocal}{\mathcal{O}}
\newcommand{\vcal}{\mathcal{V}}
\newcommand{\zcal}{\mathcal{Z}}
\newcommand{\kcal}{\mathcal{K}}
\newcommand{\unk}{{u}_{{n}_{k}}}

\newcommand{\lb}{\langle}
\newcommand{\rb}{\rangle}

\newcommand{\Pn}{{P}_{n}}

\newcommand{\p}{\mathbb{P}}
\newcommand\un{u_n(t)}
\newcommand\uns{u_n(s)}
\newcommand\unr{u_n(r)}
\newcommand\tus{\widetilde{u}(s)}

\newcommand\tuns{\widetilde{u}_n(s)}
\newcommand\tusi{\widetilde{u}(\sigma)}
\newcommand\tunsi{\widetilde{u}_n(\sigma)}
\newcommand\tun{{\widetilde{u}_n}}

\newcommand{\ilsk}[3]{{\langle #1 , #2 \rangle}_{#3}}

\newcommand{\dual}[3]{{\lb #1 , #2 \rb}_{#3}}
\newcommand{\Dual}[3]{{\Bigl< #1  , #2 \Bigr>}_{#3}}

\newcommand{\tOmega}{\widetilde{\Omega}}

\newcommand{\tfcal}{\widetilde{\fcal}}

\newcommand{\tW}{\widetilde{W}}

\def\st{{t \wedge \tau_R^n}}
\DeclareSymbolFont{bbold}{U}{bbold}{m}{n}
\DeclareSymbolFontAlphabet{\mathbbold}{bbold}
\newcommand\ind{\mathbbold{1}}
\newcommand\A{\mathrm{A}}
\newtheorem{theorem}{Theorem}[section]
\newtheorem{corollary}[theorem]{Corollary}
\newtheorem{proposition}[theorem]{Proposition}

\newtheorem{remark}[theorem]{Remark}
\newtheorem{lemma}[theorem]{Lemma}

\newtheorem{definition}[theorem]{Definition}

\newtheoremstyle{AppALem}{1}{1}
  {\itshape}{0pt}{\bfseries}{.}{ }
   {\thmname{Lemma }\thmnumber{A.{#2}}{\thmnote{}}}
   \theoremstyle{AppALem}\newtheorem{lemmaA}{Lemma}

\newtheoremstyle{AppBLem}{1}{1}
  {\itshape}{0pt}{\bfseries}{.}{ }
   {\thmname{Lemma }\thmnumber{B.{#2}}{\thmnote{}}}
   \theoremstyle{AppBLem}\newtheorem{lemmaB}{Lemma}

\newtheoremstyle{AppBRem}{1}{1}
  {\itshape}{0pt}{\bfseries}{.}{ }
   {\thmname{Remark }\thmnumber{B.{#2}}{\thmnote{}}}
   \theoremstyle{AppBRem}\newtheorem{remarkB}{Remark}

\newtheoremstyle{AppBCor}{1}{1}
  {\itshape}{0pt}{\bfseries}{.}{ }
   {\thmname{Corollary }\thmnumber{B.{#2}}{\thmnote{}}}
   \theoremstyle{AppBCor}

\newtheoremstyle{AppBThm}{1}{1}
  {\itshape}{0pt}{\bfseries}{.}{ }
   {\thmname{Theorem }\thmnumber{B.{#2}}{\thmnote{}}}
   \theoremstyle{AppBThm}\newtheorem{theoremB}{Theorem}

\begin{document}
\title[Stochastic tamed Navier-Stokes equations on $\R^3$]{Stochastic tamed Navier-Stokes equations on $\R^3$: the existence and the  uniqueness of solutions and the existence of an invariant measure}

\author[Zdzis\l aw Brze\'{z}niak]{Zdzis\l aw Brze\'{z}niak}
\address{Department of Mathematics, University of York, Heslington, York,
YO10 5DD, UK}
\email{zdzislaw.brzezniak@york.ac.uk}

\author[Gaurav Dhariwal]{Gaurav Dhariwal \\ \small{University of York and Vienna University of Technology}}
\address{Institute of Analysis and Scientific Computing, Vienna University of Technology, Wiedner Hauptstrasse 8--10, 1040 Vienna, Austria}
\email{gaurav.dhariwal@tuwien.ac.at}

\thanks{The research of Gaurav Dhariwal was supported by Department of Mathematics, University of York and partially supported by the Austrian Science Fund (FWF) grants P30000, W1245, and F65. The research of Zdzis\l aw Brze{\'z}niak has been partially supported by the Leverhulme project grant ref no RPG-2012-514.
}

\date{\today}

\begin{abstract}
R\"ockner and Zhang in \cite{[RZ09]} proved the existence of a unique strong solution to a stochastic tamed 3D Navier-Stokes equation in the whole space and for the periodic boundary case using a result from \cite{[SV79]}. In the latter case, they also proved the existence of an invariant measure. In this paper, we improve their results (but for a slightly simplified system) using a self-contained approach. In particular, we generalise their result about an estimate on the $L^4$-norm of the solution from the torus to $\R^3$, see  Lemma~\ref{lemma6.13} and thus  establish the existence of an invariant measure on $\R^3$ for a time-homogeneous damped tamed 3D Navier-Stokes equation, given by \eqref{eq:5.1}.
\end{abstract}
\maketitle

\section{Introduction}
\label{s:1}
In the present paper we study the stochastic tamed Navier-Stokes equations (NSEs) on $\R^3$ which were introduced by R\"ockner and Zhang \cite{[RZ09],[RZ09b]}. We consider the following stochastic tamed NSEs with viscosity $\nu$, $(t,x) \in [0,T] \times \dom$

\begin{equation}
\label{eq:1.1}
\begin{split}
du(t,x) & = \left[ \nu \Delta u(t,x) - (u(t,x)\cdot \nabla)u(t,x) - \nabla p(t,x) \right]dt \\
&\quad - \left[g(|u(t,x)|^2) u(t) - f(x,u(t,x)) \right]dt \\
& \quad + \sum_{j=1}^\infty \left[ (\sigma_{j}(t,x) \cdot \nabla )u(t) + \nabla \widetilde{p}_j(t,x) \right]dW_t^j\,,
\end{split}
\end{equation}
subject to the incompressibility condition
\begin{equation}
\label{eq:1.2}
\rm{div}\,u(t) = 0, \quad t \ge 0\,,
\end{equation}
and the initial condition
\begin{equation}
\label{eq:1.3}
u(0,x) = u_0(x), \quad x \in \dom\,,
\end{equation}
where $p(t,x)$ and $\widetilde{p}_j(t,x)$ are unknown scalar functions.

We assume that the  taming function $g : \R_+ \to \R_+$ is smooth and satisfies,  for some $N \in \mathbb{N}$\dela{and $\alpha \in [0,\frac1\nu)$}, the following condition

\begin{equation}
\label{eq:1.4}
\begin{cases}
g(r) = 0, \quad \quad \quad \quad \quad \,\mbox{if}\,\, r \leq N,\\
g(r) = (r - N)/\nu, \quad \quad \mbox{if}\,\, r \geq N+1,\\
0 \le g^\prime(r) \le \lipzc /(\nu \wedge 1), \,\, r \in [N, N+1]
\end{cases}
\end{equation}

Finally we assume that  $\{W_t^j; t \ge 0, j = 1 ,2, \dots\}$ is a sequence of independent one-dimensional standard $\mathbb{F} = (\mathcal{F}_t)_{t \ge 0}$-Brownian motions on a complete filtered probability space $(\Omega, \mathcal{F}, \mathbb{F}, \mathbb{P})$. The stochastic integral is understood as the It\^o integral. The drift and diffusion coefficients are given as follows:
\begin{align*}
\dom \times \R^3 \ni (x, u) \mapsto f(x,u) \in \R^3 \\
\R_+ \times \dom \ni (t,x) \mapsto \sigma(t,x) \in \R^3 \times \ell^2\,,
\end{align*}
where $\ell^2$ denotes the standard Hilbert space consisting of all square summable sequences of real numbers endowed with standard norm $\|\cdot\|_{\ell^2}$. In the following $f$ and $ \sigma$ are always assumed to be measurable with respect to all their variables.

In the case of classical deterministic Navier-Stokes equations on $\R^3$, if the initial data $u_0 \in \rV$ (see Section~\ref{s:2}), then there exists only a local strong solution, see  \cite{[Temam79]}. Cai and Jiu \cite{[CJ08]} studied the Navier-Stokes equations with damping on $\R^3$, where the damping was modeled by the term $|u|^{\beta - 1} u$, $\beta \ge 1$; the tamed term considered in the current paper corresponds to $\beta = 3$. They proved the existence of a global weak solution\footnote{i.e. the equation is satisfied in the weak sense and the solution $u$ belongs to the space $ L^\infty(0,T; \rH)\cap L^2(0,T; \rV) \cap L^{\beta+1}(0,T; L^{\beta + 1})$.} for any $\beta \ge 1$ and  $u_0 \in \rH$, see Section~\ref{s:2} and they proved the existence of a global strong solution\footnote{A pair  of function   $(u,p)$ is a strong solution iff it is a weak solution  and  $u \in L^\infty(0,T; \rV) \cap L^2(0,T; H^2) \cap L^\infty(0,T; L^{\beta+1})$.} for $\beta \ge 7/2$ and  $u_0 \in \rV \cap L^{\beta+1}$. Moreover, they were  able to show the uniqueness of strong solutions  for $7/2 \le \beta \le 5$. Later, Zhang et al. \cite{[ZWL11]},  by exploiting the Gagliardo-Nirenberg inequality, were able to lower down the parameter $\beta$ to $3$. Thus, establishing the existence of a global strong solution to Navier-Stokes equations with damping on $\R^3$ for $\beta > 3$, $u_0 \in \rV \cap L^{\beta+1}$ and proving uniqueness whenever $3 < \beta \le 5$. They also remarked that the critical value for $\beta$ is $\beta = 3$ \cite[Remark~3.1]{[ZWL11]}. But, Zhou \cite[Theorem~2.1]{[Zhou12]} was able to surpass this critical value of $\beta$.

Moreover, for any $\beta \ge 1$,   he proved that the strong solution is unique in a larger class of weak solutions, see \cite[Theorem~3.1]{[Zhou12]}. The critical case of $\beta = 3$ was studied by R\"ockner and Zhang \cite[Theorem~1.1]{[RZ09b]}, where they proved the existence of a smooth unique global solution to the deterministic tamed 3D Navier-Stokes equations for very smooth initial data and deterministic forcing $f$. Moreover, they proved \cite[Theorem~1.1]{[RZ09b]} that this unique solution converges (in $L^2(0,T; L^2(\mathcal{O}))$) to a bounded Leray-Hopf solution of 3D Navier-Stokes equations (if exists) on a bounded domain $\mathcal{O} \subset \R^3$. The non-explosion of the solution is due to the tamed term. R\"ockner and Zhang \cite{[RZ12]} also studied 3D tamed Navier-Stokes equations on a bounded domain $\mathcal{O} \subset \R^3$ with Dirichlet boundary conditions and proved the existence of a unique strong solution directly, based on the Galerkin approximation and on a kind of local monotonicity of the coefficients. Recently, You \cite{[You17]} proved the existence of a random attractor for the 3D damped ($|u|^{\beta-2} u$) Navier-Stokes equations with additive noise for $4 < \beta \le 6$ with initial data $u_0 \in \rV$ on a bounded domain $\mathcal{O} \subset \R^3$ with smooth boundary.

R\"ockner and Zhang \cite{[RZ09]} proved the existence of a strong solution of the stochastic tamed NSEs (in probabilisitc sense) by invoking the Yamada-Watanabe theorem, thus proving the existence of a martingale solution to \eqref{eq:1.1} (with more generalised noise term) in the absence of compact Sobolev embeddings and the pathwise uniqueness. They used the localization method to prove the tightness, a method introduced by Stroock and Varadhan \cite{[SV79]}. In this paper, we present a self-contained proof of the same result. In order to prove the existence of a martingale solution, R\"ockner et al. used the Faedo-Galerkin approximation with the non-classical finite dimensional space $H_n^1 = \text{span}\{e_i, i = 1 \cdots n\}$ where $\mathcal{E} = \{e_i\}_{i \in \N} \subset \vcal$ (see Section~\ref{s:2}) is the orthonormal basis of $H^1$. They also require that in the case of the periodic boundary conditions, $\mathcal{E}$ is an orthogonal  basis of $H^0$ which was essential in obtaining the $L^4$-estimate of the solution. We generalised this result to $\R^3$. Another reason for R\"ockner et al. to choose the periodic boundary conditions is the compactness of $H^2 \hookrightarrow H^1$ embedding, which along with the $L^4$-estimate of the solution was crucial in establishing the existence of invariant measures. We don't require this embedding and hence are able to obtain the existence of invariant measures for the damped tamed Navier-Stokes equations on $\R^3$.

In the present paper we prove the existence of a unique strong solution to the stochastic tamed 3D Navier-Stokes equation \eqref{eq:1.1} under some natural assumptions $(\textbf{H1}) - (\textbf{H2})$ on the drift $f$ and the diffusion $\sigma$ (see Section~\ref{s:2}). To prove the existence of strong solution we use the Yamada-Watanabe theorem \cite{[Ondrejat04], [YW71]} which states that the existence of martingale solutions plus pathwise uniqueness implies the existence of a unique strong solution. In order to establish the existence of martingale solutions, instead of using the standard Faedo-Galerkin approximations we use a different approach motivated from \cite{[FMRR14]} and \cite{[MP16]}. We study a truncated SPDE on an infinite dimensional space $\rH_n$, defined in the Section~\ref{s:6} and then use the tightness criterion, Jakubowski's generalisation of the Skorohod's theorem and the martingale representaion theorem to prove the existence of martingale solutions. The essential reason for us to incorporate this approximation scheme, is the non-commutativity of gradient operator $\nabla$ with the standard Faedo-Galerkin projection operator $P_n$ \cite[Section~5]{[BD16]}. The commutativity is essential for us to obtain a'priori bounds. We also prove the existence of invariant measures, Theorem~\ref{thm5.4}, for time homogeneous damped tamed Navier-Stokes equations \ref{eq:5.1} under the assumptions $(\textbf{H1})^\prime - (\textbf{H3})^\prime$ (see Section~\ref{s:5}). We use the technique (Theorem~\ref{thm5.3}) of Maslowski and Seidler \cite{[MS99]}, see also \cite{[BF17], [BMO17]}, working with weak topologies to establish the existence of invariant measures. We show the two conditions of Theorem~\ref{thm5.3}, boundedness in probability and sequentially weakly Feller property are satisfied for the semigroup $(T_t)_{t \ge 0}$, defined by \eqref{eq:5.2}. In contrast to the paper by R\"ockner and Zhang \cite{[RZ09]}, a'priori bound on $L^4$-norm of the solution plays an essential role in the existence of martingale solutions and not in the existence of invariant measures.

This paper is organised as follows. In Section~\ref{s:2}, we recall some standard notations and results and set the assumptions on $f$ and $\sigma$. We also establish certain estimates on the tamed term which we use later in Sections~\ref{s:6} and \ref{s:existence of sol}. We end the section by recalling a generalised version of the Gronwall Lemma for random variables from \cite{[DM09]}. In Section~\ref{s:3}, we establish the tightness criterion and state the Jakubowski's generalisation of Skorohod theorem which we use along with a'priori estimates obtained in the Section~\ref{s:existence of sol} to prove the existence of a martingale solution and path-wise uniqueness of the solution. In Section~\ref{s:6}, we introduce our truncated SPDEs and describe the approximation scheme motivated by \cite{[FMRR14],[MP16]}, along with all the machinery required. Finally, in Section~\ref{s:5} we establish the existence of an invariant measure for damped tamed 3D Navier-Stokes equations \eqref{eq:5.1}.

\section{Functional setting}
\label{s:2}

\subsection{Notations and basic definitions}
Let $(X, \|\cdot\|_X), (Y, \|\cdot\|_Y)$ be two real normed spaces. The space of all bounded linear operators from $X$ to $Y$ is denoted by $\mathcal{L}(X,Y)$. If $Y = \R$, then $X^\prime := \mathcal{L}(X, \R)$ is called the dual space of $X$. The standard duality pairing is denoted by ${}_{X^\prime} \langle \cdot, \cdot \rangle_X$. If both spaces are separable Hilbert then by $\mathcal{L}_2(Y;X)$ we will denote the space of all Hilbert-Schmidt operators from $Y$ to $X$ endowed with the standard norm $\|\cdot\|_{\mathcal{L}_2(Y;X)}$.

Assume that $X, Y$ are Hilbert spaces with scalar products $\langle \cdot, \cdot \rangle_X$ and $\langle \cdot, \cdot \rangle_Y$ respectively. For a densely defined linear operator $A : D(A) \to Y$, $D(A) \subset X$,  by $A^\ast$ we denote the adjoint operator of $A$. In particular, $D(A^\ast) \subset Y$, $A^\ast : D(A^\ast) \to X$ and
\[ \langle Ax, y \rangle_Y = \langle x, A^\ast y \rangle_X, \quad x \in D(A),\, y \in D(A^\ast).\]
Note that $D(A^\ast) = Y$ if $A$ is bounded.

Let $C_0^\infty(\dom; \R^3)$ denote the set of all smooth functions from $\dom$ to $\R^3$ with compact supports. For $p \in [1, \infty]$ the Lebesgue spaces of $\R^3$-valued functions will be denoted by $L^p(\dom; \R^3)$, and often by $L^p$ whenever the context is understood.
If $p = 2$, then $L^2(\dom; \R^3)$ is a Hilbert space with the scalar product given by
\[\langle u, \v \rangle_{L^2} := \int_{\dom} u(x)\cdot \v(x)\,dx, \quad u,\v \in L^2(\dom; \R^3)\,.\]
We define, see \cite{[Triebel83]}, the Bessel potential space $H^{s,p}(\dom; \R^3)$ for $s \ge 0$ and $p \in (1,\infty)$ as the space of all $\R^3$-valued functions $u \in L^p(\dom; \R^3)$ such that $\left(1 - \Delta\right)^{s/2} u \in L^p(\dom; \R^3)$, where $\Delta$ is the Laplace operator in $\R^3$.
In particular, $H^1(\dom; \R^3):= H^{1,2}(\dom; \R^3)$ is the space of all $u \in L^2(\dom; \R^3)$ for which the weak derivatives $D_iu \in L^2(\dom; \R^3)$, $i= 1, \dots, 3$. $H^1(\dom; \R^3)$ is a Hilbert space with the scalar product given by
\[\langle u, \v \rangle_{H^1} := \langle u, \v \rangle_{L^2} + ((  u, \v )), \quad u, \v \in H^1(\dom; \R^3)\,,\]
where
\begin{equation}
\label{eq:2.1}
 ((u,\v)):= \langle \nabla u, \nabla \v \rangle_{L^2} = \sum_{i=1}^3\int_\dom \frac{\partial u}{\partial x_i} \cdot \frac{\partial \v}{\partial x_i}\,dx, \quad u, \v \in H^1(\dom; \R^3)\,.
\end{equation}
Finally, for $s \ge 0$, the space  $H^{s,2}=:H^s$ is also a Hilbert space endowed with the norm
\[\|u\|_{H^s} = \left[\int_{\dom}\left(1+|\xi|^2\right)^s |\hat{u}(\xi)|^2\,d\xi\right]^{1/2}\,,\]
where $\hat{u}$ denotes the Fourier transform of a tempered distribution $u$.

Let
\begin{align*}
&\mathcal{V} := \left\{u \in C_0^\infty(\dom; \R^3) : \rm{div}\,u = 0\right\},\\
&\rH := \mbox{the closure of}\,\mathcal{V}\,\mbox{in}\,L^2(\dom; \R^3),\\
&\rV := \mbox{the closure of}\,\mathcal{V}\,\mbox{in}\,H^1(\dom; \R^3),\\
&\rm{D}(\A) := \rm{H} \cap H^2(\dom; \R^3),\,\;\;\;\A u:= \Pi (-\Delta u), \;\; u \in D(\A),
\end{align*}
where $\Pi$ is the orthogonal projection from $L^2(\dom; \R^3)$ to $\rH$. It is known, see e.g. \cite{Kato+Ponce_1986}, that $\Pi=(1-D_iD_j\Delta^{-1})_{i,j=1}^3$ is a pseudodifferential operator 
with matrix symbol $(\delta_{i,j}-D_iD_j\Delta^{-1})_{i,j=1}^3$; $\Pi$ is a bounded linear operator in $H^{s,p}(\mathbb{R}^3)$ by the Marcinkiewicz-Mihlin Theorem \cite{Marcinkiewicz,Mihlin},  for all $s \in \mathbb{R}$ and a $p\in (1,\infty)$.

On $\rH$ we consider the scalar product and the norm inherited from $L^2(\dom; \R^3)$ and denote them by $\langle \cdot, \cdot \rangle_{\rH}$ and $|\cdot|_{\rH}$ respectively, i.e.
\[\langle u, \v \rangle_{\rH} := \langle u, \v \rangle_{L^2}, \quad \quad |u|_{\rH} := |u|_{L^2}, \quad u,\v \in \rH\,.\]
On $\rV$ we consider the scalar product and norm inherited from $H^1(\dom; \R^3)$, i.e.
\begin{equation}
\label{eq:2.2}
\langle u , \v \rangle_{\rV} := \langle u, \v \rangle_{L^2} + ((u,\v)), \quad \quad \|u\|^2_{\rV} := |u|^2_{\rH} + |\nabla u|^2_{L^2}, \quad u, \v \in \rV\,,
\end{equation}
where $((\cdot, \cdot))$ is defined in \eqref{eq:2.1}. $\rm{D}(\A)$ is a Hilbert space under the graph norm
\[|u|^2_{\rm{D}(\A)} := |u|^2_{\rH} + |\A\,u|^2_{L^2}, \quad \quad u \in \rm{D}(\A)\,,\]
where the inner product is given by
\[\langle u, \v \rangle_{\rm{D}(\A)} := \langle u,\v\rangle_\rH + \langle \A u, \A \v\rangle_{L^2}, \quad \quad u,\v \in \rm{D}(\A)\,.\]

\subsection{Some operators}
\label{s:2.2}

Let us consider the following tri-linear form
\begin{equation}
\label{eq:2.3}
b(u,w,\v) = \int_\dom (u \cdot \nabla w)\v \,dx,
\end{equation}
defined for suitable vector fields $u, \v, w$ on $\R^3$. We will recall the fundamental properties of the form $b$ which are valid in unbounded domains.

By the Sobolev embedding theorem and H\"older inequality, we obtain the following estimates
\begin{align}
\label{eq:2.4}
|b(u,w,\v)| & \le \|u\|_{L^4}\|w\|_{\rV}\|\v\|_{L^4}, \quad u, \v \in L^4, w \in \rV, \\
\label{eq:2.5}
& \le c \|u\|_{\rV}\|w\|_{\rV}\|\v\|_\rV, \quad u, \v, w \in \rV,
\end{align}
for some positive constant $c$. Thus the form $b$ is continuous on $\rV$. Moreover, if we define a bilinear map $B$ by $B(u,w) := b(u,w, \cdot)$, then by inequality \eqref{eq:2.5} we infer that $B(u,w) \in \rV^\prime$ for all $u, w \in \rV$ and that the following inequality holds
\begin{equation}
\label{eq:2.6}
|B(u,w)|_{\rV^\prime} \le c \|u\|_{\rV}\|w\|_{\rV}, \quad u,w \in \rV.
\end{equation}
Moreover, the mapping $B : \rV \times \rV \to \rV^\prime$ is bilinear and continuous.

Let us, for any $s > 0$, define the following standard scale of Hilbert spaces
\begin{equation}
\label{eq:scaled spaces}
\rV_s := \mbox{the closure of }\, \mathcal{V}\, \mbox{ in }\, H^s(\dom; \R^3)\,.
\end{equation}
If $ s > \frac{d}{2} + 1$ then by the Sobolev Embedding theorem,
\begin{equation}
\label{eq:2.7}
H^{s-1}(\R^d; \R^3) \hookrightarrow C_b(\R^d; \R^3) \hookrightarrow L^\infty(\R^d; \R^3)\,.
\end{equation}
Here $C_b(\R^d; \R^3)$ denotes the space of continuous and bounded $\R^3$-valued functions defined on $\R^d$. If $u, w \in \rV$ and $\v \in \rV_s$ with $s > \frac32$ then
\[|b(u,w,\v)| = |b(u,\v,w)| \le |u|_{L^2}|w|_{L^2}\|\nabla \v\|_{L^\infty} \le c |u|_{L^2} |w|_{L^2}\|\v\|_{\rV_s}\]
for some constant $c > 0$. Thus $b$ can be uniquely extended to the tri-linear form (denoted by the same letter)
\[b: \rH \times \rH \times \rV_s \to \R\]
and
\[|b(u,w,\v)| \le c |u|_{L^2} |w|_{L^2}\|\v\|_{\rV_s}, \quad u, w \in \rH, \v \in \rV_s\,.\]
At the same time, the operator $B$ can be uniquely extended to a bounded linear operator
\[B : \rH \times \rH \to \rV_s^\prime\,.\]
In particular, it satisfies the following estimate
\begin{equation}
\label{eq:2.8}
|B(u,w)|_{\rV_s^\prime} \le c |u|_\rH |w|_\rH, \quad u,w \in \rH, \qquad s > \tfrac32.
\end{equation}
We will also use the following notation, $B(u) := B(u,u)$.

Let us assume that $s > 1$. It is clear that $\rV_s$ is dense in $\rV$ and the embedding $j_s : \rV_s \hookrightarrow \rV$ is continuous. Then there exists \cite[Lemma~C.1]{[BM13]} a Hilbert space $U$ such that $U \subset \rV_s$, $U$ is dense in $\rV_s$ and
\[\mbox{the natural embedding }\, i_s : U \hookrightarrow \rV_s\,\mbox{ is compact}\,.\]
Therefore, the following embedding of the spaces hold
\[\rU \hookrightarrow \rV_s \hookrightarrow \rV \hookrightarrow \rH \equiv \rH^\prime \hookrightarrow \rU^\prime .\]

The following Gagliardo-Nirenberg interpolation inequality will be used frequently. Let $q \in [1, \infty]$ and $m \in \N$. If
\[ \frac{1}{q} = \frac12 - \frac{m \alpha}{3}, \quad 0 \le \alpha \le 1\,,\]
then there exists a constant $C_{m,q}$ depending on $m$ and $q$ such that
\begin{equation}
\label{eq:2.9}
\|u\|_{L^q} \le C_{m,q} \|u\|^\alpha_{H^m}|u|^{1-\alpha}_{L^2}\,, \qquad \mbox{for }\, u \in H^m\,.
\end{equation}

Recall that $\Pi$ is the orthogonal projection from $L^2(\dom; \R^3)$ to $\rH$. For any $u \in \rH$ and $\v \in L^2(\dom; \R^3)$, we have
\[\langle u, \v \rangle_{\rH} := \langle u, \Pi \v \rangle_{\rH} = \langle u, \v \rangle_{L^2}\,.\]

The Stokes operator $\A \colon \rm{D}(\A) \to \rH,$ is given by
\begin{align*}
\A\,u &= - \Pi(\Delta \, u),\quad  \quad u \in \mathrm{D}(\A)\,, \\
\mathrm{D}(\A) &= \rH \cap H^2(\dom; \R^3)\,.
\end{align*}

\subsection{Assumptions}
\label{s:2.3}

We now introduce the following assumptions on the coefficients $f$ and $\sigma$:
\begin{trivlist}
\item{(\textbf{H1})} $f:\dom\times  \R^3 \to \R^3$ is a continuous function such that there exists a constant $C_{f} > 0$ and $b_f \in L^1(\R^3)$ such that for all $x \in \dom$, $u \in \R^3$,
\[
|f(x,u)|^2  \le C_{f}|u|^2 + b_f(x).
\]
Moreover, for $u_1, u_2 \in \R^3$
\[|f(x,u_1) - f(x,u_2)| \le C_f |u_1 - u_2|, \qquad \forall\, x \in \R^3.\]
\item{(\textbf{H2})} A  measurable  function $\sigma :[0,\infty)\times \dom \to \ell^2 $ of $C^1$ class with respect to the $x$-variable and for any $T > 0$ there exists a constant  $C_{\sigma,T} > 0$ such that for all $t \in [0,T], x \in \dom$
\[\|\partial_{x^j} \sigma(t,x)\|_{\ell^2} \le C_{\sigma, T}, \quad j = 1,2,3\]
and, for all $t \in [0,\infty), x \in \dom$,
\begin{equation}
\label{eq:2.10}
\|\sigma(t,x)\|^2_{\ell^2} \le \frac14\,.
\end{equation}
\end{trivlist}
Below for the sake of simplicity the variable $``x"$ in the coefficients will be dropped.

Define, for $j \in \N$, a map $G_j : [0,T] \times \rH \to \rm{H}$ by
\begin{equation}
\label{eq:2.11}
G_j(t,u) := \Pi[(\sigma_j(t)\cdot \nabla)u\,]\,.
\end{equation}
Then $G : \rH \to \mathcal{L}_2(\ell^2; \rH)$ is given by
\begin{equation}
\label{eq:2.12}
G(u)(k) = \sum_{j=1}^\infty k_j G_j(u)\,, \qquad u \in \rH,\, k \in \ell^2\,.
\end{equation}
Let $\{e_j\}_{j = 1}^\infty$ be the orthonormal basis of $\ell^2$ then
\[G(t, u)(e_j) = G_j(u)\,, \quad t \ge 0,\, u \in \rH\,.\]

For simplicity we will assume that $\nu = 1$. In particular, the function $g$ defined by \eqref{eq:1.4} will from now on be given by
\begin{equation}
\label{eq:g_new}
\begin{cases}
g(r) = 0, \quad \quad \quad \quad \,\mbox{if}\,\, r \leq N,\\
g(r) = (r - N), \quad \,\,\,\, \mbox{if}\,\, r \geq N+1,\\
0 \le g^\prime(r) \le \lipzc , \, \quad \quad r \in [N, N+1]\,.
\end{cases}
\end{equation}
Observe that the function $g$ defined in this way satisfies
\begin{equation}
\label{eq:g_bounded}
|g(r)| \le r, \quad \quad r \ge 0\,,
\end{equation}
and
\begin{equation}
\label{eq:g_lipschitz}
|g(r) - g(r^\prime)| \le \lipzc |r - r^\prime|, \quad \quad r, r^\prime \ge 0\,.
\end{equation}

We are interested in proving the existence of solutions to \eqref{eq:1.1} - \eqref{eq:1.3}. In particular, we want to prove the existence of a random divergence free vector field $u$ and scalar pressure $p$ satisfying \eqref{eq:1.1} and \eqref{eq:1.3}. Thus we project equation \eqref{eq:1.1} using the orthogonal projection operator $\Pi$ on the space $\rH$ of the $L^2$-valued, divergence free vector fields. On projecting, we obtain the following abstract stochastic evolution equation:
\begin{equation}
\label{eq:2.13}
\begin{cases}
du(t) = \left[- \A u(t) - B(u(t)) - \Pi [g(|u(t)|^2)u(t)] + \Pi f(u(t))\right]dt \\
 \hspace{1.5truecm}+ \sum_{j=1}^\infty G_j(t,u(t))\,dW^j(t), \\
u(0) = u_0 ,
\end{cases}
\end{equation}
where we assume that $u_0\in \rV$ and $W(t) = (W^j(t))_{j=1}^\infty$ is a cylindrical Wiener process on $\ell^2$ and $\{W^j(t), t \ge 0, j \in \N\}$ is an infinite sequence of independent standard Brownian motions. We will repeatedly use the following notation
\[G(t,u)\,dW(t) := \sum_{j=1}^\infty G_j(t,u)\,dW^j(t)\,.\]

We will need the following lemma later to obtain the a'priori estimates.

\begin{lemma}
\label{lemma2.1}
\begin{trivlist}
\item{i)} For any $u \in \rm{D}(\A )$
\begin{equation}
\label{eq:2.14}
|\langle B(u), u \rangle_{\rV}| \le \frac12 |u|^2_{\rm{D}(\A )} + \frac12 \big ||u|\cdot|\nabla u|\big |^2_{L^2}\,.
\end{equation}
\item{ii)} If $u \in \rH$, then
\begin{equation}
\label{eq:g_estimate}
\begin{cases}
((-g(|u|^2)u,u)) \le C_N |\nabla u|^2_{L^2} - 2 \big||u|\cdot|\nabla u|\big|^2_{L^2},\\
\langle -g(|u|^2)u, u \rangle_{L^2} \le - \|u\|^4_{L^4} + C_N|u|^2_\rH,
\end{cases}
\end{equation}
where $((\cdot, \cdot))$ is defined in \eqref{eq:2.1} and $C_N > 0$ is a generic constant depending on $N$.
\item{iii)} If $T > 0$ then $\exists$ $C_{\sigma, T} > 0$ such that for any $ t \in [0,T]$ and $u \in \rm{D}(\A )$,
\begin{align}
\label{eq:2.15}
\|G(t,u)\|^2_{\mathcal{L}_2(\ell^2;\rH)} & \le \frac14 |\nabla u(t)|^2_{L^2}\,, \\
\label{eq:2.16}
\|G(t,u)\|^2_{\mathcal{L}_2(\ell^2;\rV)} & \le \frac12 |\A \,u(t)|^2_{L^2} + C_{\sigma,T}|\nabla u(t)|^2_{L^2}\,.
\end{align}
\end{trivlist}
\end{lemma}

\begin{proof}
Let $u \in \rm{D}(\A )$. Since $\langle B(u), u \rangle_\rH = 0,$ using the Cauchy-Schwarz and the Young's inequality we get
\begin{align*}
&|\langle B(u), u \rangle_{\rV}|  = |\langle B(u), (I - \Delta)u \rangle_{\rH}|  \le |-\Delta\,u|_{L^2}|(u \cdot \nabla)u |_{L^2} \\
& \quad \le \frac12|- \Delta \,u|_{L^2}^2 + \frac12|(u\cdot \nabla)u|^2_{L^2} \le \frac12 |u|^2_{\rm{D}(\A )} + \frac12 \big ||u|\cdot|\nabla u|\big|^2_{L^2}.
\end{align*}
Let us introduce a function $\phi \colon \R_+ \to \R$ such that $g(r) = r - \phi(r)$ which can be written as
\begin{equation*}
\phi(r) =
\begin{cases}
r, \quad \quad r \le N, \\
N, \quad \quad r \ge N + 1.
\end{cases}
\end{equation*}
Also $\phi^\prime(r) = 1 - g^\prime(r)$, and there exists a constant $\widetilde C_N > 0$\footnote{One can show that the constant $\widetilde C_N$ is independent of $N$. Moreover, $|\phi^\prime(r)| \le 1$ and thus $|\phi^\prime(r)r| \le (N+1)$ for every $r > 0$.} such that $|\phi^\prime(r)| \le \widetilde{C}_N$ for every $r \ge 0 $. Moreover
\begin{equation*}
\phi^\prime(r) =
\begin{cases}
1, \quad \quad r \le N, \\
0, \quad \quad r \ge N+1.
\end{cases}
\end{equation*}
Thus $|\phi^\prime(r) r|$ is bounded by some positive constant $C_N$. Let $u \in \rH$, then using the definitions of $g$ and $((\cdot, \cdot))$ we get
\begin{align*}
& ((-g(|u|^2)u, u))  = -\langle g(|u|^2)u, - \Delta\,u\rangle_{L^2} \\
&\; = - \int_{\dom} g(|u(x)|^2)u(x) \cdot \left(- \Delta\,u(x)\right)\,dx \\
& \; = - \int_{\dom} |u(x)|^2 u(x)\cdot \left(-\Delta\,u(x)\right)\,dx + \int_{\dom} \phi(|u(x)|^2) u(x) \cdot \left(-\Delta\,u(x) \right)\,dx\,.
\end{align*}
Thus, by the integration by parts formula we get
\begin{align}
\label{eq:g_estimate 1}
 ((-g(|u|^2)u, u)) = & - \int_{\dom} |u(x)|^2 |\nabla u(x)|^2 dx - 2 \int_{\dom} |u(x)|^2 |\nabla u(x)|^2 dx \nonumber \\
&\,\,\, + \sum_{j,k = 1}^3 \int_{\dom}D_k\left(\phi(|u(x)|^2)\right) u_j(x) D_k u_j(x)\,dx \nonumber \\
& \,\,\, + \int_{\dom} \phi(|u(x)|^2)|\nabla u(x)|^2\,dx \;.
\end{align}
Using the above bound on $|\phi^\prime(r) r |$, we obtain
\begin{align}
\label{eq:g_estimate 2}
\sum_{j,k = 1}^3 & \int_{\dom}D_k\left(\phi(|u(x)|^2)\right) u_j(x)D_k u_j(x)\,dx \nonumber \\
& = 2 \sum_{k=1}^3 \int_{\dom} \phi^\prime(|u(x)|^2) \langle u(x), D_k u(x) \rangle_{\dom}^2 \, dx \nonumber \\
& \le 2 \int_{\dom} |\phi^\prime(|u(x)|^2)| |u(x)|^2 |\nabla u(x)|^2\,dx \le C_N |\nabla u|^2_{L^2}\,.
\end{align}
Since $g(r) \ge 0$, $|\phi(r)| \le r$ for all $r \ge 0$. Thus using \eqref{eq:g_estimate 2} in \eqref{eq:g_estimate 1}, we obtain
\begin{align*}
((-g(|u|^2)u, u))  &\le -3 \big| |u|\cdot|\nabla u| \big|^2_{L^2} + C_N |\nabla u|^2_{L^2} + \big| |u| \cdot |\nabla u| \big|^2_{L^2} \\
&  = C_N |\nabla u|^2_{L^2} - 2 \big| |u| \cdot |\nabla u| \big|^2_{L^2}\,.
\end{align*}

Now to prove the second inequality in \eqref{eq:g_estimate}, we take the similar approach. Let $u \in \rH$, then
\begin{align*}
\langle - g(|u|^2)u,u \rangle_{L^2} = - \int_{\dom}|u(x)|^2 |u(x)|^2\,dx + \int_{\dom} \phi(|u(x)|^2)|u(x)|^2\,dx\,.
\end{align*}
By the definition of $\phi$ there exists a constant $C_N > 0$ such that $|\phi(r)| \le C_N$ for all $r > 0$, thus
\begin{align*}
\langle - g(|u|^2)u,u \rangle_{L^2} \le - \|u\|^4_{L^4} + C_N |u|^2_\rH\,.
\end{align*}
This completes the proof of part (ii).

Now for (iii), by $(\textbf{H1})$ and $(\textbf{H2})$ we have
\begin{align*}
&\|G(t,u)\|^2_{\mathcal{L}_2(\ell^2;\rH)}  = \sum_{j=1}^\infty|G_j(t,u)|^2_\rH = \sum_{j = 1}^\infty \int_\dom |G_j(t,x,u(t,x))|^2\,dx \\
& \quad \quad \le \int_{\R^3} \|\sigma(t,x)\|^2_{\ell^2}|\nabla u(t,x)|^2\,dx \le \left(\sup_{ x \in \R^3} \|\sigma(t,x)\|^2_{\ell^2}\right) |\nabla u(t)|^2_{L^2} \\
& \quad \quad  \le \frac14 |\nabla u(t)|^2_{L^2}\,.
\end{align*}
Secondly, noting that
\[\|G(t,u)\|^2_{\mathcal{L}_2(\ell^2;\rV)} = \|G(t,u)\|^2_{\mathcal{L}_2(\ell^2;\rH)} + \|\nabla G(t,u)\|^2_{\mathcal{L}_2(\ell^2;\rH)}\]
and
\begin{align*}
\partial_{x^j}G_k(t,u) & = \Pi \partial_{x^j} [(\sigma_k(t,x) \cdot \nabla )u(t)]\\
& = \Pi\left[ (\partial_{x^j} \sigma_k(t,x) \cdot \nabla)u(t) + (\sigma_k(t,x) \cdot \nabla) \partial_{x^j}u(t) \right]
\end{align*}
by $(\textbf{H1})$ and $(\textbf{H2})$ we have
\[\|G(t,u)\|^2_{\mathcal{L}_2(\ell^2;\rV)}  \le \frac12 |\A \,u(t)|^2_{L^2} + C_{T,\sigma}|\nabla u(t)|^2_{L^2}\,.\]
\end{proof}
On a purely heuristic level, by application of the It\^o formula to the function $\rH \ni \xi \mapsto |\xi|^2_\rH \in \R$ and a solution $u$ to \eqref{eq:1.1} and using Lemma~\ref{lemma2.1}, one obtains the following inequality
\begin{align}\label{eq:energy inequality}
\nonumber
&\frac12 d|u(t)|_{\rm{H}}^2 = \langle  u(t), - \A  u(t) - B(u(t)) - g(|u(t)|^2)u(t) + f(u(t)) \rangle_{\rH}   \\
&\; + \left\langle  u(t), G(t,u(t)) \, dW_t \right\rangle_{\rm{H}} + \frac12 \|G(t,u_n(t))\|^2_{\mathcal{L}_2(\ell^2;\rH)}
\\
&\; \le - \frac78 |\nabla u(t)|^2_{L^2} - \|u(t)\|^4_{L^4} + C_{N,f} |u(t)|^2_\rH + \left\langle  u(t), G(t,u(t)) \, dW_t \right\rangle_{\rm{H}},
\nonumber
\end{align}
which could lead to a'priori estimates that can be used further to prove the existence of the solution.

We will require the following version of the Gronwall Lemma \cite[Lemma~3.9]{[DM09]}$\colon$

\begin{lemma}
\label{lemma2.2}
Let $X,Y, I$ and $\varphi$ be non-negative processes and $Z$ be a non-negative integrable random variable. Assume that $I$ is non-decreasing and there exist non-negative constants $C,\alpha, \beta, \gamma, \eta$ with the following properties
\begin{equation}
\label{eq:2.17}
\int_0^T \varphi(s)\,ds \le C \; a.s., \quad \quad 2 \beta e^C \le 1, \quad \quad 2 \eta e^C \le \alpha,
\end{equation}
and such that for some constant $\widetilde{C} > 0$  and all  $ t \in [0,T]$,
\begin{equation}
\label{eq:2.18}
X(t) + \alpha Y(t) \le Z + \int_0^t \varphi(r) X(r)\,dr + I(t) \quad a.s.,
\end{equation}
\begin{equation}
\label{eq:2.19}
\E(I(t)) \le \beta \E(X(t)) + \gamma \int_0^t \E(X(s))\,ds + \eta \E(Y(t)) + \widetilde{C}.
\end{equation}
 If $X \in L^\infty([0,T] \times \Omega)$, then we have
\begin{equation}
\label{eq:2.20}
\E \left[X(t) + \alpha Y(t) \right] \le 2 \exp{\left( C + 2 t \gamma e^C \right)}\left(\E(Z) + \widetilde{C}\right), \qquad t \in [0,T]\,.
\end{equation}
\end{lemma}

\section{Compactness}
\label{s:3}

Let $(\ocal_R)_{R\in \N}$ be a sequence of bounded open subsets of ${\R^3}$ with regular boundaries $\partial \ocal_R$ such that $\ocal_R \subset \ocal_{R+1}$. Let us consider the following functional spaces, which were already introduced in \cite{[BM13]}:

\noindent $\ccal([0,T]; \mathrm{U}') :=$ the vector space of continuous functions $u: [0,T] \to \mathrm{U}^\prime$, where $\rU^\prime$ is a vector space, with the topology $\mathcal{T}_1$ induced by the norm $|u|_{\ccal([0,T]; \mathrm{U}^\prime)} := \sup_{t \in [0,T]}|u(t)|_{\mathrm{U}^\prime}$,

\noindent $L^2_w(0,T; \mathrm{D}(\A)) :=$ is the Hilbert space $L^2(0,T; \mathrm{D}(\A))$ with the weak topology $\mathcal{T}_2$,

\noindent $L^2(0,T; \rH_{loc}) :=$ the space of measurable functions $u : [0,T] \to \mathrm{H}$ such that for all $R \in \N$
\begin{equation}
\label{E:seminorms}
q_{T,R}(u) := \|u\|_{L^2(0,T; H_{\ocal_R})} = \left(\int_0^T \int_{\ocal_R} |u(t,x)|^2\,dx \,dt \right)^{1/2} < \infty\,,
\end{equation}
with the topology $\mathcal{T}_3$ induced by the semi-norms $(q_{T,R})_{R \in \N}$.

The following lemma is inspired by the classical Dubinsky theorem \cite[Theorem~IV.4.1]{[VF88]} (see also \cite{[Lions69]}) and the compactness result due to Mikulevicus and Rozovskii \cite[Lemma~2.7]{[MR05]}.

\begin{lemma}
\label{lemma6.3.1}
Let
\begin{equation}
\label{eq:6.3.1}
\widetilde{\zcal}_T := \ccal([0,T]; \mathrm{U}^\prime) \cap L^2_w(0,T; \rD(\A)) \cap L^2(0,T; \rH_{loc})
\end{equation}
and let $\widetilde{\mathcal{T}}$\footnote{$\widetilde{\mathcal{T}}$ is the supremum of topologies $\mathcal{T}_1, \mathcal{T}_2, \mathcal{T}_3$, i.e. it is the coarsest topology on $\widetilde{\zcal}_T$ that is finer than each of $\mathcal{T}_1, \mathcal{T}_2$ and $\mathcal{T}_3$.} be the supremum of the corresponding topologies. Then a set $\kcal \subset \widetilde{\zcal}_T$ is $\widetilde{\mathcal{T}}$-relatively compact if the following two conditions hold $\colon$
\begin{itemize}
\item[(i)] $\sup_{u \in \kcal} \int_0^T |u(s)|^2_{\rD(\A)}\,ds < \infty\,,\,$ i.e. $\kcal$ is bounded in $L^2(0,T; \rD(\A))\,,$
\item[(ii)] $\lim_{\delta \to 0} \sup_{u \in \kcal} \sup_{\overset{s,t \in[0,T]}{|t-s|\le \delta }} {|u(t)-u(s)|}_{\mathrm{U}^{\prime }} =0\,.$
\end{itemize}
\end{lemma}
\noindent The above lemma can be proved by modifying the proof of \cite[Lemma~3.1]{[BM13]}, see also \cite[Theorem~IV.4.1]{[VF88]}.

\noindent Let $\rV_w$ denote the Hilbert space $\rV$ endowed with the weak topology.\\
\noindent $\ccal([0,T]; \mathrm{V}_w) :=$ the space of weakly continuous functions $u: [0,T] \to \mathrm{V}$ endowed with the weakest topology $\mathcal{T}_4$ such that for all $h \in \mathrm{V}$ the mappings
\[\ccal([0,T]; \mathrm{V}_w) \ni u \to \langle u(\cdot), h \rangle_{\mathrm{V}} \in \ccal([0,T]; \R)\]
are continuous. In particular, $u_n \to u$ in $\ccal([0,T]; \rV_w)$ iff for all $h \in \rV \colon$
\[\lim_{n \to \infty} \sup_{t \in [0,T]} \left|\langle u_n(t) - u(t), h \rangle_\rV \right| = 0\,.\]

Consider the ball
\begin{align}
\label{eqn-ball}
\mathbb{B} := \{ x \in \rV : \|x\|_{\rV} \le r \}.
\end{align}

Let $q$ be the metric compatible with the weak topology on $\mathbb{B}$. Let us recall the following subspace of the space $\ccal([0,T]; \rV_w)$
\begin{align}
\label{eq:6.3.2}
\ccal([0,T]; \mathbb{B}_w) := \,\, & \mbox{ the space of weakly continuous functions }  \nonumber \\
& u \colon [0,T] \to \rV \mbox{ such that } \sup_{t \in [0,T]}\|u(t)\|_\rV \le r\,.
\end{align}
The space $\ccal([0,T]; \mathbb{B}_w)$ is metrizable with metric
\begin{equation}
\label{eq:6.3.3}
\varrho(u,\v) = \sup_{t \in [0,T]} q(u(t),\v(t))\,.
\end{equation}
Since by the Banach-Alaoglu theorem \cite{[Rudin91]}, the set $\mathbb{B}_w$ is compact, $(\ccal([0,T]; \mathbb{B}_w), \varrho)$ is a complete metric space.

The following lemma says that any sequence $(u_n)_{n \in \N} \subset L^\infty([0,T]; \mathbb{B})$, convergent in $\ccal([0,T]; \rU^\prime)$ is also convergent in the space $\ccal([0,T]; \mathbb{B}_w)$. The proof of the lemma is similar to the proof of \cite[Lemma~2.1]{[BM14]} (see also \cite[Lemma~4.1]{[BHW19]}).

\begin{lemma}
\label{lemma6.3.2}
Let $r > 0$ and $\left(u_n\right)_{n\in\N} \subset L^\infty(0,T; \rV)$ be a sequence of functions such that
\begin{itemize}
\item[(i)] $\sup_{n \in \N} \|u_n(s)\|_{L^\infty(0,T; \rV)} \le r\,,$
\item[(ii)] $u_n \to u$ in $\ccal([0,T]; \rU^\prime)\,.$
\end{itemize}
Then $u,u_n \in \ccal([0,T]; \mathbb{B}_w)$ and $u_n \to u$ in $\ccal ([0,T]; \mathbb{B}_w)$ as $n \to \infty$.
\end{lemma}

\begin{proof}
The lemma can be proved by following the steps of the proof of Lemma~4.1 \cite{[BHW19]} for our choice of functional spaces.
\end{proof}
Let
\begin{equation}
\label{eq:3.1}
\mathcal{Z}_T := \ccal([0,T]; \mathrm{U}^\prime) \cap L^2_w(0,T; \mathrm{D}(\A)) \cap L^2(0,T; \rH_{loc}) \cap \ccal([0,T]; \mathrm{V}_w),
\end{equation}
and let $\mathcal{T}$ be the supremum of the corresponding topologies.

Now we formulate the compactness criterion analogous to the result due to Mikulevicus and Rozowskii \cite{[MR05]}, Brze\'zniak and Motyl \cite[Lemma~3.3]{[BM13]} for the space $\mathcal{Z}_T$ .

\begin{lemma}
\label{lemma6.3.3}
Let $\left(\mathcal{Z}_T, \mathcal{T}\right)$ be as defined in \eqref{eq:3.1}. Then a set $\kcal \subset \mathcal{Z}_T$ is $\mathcal{T}$-relatively compact if the following three conditions hold
\begin{itemize}
\item[(a)] $\sup_{u \in \kcal} \sup_{s \in [0,T]} \|u(s)\|_\rV < \infty\,,$
\item[(b)] $\sup_{u \in \kcal} \int_0^T |u(s)|^2_{\mathrm{D}(\A)}\,ds < \infty\,$, i.e. $\kcal$ is bounded in $L^2(0,T; \mathrm{D}(\A))\,,$
\item[(c)] $\lim_{\delta \to 0} \sup_{u \in \kcal} \sup_{\underset{|t-s| \le \delta}{s,t \in [0,T]}}|u(t) - u(s)|_{\rH} = 0\,.$
\end{itemize}
\end{lemma}

\begin{proof}  Let us denote  $r = \sup_{u \in \kcal}\sup_{s \in [0,T]}\|u(s)\|_\rV$, which in view of assumption $(c)$, is $<\infty$. Define the ball $\mathbb{B}$  of radius $r$  by \eqref{eqn-ball}.

Let us notice that $\zcal_T = \widetilde{\zcal}_T \cap \ccal([0,T]; \rV_w)$, where $\widetilde{\zcal}_T$ is defined by \eqref{eq:6.3.1}. Let $\kcal$ be a subset of $\mathcal{Z}_T$. Because of the assumption $(a)$ we may consider the metric space $\ccal([0,T]; \mathbb{B}_w) \subset \ccal([0,T]; \rV_w)$ defined by \eqref{eq:6.3.2} and \eqref{eq:6.3.3}. Because of the assumption $(b)$,  the restriction to $\kcal$ of the weak topology in $L^2(0,T; \mathrm{D}(\A))$ is metrizable. Since the restrictions to $\kcal$ of the four topologies considered in $\mathcal{Z}_T$ are metrizable, compactness of a subset of $\mathcal{Z}_T$ is equivalent to its sequential compactness.


Let $(u_n)$ be a sequence in $\kcal$. By Lemma \ref{lemma6.3.1}, the boundedness of the set $\kcal $ in ${L}^{2}(0,T; \rD(\A))$ and assumption $(c)$ imply that $\kcal $ is compact in $\widetilde{\zcal}_T$. 
Hence in particular, there exists a subsequence, still denoted by $(u_n)$, convergent in $\ccal([0,T];\rU^\prime)$. Therefore by Lemma~\ref{lemma6.3.2} and assumption $(a)$, $(u_n)$ is convergent in $\ccal([0,T]; \mathbb{B}_w)$. This completes the proof of the lemma.
\end{proof}

\subsection{Tightness}
\label{s:3.1}
Let $(\mathbb{S}, \varrho)$ be a separable and complete metric space.

\begin{definition}
\label{defn3.1}
Let $u \in \ccal([0,T]; \mathbb{S})$. The modulus of continuity of $u$ on $[0,T]$ is defined by
\[m(u, \delta) :=  \sup_{s,t \in [0,T],\,|t - s|\le \delta} \varrho (u(t), u(s)), \quad \delta > 0.\]
\end{definition}

Let $(\Omega, \mathcal{F}, \mathbb{P})$ be a probability space with filtration $\mathbb{F}:= (\mathcal{F}_t)_{t \in [0,T]}$ satisfying the usual conditions, see \cite{[Metivier82]}, and let $(X_n)_{n \in \mathbb{N}}$ be a sequence of continuous $\mathbb{F}$-adapted $\mathbb{S}$-valued processes.

\begin{definition}
\label{defn3.2}
We say that the sequence $(X_n)_{n \in \mathbb{N}}$ of $\ccal([0,T]; \mathbb{S})$-valued random variables satisfies condition $[\mathbf{T}]$ iff $\forall\, \varepsilon >0, \forall\, \eta > 0,\, \exists\, \delta > 0\colon$
\begin{equation}
\label{eq:3.2}
\sup_{n \in \mathbb{N}} \mathbb{P}\left\{m(X_n, \delta) > \eta\right\} \le \varepsilon.
\end{equation}
\end{definition}

\begin{lemma}
\label{lemma3.3}{(See \cite[Lemma~2.4]{[BM14]})}
Assume that $(X_n)_{n \in \mathbb{N}}$ satisfies condition $[\mathbf{T}]$. Let $\mathbb{P}_n$ be the law of $X_n$ on $\ccal([0,T]; \mathbb{S}), n \in \mathbb{N}$. Then for every $\varepsilon > 0$ there exists a subset $A_\varepsilon \subset \ccal([0,T]; \mathbb{S})$ such that
\[\sup_{n \in \mathbb{N}} \mathbb{P}_n(A_\varepsilon) \ge 1 - \varepsilon\]
and
\begin{equation}
\label{eq:3.3}
\lim_{\delta \to 0} \sup_{u \in A_\varepsilon} m(u, \delta) = 0.
\end{equation}
\end{lemma}

Now we recall the Aldous condition which is connected with condition $[\mathbf{T}]$ (see \cite{[Metivier88]} and \cite{[Aldous78]}). This condition allows to investigate the modulus of continuity for the sequence of stochastic processes by means of stopped processes.

\begin{definition}
\label{defn3.4}
A sequence $(X_n)_{n \in \mathbb{N}}$ satisfies condition $[\mathbf{A}]$ iff $\forall\, \varepsilon > 0$, $\forall\, \eta > 0$, $\exists \, \delta > 0$ such that for every sequence $(\tau_n)_{n \in \mathbb{N}}$ of $\mathbb{F}$-stopping times with $\tau_n \le T$ one has
\[\sup_{n \in \mathbb{N}} \sup_{0 \le \theta \le \delta} \mathbb{P}\left\{\varrho(X_n(\tau_n + \theta), X_n(\tau_n)) \ge \eta \right\} \le \varepsilon.\]
\end{definition}

\begin{lemma} {(See \cite[Theorem~3.2]{[Metivier88]})}
\label{lemma3.5} Conditions $[\mathbf{A}]$ and $[\mathbf{T}]$ are equivalent.
\end{lemma}

Using the compactness criterion from Lemma~\ref{lemma6.3.3} and above results corresponding to Aldous condition we obtain the following corollary which we will use to prove the tightness of the laws defined by the truncated SPDE \eqref{eq:6.26}.

\begin{corollary}[Tightness criterion]
\label{cor3.6}
Let $(X_n)_{n \in \mathbb{N}}$ be a sequence of $\mathbb{F}$-adapted continuous  $\rm{H}$-valued processes such that
\begin{trivlist}
\item{(a)} there exists a constant $C_1 > 0$ such that
\[\sup_{n \in \mathbb{N}} \E \left[ \sup_{s \in [0,T]} \|X_n(s)\|^2_{\rm{V}} \right] \le C_1,\]

\item{(b)} there exists a constant $C_2 > 0$ such that
\[\sup_{n \in \mathbb{N}} \E \left[ \int_0^T |X_n(s)|^2_{\rm{D}(\A )}\,ds \right] \le C_2,\]

\item{(c)} $(X_n)_{n \in \mathbb{N}}$ satisfies the Aldous condition $[\mathbf{A}]$ in $\rm{H}$.
\end{trivlist}
Let ${\mathbb{P}}_n$ be the law of $X_n$ on $\mathcal{Z}_T$. Then for every $\varepsilon > 0$ there exists a compact subset $K_\varepsilon$ of $\mathcal{Z}_T$ such that
\[ \sup_{n \in \mathbb{N}} {\mathbb{P}}_n(K_\varepsilon) \ge 1 - \varepsilon\,.\]
\end{corollary}

\begin{proof}
Let $\varepsilon > 0$. By the  Chebyshev inequality and (a), we infer that for any $n \in \N $ and any $r>0$
\[
  {\mathbb{P}} \left(\left\{X_n \in \mathcal{Z}_T \colon \sup_{s \in [0,T]} \|X_n (s){\|}_{\rV}^{2} > r \right\}\right)
  \le \frac{ {\E} \bigl[ \sup_{s \in [0,T]} \|X_n (s) {\|}_{\rV}^{2} \bigr]}{r}
  \le \frac{{C}_{1}}{r}\,.
\]
Let ${R}_{1}$ be such that $\frac{{C}_{1}}{{R}_{1}} \le \frac{\varepsilon }{3}$. Then
\[
  \sup_{n\in \N } {\mathbb{P}}_n \left(\left\{ \sup_{s \in [0,T]}\|X_n (s){\|}_{\rV}^{2} > {R}_{1}\right\} \right) \le \frac{\varepsilon }{3}\,.
\]
Let ${B}_{1}:= \left\{ X_n \in \zcal_T :\, \, \sup_{s \in [0,T]}\|X_n(s) {\|}_{\rV}^{2} \le {R}_{1} \right\} $.

By the  Chebyshev inequality and (b), we infer that for any $n \in \N $ and any $r>0$
\[
{\mathbb{P}} \bigl(\left\{X_n \in \mathcal{Z}_T \colon \|X_n\|_{{L}^{2}(0,T; \rD(\A))} > r \right\}  \bigr)
  \le \frac{{\E} \bigl[ \|X_n\|_{{L}^{2}(0,T; \rD(\A))}^2 \bigr]  }{{r}^{2}}
  \le \frac{{C}_{2}}{{r}^{2}}\,.
\]
Let ${R}_{2}$ be such that $\frac{{C}_{2}}{{R}_{2}^{2}} \le \frac{\varepsilon }{3}$. Then
\[
  \sup_{n\in \N }  {\mathbb{P}}_n \bigl(\left\{ \|X_n\|_{{L}^{2}(0,T; \rD(\A))} > {R}_{2} \right\} \bigr) \le \frac{\varepsilon }{3}\,.
\]
Let ${B}_{2} := \left\{ X_n \in \zcal_T : \, \, \|X_n\|_{{L}^{2}(0,T; \rD(\A))} \le {R}_{2} \right\} $.

By Lemmas \ref{lemma3.3} and \ref{lemma3.5} there exists a subset
${A}_{\frac{\varepsilon}{3}} \subset \ccal([0,T], \rH)$ such that
${{\mathbb{P} }}_{n} \bigl( {A}_{\frac{\varepsilon }{3}}\bigr) \ge 1 - \frac{\varepsilon }{3}$ and
\[
   \lim_{\delta \to 0 }  \sup_{u \in {A}_{\frac{\varepsilon }{3}}}
\sup_{\underset{|t-s| \le \delta }{s,t \in [0,T]}}  |u(t) - u(s){|}_{\rH} = 0\,.
\]
It is sufficient to define ${K}_{\varepsilon } $ as the closure  of the set ${B}_{1} \cap {B}_{2} \cap {A}_{\frac{\varepsilon }{3}}$ in $\zcal_T$. By Lemma~\ref{lemma6.3.3}, ${K}_{\varepsilon }$ is compact in $\zcal_T$. The proof is thus complete.
\end{proof}

\subsection{The Skorohod theorem}
\label{s:3.2}

We will use the following Jakubowski's generalisation of the Skorohod theorem in the form given by Brze\'{z}niak and Ondrej\'{a}t \cite{[BO11]}, see also \cite{[Jakubowski97]}.

\begin{theorem}
\label{thm3.7}
Let $\mathcal{X}$ be a topological space such that there exists a sequence $\{f_m\}_{m \in \mathbb{N}}$ of continuous functions $f_m : \mathcal{X} \to \R$ that separates points of $\mathcal{X}$. Let us denote by $\mathcal{S}$ the $\sigma$-algebra generated by the maps $\{f_m\}$. Then
\begin{trivlist}
\item{(a)} every compact subset of $\mathcal{X}$ is metrizable,
\item{(b)} if $(\mu_m)_{m \in \mathbb{N}}$ is a tight sequence of probability measures on $(\mathcal{X}, \mathcal{S})$, then there exists a subsequence $(m_k)_{k \in \mathbb{N}}$, a probability space $(\Omega, \mathcal{F}, \mathbb{P})$ with $\mathcal{X}$-valued Borel measurable variables $\xi_k, \xi$ such that $\mu_{m_k}$ is the law of $\xi_k$ and $\xi_k$ converges to $\xi$ almost surely on $\Omega$. Moreover, the law of $\xi$ is a Radon measure.
\end{trivlist}
\end{theorem}

Using Theorem~\ref{thm3.7}, we obtain the following corollary which we will apply to construct a martingale solution of the stochastic tamed Navier-Stokes equations.

\begin{corollary}
\label{cor3.8}
Let $(\eta_n)_{n \in \mathbb{N}}$ be a sequence of $\mathcal{Z}_T$-valued random variables such that their laws $\emph{law}(\eta_n)$ on $(\mathcal{Z}_T, \mathcal{T})$ form a tight sequence of probability measures. Then there exists a subsequence $(n_k)$, a probability space $(\widetilde{\Omega}, \widetilde{\mathcal{F}}, \widetilde{\mathbb{P}})$ and $\mathcal{Z}_T$-valued random variables $\widetilde{\eta}$, $\widetilde{\eta}_k, k \in \mathbb{N}$ such that the variables $\eta_k$ and $\widetilde{\eta}_k$ have the same laws on $\mathcal{Z}_T$ and $\widetilde{\eta}_k$ converges to $\widetilde{\eta}$ almost surely on $\widetilde{\Omega}$.
\end{corollary}

\begin{proof}
It is sufficient to prove that on each space appearing in the definition \eqref{eq:3.1} of the space $\mathcal{Z}_T$, there exists a countable set of continuous real-valued functions separating points.

Since the spaces $\ccal ([0,T]; \mathrm{U}^\prime)$ and ${L}^{2}(0,T; \rH_{loc})$ are separable, metrizable and complete, this condition is satisfied, see \cite{[Badrikian70]}, expos\'{e} 8.

For the space ${L}^{2}_{w}(0,T; \rD(\A))$ it is sufficient to put
\[
    {f}_{m}(u):= \int_{0}^{T} \langle u(t), \v_m(t) \rangle_{\rD(\A)} \, dt \in \R ,
 \quad u \in {L}^{2}_{w}(0,T; \rD(\A)),\quad m \in \N ,
\]
where $\{ {\v}_{m}, m \in \N  \} $ is a dense subset of ${L}^{2}(0,T; \rD(\A))$.

Let us consider the space $\ccal ([0,T];{\rV}_{w})$. Let $\{ {h}_{m}, \, m \in \N  \}  $ be any dense subset of $\rV$ and let ${\mathbb{Q}}_{T}$ be the  set of rational numbers belonging to the interval $[0,T]$.
Then the family $\{ {f}_{m,t}, \, m \in \N , \, \, t \in {\mathbb{Q}}_{T} \} $ defined by
\[
      {f}_{m,t}(u):= \langle u(t), h_m \rangle_\rV \in \R ,
 \qquad u \in  \ccal ([0,T];{\rV}_{w}), \quad m \in \N ,
      \quad t \in {\mathbb{Q}}_{T}
\]
consists of continuous functions separating points in $\ccal ([0,T];{\rV}_{w})$. The statement of the corollary follows from Theorem~\ref{thm3.7}, concluding the proof.
\end{proof}

We end this section by giving the definitions of a martingale and strong solution to \eqref{eq:2.13}.

\begin{definition}
\label{defn_stoc_basis}
A stochastic basis $(\Omega, \mathcal{F}, \mathbb{F}, \mathbb{P})$ is a probability space equipped with the filtration $\mathbb{F} = \{\mathcal{F}_t\}_{t \ge 0}$ of its $\sigma-$field $\mathcal{F}$.
\end{definition}

\begin{definition}
\label{defn3.9}
A \emph{martingale solution} of \eqref{eq:2.13} is a system
\[\left(\widehat{\Omega}, \widehat{\mathcal{F}}, \widehat{\mathbb{F}}, \widehat{\mathbb{P}}, \widehat{W}, \widehat{u}\right) \]
where $\left(\widehat{\Omega}, \widehat{\mathcal{F}}, \widehat{\mathbb{P}}\right)$ is a probability space and $\widehat{\mathbb{F}} = \left(\widehat{\mathcal{F}}_t\right)_{t\ge0}$ is a filtration on it, such that
\begin{itemize}
\item $\widehat{W}$ is an $\ell^2$-valued cylindrical Wiener process on $\left(\widehat{\Omega}, \widehat{\mathcal{F}}, \widehat{\mathbb{F}}, \widehat{\mathbb{P}}\right)$, 
\item $\widehat{u}$ is $\rm{D}(\rm{A})$-valued progressively measurable process,$\rV$-valued weakly continuous $\widehat{\mathbb{F}}$-adapted process such that 
\[\widehat{u}(\cdot, \omega) \in \ccal([0,T]; \rV_w) \cap L^2(0,T; \rm{D}(\A )),\]
\[\widehat{\E}\left(\sup_{t\in[0,T]}\|\widehat{u}(t)\|^2_{\rV} + \int_0^T |\widehat{u}(t)|^2_{\rm{D}(\rm{A})}dt\right) < \infty\]
and
\begin{equation}
\label{eq:3.4}
\begin{split}
&\langle \widehat{u}(t), \v \rangle + \int_0^t \langle \A \widehat{u}(s), \v \rangle\,ds + \int_0^t \langle B(\widehat{u}(s)), \v \rangle\,ds \\
& \; + \int_0^t \langle g(|\widehat{u}(s)|^2)\,\widehat{u}(s), \v \rangle\,ds = \langle u_0, \v \rangle + \int_0^t \langle f(\widehat{u}(s)), \v \rangle\,ds \\
&\qquad + \left \langle \int_0^t G(s,\widehat{u}(s))\,dW(s) , \v  \right\rangle.
\end{split}
\end{equation}
for all $t \in [0,T]$ and all $\v \in \mathcal{V}$, $\widehat{\mathbb{P}}$-a.s.
\end{itemize}
\end{definition}

\begin{definition}
\label{defn3.10}
We say that problem \eqref{eq:2.13} has a \emph{strong solution} if for every stochastic basis $(\Omega, \mathcal{F}, \mathbb{F}, \mathbb{P})$ and $\ell^2$-valued cylindrical Wiener process $W(t)$ on the given filtered probability space there exists a $\rm{D}(\rm{A})$-valued progressively measurable process,$\rV$-valued continuous ${\mathbb{F}}$-adapted process $u$ such that
\[u(\cdot, \omega) \in \ccal([0,T]; \rV_w) \cap L^2(0,T; \rD(\A ))\,,\]
and satisfies \eqref{eq:3.4} for all $t \in [0,T]$ and all $\v \in \mathcal{V}$, $\mathbb{P}$-a.s.
\end{definition}

\begin{remark}
\label{rem_weak_strong}
The strong solution defined in Definition~\ref{defn3.10} is a probabilistically strong solution with the same regularity as a martingale solution. But, in Theorem~\ref{thm6.27} we show that the strong solution $u$ is more regular, i.e. $u \in \ccal([0,T]; \rV)\cap L^2(0,T; \rD(\A))$.
\end{remark}

\section{Truncated SPDEs}
\label{s:6}
The approximation scheme described in this section to define truncated SPDEs was first introduced by \cite{[FMRR14]} and also later used by Manna et al. in \cite{[MP16]}.

In order to describe the approximation scheme, we will use the following notations and spaces.
\[B_n := \left\{ \xi \in \R^3 : |\xi| \le n \right\} \subset \R^3 ,\;\; n \in \N \;,\]
where $|\cdot|$ is the Euclidean norm on $\R^3$. We will use $\mathcal{F}(u)$ and $\hat{u}$ interchangeably to denote the Fourier transform of $u$. The inverse Fourier transform will be given by $\invf$.

We define $\rH_n$ as the subspace of $\rH$,
\[\rH_n := \{u \in \rH : \rm{supp}(\hat{u}) \subset B_n\}\,.\]
The norm on $\rH_n$ is inherited from $\rH$. For $n \in \N$, let us define a map $P_n$ by
\begin{equation}
\label{eq:6.1}
P_n\,u := \invf(\ind_{B_n} \hat{u})\,.
\end{equation}
Firstly, note that $P_n : \rH \to \rH$ is a linear and bounded map. Moreover, $\mathrm{Range}(P_n) \subset H_n$ and hence one can deduce that
\[P_n : \rH \to \rH_n.\]
In addition, $P_n$ is an orthogonal projection onto $\rH_n$ i.e. $\forall\, u \in \rH$, $u - P_n\,u \perp \rH_n$. In other words 
\[\langle u - P_n u, \v\rangle_{H} = 0, \qquad \forall\, \v \in \rH_n.\]
Indeed, for $u \in \rH$ and $\v \in \rH_n$, we have
\begin{align*}
\langle P_n u, \v\rangle_{H} & = \langle P_n u, \v \rangle_{L^2} = \langle \invf\left(\ind_{B_n} \hat{u}\right), \invf\left(\hat{\v}\right)\rangle_{L^2} \\
& = \langle \ind_{B_n} \hat{u}, \hat{\v}\rangle_{L^2} = \int_{|\xi| \le n} \hat{u}(\xi)\cdot\hat{\v}(\xi)\,d\xi = \int_{\dom}\hat{u}(\xi)\cdot\hat{\v}(\xi)\,d\xi \\
& = \langle \hat{u}, \hat{\v}\rangle_{L^2} = \langle u, \v \rangle_{\rH}.
\end{align*}

Let us recall that $\rm{D}(\A ) := \rH \cap H^{2,2}$ and the Stokes operator is given by
\[\A \,u = - \Pi (\Delta\,u), \quad \quad u \in \rm{D}(\A )\,,\]
and $\rm{D}(\A )$ is a Hilbert space under the graph norm
\[|u|^2_{\rm{D}(\A )} := |u|^2_\rH + |\A \,u|^2_\rH\,.\]

\begin{lemma}
\label{lemma6.1}
Let $P_n$ be the orthogonal projection given by \eqref{eq:6.1}, then $P_n : \rV \to \rV$ is uniformly bounded.
\end{lemma}

\begin{proof}
Let $u \in \rV$, then by the definition of $P_n$ and $\rV$
\begin{align*}
& \|P_n u \|_\rV = \left[ \int_{\R^3} (1 + |\xi|^2) |\fcal \left(P_n u\right)(\xi)|^2\,d\xi \right]^{1/2} \\
&\quad = \left[ \int_{\R^3} (1+|\xi|^2) |\ind_{B_n}(\xi)\hat{u}(\xi)|^2\,d \xi \right]^{1/2}  = \left[ \int_{|\xi| \le n} (1 + |\xi|^2)|\hat{u}(\xi)|^2\, d \xi \right]^{1/2} \\
& \quad \le \left[ \int_{\R^3} (1 + |\xi|^2)|\hat{u}(\xi)|^2\, d \xi \right]^{1/2} = \|u\|_\rV\,.
\end{align*}
Thus we have shown that
\[\|P_n u\|_\rV \le \|u\|_\rV\,,\]
and hence $P_n$ is uniformly bounded in $\rV$.
\end{proof}

\begin{lemma}
\label{lemma6.2}
If $u \in \rm{D}(\A )$ then $\Delta \, u \in \rH$. In particular, if $u \in \rm{D}(\A )$ then $\A \,u = - \Delta\,u$.
\end{lemma}

\begin{proof}
Since $u \in \rm{D}(\A )$, it is clear that $\Delta \, u \in L^2$. Thus we are left to show that $\rm{div}(\Delta \, u) = 0$ in the weak sense. Let $\varphi \in C_0^\infty (\R^3)$, then using the definition of $\rm{div}$ and $\Delta$, we get
\begin{align*}
\langle \rm{div}(\Delta u)\, |\, \varphi \rangle & = - \langle \Delta\, u \, |\, \nabla \varphi \rangle = - \langle u \, | \, \Delta (\nabla \varphi) \rangle  = \langle \rm{div}\, u \, |\,  \Delta \, \varphi \rangle = 0\,.
\end{align*}
By definition $\A \,u = - \Pi(\Delta\,u)$, but since $\Delta\, u \in \rH$, and $\Pi : L^2 \to \rH$ is an orthogonal projection, $\Pi(\Delta\,u) = \Delta\,u$ and hence,
\begin{equation}
\label{eq:6.2}
\A \,u = - \Delta\,u, \quad \quad u \in \rm{D}(\A)\,.
\end{equation}
\end{proof}

\begin{lemma}
\label{lemma6.3}
If $n \in \N$, then $H_n \subset \rm{D}(\A )$ and
\begin{equation}
\label{eq:6.3}
P_n (\A  u) = \A  u, \quad \quad u \in \rH_n\,.
\end{equation}
\end{lemma}

\begin{proof}
We start with proving the first statement. Let $u \in \rH_n$; by definition
\[ \rm{D}(\A ) = \{u \in \rm{H} : u \in H^{2,2}\} = \left\{ u \in \rm{H} : \int_{\R^3} \left( 1 + |\xi|^2 \right)^2 |\hat{u}(\xi)|^2 d\xi < \infty \right\}. \]
Since $u \in \rH_n, \rm{supp}(\hat{u}) \subset B_n$,
\begin{align*}
& \int_{\R^3} \left( 1 + |\xi|^2 \right)^2 |\hat{u}(\xi)|^2 d\xi  = \int_{|\xi| \le n} \left(1 + |\xi|^2 \right)^2 |\hat{u}(\xi)|^2 d \xi \\
& \quad \le (1 + n^2)^2 \int_{|\xi| \le n} |\hat{u}(\xi)|^2 d \xi = (1+n^2)^2 \int_{\R^3} |\hat{u}(\xi)|^2 d\xi \\
& \quad  = (1+n^2)^2 \|u\|^2_{H_n} < \infty\,.
\end{align*}
Thus we have proved that $u \in \rm{D}(\A )$ and hence $\rH_n \subset \rm{D}(\A )$. Moreover, we showed that there exists a constant $C_n > 0$, depending on $n$ such that
\begin{equation}
\label{eq:6.4}
|u|_{\rm{D}(\A )} \le C_n \|u\|_{\rH_n}, \quad \quad u \in \rH_n\,.
\end{equation}
Now in order to establish the equality \eqref{eq:6.3}, we just need to show that $\A u \in \rm{H}_n$. Since $u \in \rm{H}_n$, $u \in \rm{D}(\A )$. Hence from Lemma~\ref{lemma6.1}, $\A \,u = - \Delta\,u$. We are left to show that $\rm{supp}\,(\fcal(\A u)) \subset B_n$. Using the definition of $\A \,u$, we get following equalities
\[\fcal (\A  u )(\xi) = - \fcal(\Delta\,u)(\xi) = - |\xi|^2 \hat{u}(\xi)\,.\]
Thus
\[ \mathrm{supp}(\fcal({\A u})) \subset \mathrm{supp}(|\cdot|^2) \cap \mathrm{supp}(\hat{u}) \subset B_n\,.\]
Hence $\A u \in \rm{H}_n$. Since $P_n : \rH \to \rH_n$ is an orthogonal projection, we infer that
\[P_n(\A u) = \A u\,.\]
\end{proof}

\begin{lemma}
\label{lemma6.4}
If $n \in \N$, then the map $\A _n := \A \big|_{H_n} : H_n \to H_n,$ is linear and bounded.
\end{lemma}

\begin{proof}
In Lemma~\ref{lemma6.2} we showed that $\A _n$ is well defined and it's straightforward to show it is linear. We are left to show that it is bounded. Let $u \in \rm{H}_n$, then by Parseval equality and the definition of $\rH_n$
\begin{align*}
\|\A _n u \|_{\rm{H}_n} & = |- \Delta\,u|_{L^2} = \left[ \int_{\R^3} |\xi|^2 |\hat{u}(\xi)|^2 \,d\xi \right]^{1/2}  \\
& = \left[ \int_{|\xi| \le n} |\xi|^2 |\hat{u}(\xi)|^2\,d \xi \right]^{1/2} \le \left[n^2 \int_{|\xi| \le n}|\hat{u}(\xi)|^2\,d \xi \right]^{1/2} \\
& = \left[ n^2 \int_{\R^3} |\hat{u}(\xi)|^2\,d \xi \right]^{1/2} = n \|u\|_{H_n}\,.
\end{align*}
Thus,
\begin{equation}
\label{eq:6.5}
\|\A _n u \|_{\rH_n} \le n \|u\|_{\rm{H}_n}\,.
\end{equation}
\end{proof}

\begin{lemma}
\label{lemma6.5}
If $n \in \N$, then the map
\begin{equation}
\label{eq:6.6}
B_n  : \rH_n \times \rH_n \ni (u,\v) \mapsto P_n(B(u,\v)) \in \rH_n
\end{equation}
is well defined and Lipschitz on a ball $\mathbb{B}_R := \left\{u \in \rH_n : \|u\|_{\rH_n} \le R \right\}$, $R > 0$. Moreover
\begin{align}
\label{eq:6.7}
\langle B_n(u), u \rangle_\rH & = 0, \quad \quad u \in \rH_n,\\
\label{eq:6.8}
\left|((B_n(u), u))\right| & \le \frac12 |u|^2_{\rm{D}(\A )} + \frac12 \big||u|\cdot|\nabla u|\big|^2_{L^2}, \quad u \in \rH_n,
\end{align}
where $B_n(u):= B_n(u,u)$ and $((\cdot, \cdot))$ is defined in \eqref{eq:2.1}.
\end{lemma}

\begin{proof}
We will show that $\forall\, u,\v \in \rH_n, B(u,\v) \in \rH$. Since $u, \v \in \rm{H}_n$, $u, \v \in \rm{D}(\A )$. Thus,
\[
|B(u,\v)|_{\rH} = |\Pi\left(u \cdot \nabla \v \right)|_{\rH} \le |u \cdot \nabla \v|_{L^2} \le \|u\|_{L^\infty} |\nabla \v|_{L^2}\,.\]
Since $H^{s,2}(\R^d) \hookrightarrow L^\infty(\R^d)$ for every $s > \frac{d}{2}$, there exists a constant $C > 0$ such that
\[\|u\|_{L^\infty(\R^3)} \le C \|u\|_{H^{s,2}(\R^3)}, \quad \mbox{for } s > \tfrac32.\]
In particular, it holds true for $s = 2$. Thus, we have
\[|B(u,\v)|_\rH \le C \|u\|_{H^2}\|\v\|_{H^1}\,.\]
Now by \eqref{eq:6.4} and \eqref{eq:6.15}
\begin{equation}
\label{eq:6.9}
|B(u,\v)|_\rH \le K_n \|u\|_{\rH_n} \|\v\|_{\rH_n} < \infty\,.
\end{equation}
Hence $B(u, \v) \in \rH$, which implies $B_n(u,\v) \in \rH_n$ and is well defined.\\
Let $R > 0$ be fixed and $u,\v \in \mathbb{B}_R$. Then, as before, using the embedding $H^2 \hookrightarrow L^\infty$, we have
\begin{align*}
&\|B_n(u) - B_n(\v)\|_{\rH_n} \le |B(u) - B(\v)|_\rH \le |u\cdot\nabla u - \v \cdot \nabla \v|_{L^2} \\
&\quad \le \|u-\v\|_{L^{\infty}}|\nabla u|_{L^2} + \|\v\|_{L^\infty}|\nabla(u-\v)|_{L^2} \\
&\quad \le \|u- \v\|_{H^2}\|u\|_{H^1} + \|\v\|_{H^2}\|u - \v\|_{H^1}\,.
\end{align*}
Since $u, \v \in \mathbb{B}_R$, and using \eqref{eq:6.4} and \eqref{eq:6.15}, we get
\begin{equation}
\label{eq:6.11}
\|B_n(u) - B_n(\v)\|_{\rH_n} \le C_{n,R}\|u-\v\|_{\rH_n}, \quad u, \v \in \mathbb{B}_R\,.
\end{equation}
Since $u \in \rH_n$ and $P_n$ is the orthogonal projection on $\rH$,
\begin{align*}
\langle B_n(u), u \rangle_\rH = \langle P_n(B(u,u)), u \rangle_\rH = \langle B(u,u), P_n u \rangle_\rH = \langle B(u,u) , u \rangle_\rH = 0\,.
\end{align*}
Also by using the definition of $((\cdot, \cdot))$ and the Cauchy-Schwarz inequality we get
\begin{align*}
\left|((B_n(u),u))\right| & = \left|\langle B_n(u), - \Delta\,u \rangle_\rH \right| = \left| \langle B(u,u), - P_n(\Delta\,u) \rangle_\rH \right|  \\
& = \left| \langle B(u,u), - \Delta\,u\rangle_\rH \right| \le |B(u,u)|_\rH\,|(- \Delta\,u)|_\rH \\
& \le \frac12 |u|^2_{\mathrm{D}(\A )} + \frac12 \big||u|\cdot|\nabla u|\big|^2_{L^2} \,.
\end{align*}
\end{proof}

\begin{lemma}
\label{lemma6.6}
The map
\begin{align}
\label{eq:6.12}
g_n  :\,\rH_n \ni u \mapsto P_n\left[ \Pi(g(|u|^2)\,u)\right] \in \rH_n\,,
\end{align}
is well defined and Lipschitz on a ball $\mathbb{B}_R$, $R > 0$. Moreover
\begin{equation}
\label{eq:6.13}
\begin{cases}
((-g_n(u),u)) \le C_N |\nabla u|^2_{L^2} - 2 \big||u|\cdot|\nabla u|\big|^2_{L^2}, \quad u \in \rH_n,\\
\langle -g_n(u), u \rangle_\rH \le - \|u\|^4_{L^4} + C_N|u|^2_\rH,  \quad \quad \quad \,\quad u \in \rH_n\,.
\end{cases}
\end{equation}
\end{lemma}

\begin{proof}
Let $u \in \rH_n$, then by the definition of $g$ \eqref{eq:g_new}, the estimate \eqref{eq:g_bounded} and the embedding of $H^1 \hookrightarrow L^6$, we have
\begin{align}
\label{eq:6.14}
& \|g_n(u)\|_{\rH_n}  = \left\|P_n\left[ \Pi(g(|u|^2)\,u)\right] \right\|_{\rH_n} \le |\Pi(g(|u|^2)\,u)|_\rH \le |g(|u|^2)\,u|_{L^2} \nonumber \\
&\; = \left[ \int_{\R^3} \left|g(|u(x)|^2)\right|^2\,|u(x)|^2\,dx \right]^{1/2} \le \left[\int_{\R^3} |u(x)|^6 \,dx \right]^{1/2} = \|u\|_{L^6}^3  \nonumber \\
&\; \le C\|u\|^3_{H^1} = C \left[\int_{\R^3}(1+|\xi|^2)|\hat{u}(\xi)|^2\,d\xi \right]^{3/2} \nonumber \\
&\; = C \left[ \int_{|\xi| \le n}(1+|\xi|^2)|\hat{u}(\xi)|^2\,d\xi \right]^{3/2} \nonumber \\
&\;  \le C(1+n^2)^{3/2} \left[ \int_{|\xi| \le n} |\hat{u}(\xi)|^2\,d\xi \right]^{3/2}  = C(1+n^2)^{3/2}\left[\int_{\R^3}|\hat{u}(\xi)|^2\,d\xi \right]^{3/2} \nonumber \\
&\; = C(1+n^2)^{3/2}|u|^3_{L^2} = C_n\|u\|^3_{\rH_n} < \infty\,.
\end{align}
Therefore $g_n : \rH_n \to \rH_n$ is well defined. From above we can also infer that there exists a constant $C_n > 0$ depending on $n$ such that
\begin{equation}
\label{eq:6.15}
\|u\|_{H^1} \le C_n \|u\|_{\rH_n}, \quad \quad u \in \rH_n\,.
\end{equation}
Let $R > 0$ be fixed and $u, \v \in \mathbb{B}_R$. Then, using \eqref{eq:g_lipschitz}, we have
\begin{align*}
\|& g_n(u) - g_n(\v)\|_{\rH_n} \le |\Pi(g(|u|^2)\,u) - \Pi(g(|\v|^2)\,\v)|_\rH \\
& \le |g(|u|^2)\,u - g(|\v|^2)\, \v|_{L^2} \\
& \le \left|\left(g(|u|^2) - g(|\v|^2) \right)\v\right|_{L^2} + |g(|u|^2)(u-\v)|_{L^2} \\
& \le  8 \left[\int_{\R^3}|u(x) - \v(x)|^2\left[|u(x)|^2+|\v(x)|^2\right]|\v(x)|^2\,dx\right]^{1/2}\\
& \quad + \left[\int_{\R^3}|u(x)|^4|u(x)-\v(x)|^2\,dx \right]^{1/2}\,.
\end{align*}
Since $H^1 \hookrightarrow L^6$, we obtain
\begin{align*}
&\| g_n(u) - g_n(\v)\|_{\rH_n}  \le  \left[\int_{\R^3}|u(x)|^6\,dx \right]^{1/3}\left[\int_{\R^3}|u(x)- \v(x)|^6\,dx \right]^{1/6} \\
&\quad  + 8 \left[\int_{\R^3}|u(x) - \v(x)|^6\,dx \right]^{1/6} \Bigl[\bigl[\int_{\R^3}|u(x)|^6\,dx \bigr]^{1/3}
\\
& \qquad \qquad +  \bigl[\int_{\R^3}|\v(x)|^6\,dx \bigr]^{1/3} \Bigr]^{1/2} \left[\int_{\R^3}|\v(x)|^6\,dx\right]^{1/6}\\
& \quad  =  \left[\|u\|^2_{L^6} \|u-\v\|_{L^6} + 8 \|u-\v\|_{L^6} \left(\|u\|^2_{L^6} + \|\v\|^2_{L^6}\right)^{1/2}\|\v\|_{L^6}\right] \\
& \quad \le C\|u-\v\|_{H^1} \left[\|u\|^2_{H^1} + 8 \left(\|u\|^2_{H^1} + \|\v\|^2_{H^1}\right)^{1/2}\|\v\|_{H^1}\right]\,.
\end{align*}
Since $u , \v \in \mathbb{B}_R$, using \eqref{eq:6.15}, we get
\begin{align}
\label{eq:6.16}
\|g_n(u) - g_n(\v)\|_{\rH_n} & \le \widehat{C}_n \|u - \v\|_{\rH_n} \left[\|u\|^2_{\rH_n} + 8 \left(\|u\|^2_{\rH_n} +\|\v\|^2_{\rH_n}\right)^{1/2}\|\v\|_{\rH_n}\right] \nonumber \\
& \le C_{n,R}\|u-\v\|_{\rH_n}\,.
\end{align}
Let $u \in \rH_n$, then using Lemmas~\ref{lemma6.2} and \ref{lemma6.3}, the definitions of $g_n$ and $((\cdot, \cdot))$ we get
\begin{align*}
((-g_n(u), u)) & = -\langle g_n(u), - \Delta\,u\rangle_\rH = - \langle \Pi(g(|u|^2)\,u, P_n(-\Delta\,u) \rangle_\rH  \\
& = - \langle g(|u|^2)\,u , \Pi(-\Delta\,u)\rangle_{L^2} = - \langle g(|u|^2)\,u, - \Delta\,u \rangle_{L^2}\,.
\end{align*}
Also, note that
\begin{align*}
&\langle - g_n(u),u \rangle_\rH  = - \langle \Pi(g(|u|^2))\,u, P_n\,u \rangle_\rH \\
&\; = -\langle g(|u|^2)\,u , \Pi(u) \rangle_{L^2} = -\langle g(|u|^2)\,u, u \rangle_{L^2}\,.
\end{align*}
Hence the inequalities \eqref{eq:6.13} can be established with the help of the above two relations and Lemma~\ref{lemma2.1} (ii). This completes the proof of the lemma.
\end{proof}

\begin{lemma}
\label{lemma6.7}
Let $f$ satisfy the assumption $({\bf H1})$.Then the map
\begin{align}
\label{eq:6.17}
f_n  : \,& \rH_n \ni u \mapsto P_n\left[ \Pi(f(u))\right] \in \rH_n
\end{align}
is well defined and Lipschitz.
\end{lemma}

\begin{proof}
Let $u \in \rH_n$, then by the assumption $({\bf H1})$,
\begin{align*}
\|f_n(u)\|_{\rH_n} & \le |\Pi(f(u))|_\rH \le |f(u)|_{L^2} \\
& \le C\left(C_f |u|_{L^2} + |b_f|^{1/2}_{L^1}\right) = C\left(C_f \|u\|_{\rH_n} + |b_f|^{1/2}_{L^1}\right) < \infty\,.
\end{align*}
Therefore $f_n : \, \rH_n \to \rH_n$ is well defined. Let $u, \v \in \rH_n$, then
\begin{align}
\label{eq:6.18}
\|f_n(u) - f_n(\v)\|_{\rH_n} & \le |\Pi f(u) - \Pi f(\v)|_\rH \le |f(u) - f(\v)|_{L^2} \nonumber \\
& \le C_f |u-\v|_{L^2} = C_f\|u-\v\|_{\rH_n}\,.
\end{align}
\end{proof}

\begin{lemma}
\label{lemma6.8}
Let $\sigma$ satisfy the assumption $({\bf H2})$. Then the map
\begin{align}
\label{eq:6.19}
G_n  : \, \rH_n \ni u \mapsto P_n \circ (G(u)) \in \mathcal{L}_2(\ell^2;\rH_n)
\end{align}
is well defined and Lipschitz.
\end{lemma}

\begin{proof}
Let $u \in \rH_n$, then
\begin{align*}
\|G_n(u)\|_{\mathcal{L}_2(\ell^2;\rH_n)} & \le \|(G(u))\|_{\mathcal{L}_2(\ell^2; \rH)} \le \left[ \int_{\R^3} \|\sigma(x)\|^2_{\ell^2}|\nabla u(x)|^2\,dx\right]^{1/2} \\
& \le \left[ \sup_{x \in \R^3} \|\sigma(x)\|^2_{\ell^2} \right]^{1/2} |\nabla u |_{L^2} \le \frac12 \|u\|_{H^1}\,.
\end{align*}
Using \eqref{eq:6.15}, we infer
\begin{equation}
\label{eq:6.20}
\|G_n(u)\|_{\mathcal{L}_2(\ell^2;\rH_n)} \le C_n \|u\|_{\rH_n} < \infty\,.
\end{equation}
Thus $G_n : \, \rH_n \to \mathcal{L}_2(\ell^2; \rH_n)$ is well defined. Let $u, \v \in \rH_n$, then
\begin{align*}
& \|G_n(u) - G_n(\v)\|_{\mathcal{L}_2(\ell^2; \rH_n)}  \le \|G(u) -  G(\v)\|_{\mathcal{L}_2(\ell^2; \rH)} \\
&\quad  \le \left[ \int_{\R^3} \sum_{j=1}^\infty |\sigma_j(x)|^2|\nabla (u - \v)(x)|^2\,dx \right]^{1/2} \\
& \quad = \left[\int_{\R^3} \|\sigma(x)\|^2_{\ell^2}|\nabla(u - \v)(x)|^2\,dx \right]^{1/2} \\
& \quad \le \left(\sup_{x \in \R^3} \|\sigma(x)\|^2_{\ell^2} \right)^{1/2} | \nabla (u -\v)|_{L^2} \le \frac12 \|u - \v\|_{H^1}\,.
\end{align*}
Using \eqref{eq:6.15}, we infer
\begin{equation}
\label{eq:6.21}
\|G_n(u) - G_n(\v)\|_{\mathcal{L}_2(\ell^2; \rH_n)} \le C_n \|u- \v\|_{\rH_n}\,.
\end{equation}
\end{proof}

\begin{proposition}
\label{prop6.9}
$L^2, H^1$ and $\rm{D}(\A )$ norms on $\rH_n$ are equivalent (with constants depending on $n$).
\end{proposition}

\begin{proof}
Let $u \in \rH_n$, then using the Parseval's identity
\begin{equation}
\label{eq:6.22}
|u|_{L^2} = \left[ \int_{\R^3}|\hat{u}(\xi)|^2 \, d \xi \right]^{1/2} = \left[\int_{|\xi| \le n}|\hat{u}(\xi)|^2\, d\xi \right]^{1/2} = \|u\|_{\rH_n}\,.
\end{equation}
Thus if $u \in \rH_n$ then $L^2$ and $\rH_n$ have equal norms. The equivalence of $H^1$ and $\rH_n$ norms is established from \eqref{eq:6.15}. Using \eqref{eq:6.5} and \eqref{eq:6.22} we can establish equivalence of $\rm{D}(\A )$ and $\rH_n$ norms.
\end{proof}

As discussed earlier in the introduction instead of using the standard Galerkin approximation of SPDE on the finite dimensional space we will look at truncated SPDEs on infinite dimensional space $\rH_n$. For every $n \in \N$, we will establish the existence of a unique global solution to the truncated SPDE and obtain a'priori estimates in order to prove the tightness of measures on a suitable space.

In order to study the truncated SPDE on $\rH_n$ we project the SPDE \eqref{eq:2.13} on $\rH_n$ using $P_n$. The projected SPDE on $\rH_n$ is given by
\begin{equation}
\label{eq:6.23}
\begin{split}
du_n(t) & = -  \left[\A_n  u_n(t) + B_n(u_n(t)) + g_n(u_n(t)) - f_n(u_n(t)) \right]\,dt \\
& \hspace{1.5truecm} + G_n(u_n(t))dW(t),\\
u_n(0) & = P_n (u_0),
\end{split}
\end{equation}
where
$u_n \in \rH_n, u_0 \in \rV$ and other operators $\A _n, B_n, g_n, f_n$ and $G_n$ are as defined in Lemmas~\ref{lemma6.4} -- \ref{lemma6.8}.

\begin{lemma}
\label{lemma6.10}
Let us define $F :\, \rH_n \to \R$ by
\begin{equation}
\label{eq:6.24}
F(u) := \|G_n(u)\|_{\mathcal{L}_2(\ell^2;\rH_n)} + 2 \langle u, - \A_n u - B_n(u) - g_n(u) + f_n(u) \rangle_\rH\,.
\end{equation}
Then for every $u \in \rH_n$ there exists $K_1 > 0$, independent of $n$, such that
\begin{equation}
\label{eq:6.25}
F(u) \le K_1(1+\|u\|^2_{\rH_n})\,.
\end{equation}
\end{lemma}

\begin{proof}
From the definition of $B_n, g_n$ and $f_n$, we have
\begin{align*}
&\|G_n(u)\|_{\mathcal{L}_2(\ell^2;\rH_n)}  + 2 \langle u, - \A_n u - B_n(u) - g_n(u) + f_n(u) \rangle_\rH \\
& \quad = \|G_n(u)\|_{\mathcal{L}_2(\ell^2;\rH_n)} + 2 \left\langle u, - \Pi (\Delta \,u)  - P_n (B(u))  \right \rangle_\rH \\
& \qquad - 2 \left\langle u, P_n[ \Pi(g(|u|^2)\,u)] - P_n[\Pi(f(u))]\right \rangle_\rH\,.
\end{align*}
Using Lemma~\ref{lemma6.8} and since $u \in \rH_n$, we get
\begin{align*}
F(u) & \le \frac14 \|u\|^2_{\rH_n} - 2|\nabla u|^2_{L^2} - 2\langle u, B(u) \rangle_\rH - 2\langle u, g(|u|^2)\,u \rangle_\rH + 2 \langle u, f(u) \rangle_\rH \\
& \le \frac34 \|u\|^2_{\rH_n} - 2|\nabla u|^2_{L^2} - 2\big|\sqrt{g(|u|^2)}|u|\big|^2_{L^2} + 2C_f \|u\|^2_{\rH_n} + 2|b_f|_{L^1},
\end{align*}
where in the last inequality we used the Young's inequality and hypothesis (\textbf{H1}). On rearranging, we get
\begin{align*}
 &F(u) + 2|\nabla u |^2_{L^2} + 2\big|\sqrt{g(|u|^2)}|u|\big|^2_{L^2}  \le \frac34 \|u\|^2_{\rH_n} + C_f \|u\|^2_{\rH_n} + 2|b_f|_{L^1}\\
& \quad  \le K_1 (1 + \|u\|^2_{\rH_n})\,,
\end{align*}
for appropriately chosen $K_1$. Thus, in particular
\[F(u) \le K_1 (1 + \|u\|^2_{\rH_n}).\]
\end{proof}

We will use the following theorem from \cite[Theorem~3.1]{[ABW10]} to prove Theorem~\ref{thm6.12}. We have modified it in the way we will be using it.

\begin{theorem}
\label{thm6.11}
Let $X$ be an abstract Hilbert space. Assume that $\sigma$ and $b$ satisfy the following conditions
\begin{itemize}
\item[(i)] For any $R > 0$ there exists a constant $C > 0$ such that
\[\|\sigma(u) - \sigma(\v)\|_{\mathcal{L}_2(\ell^2;X)} + \|b(u) - b(\v)\|_X \le C \|u -\v\|^2_X, \quad \|u\|_X,\|\v\|_X \le R\,.\]
\item[(ii)] There exists a constant $K_1 > 0$ such that
\[\|\sigma(u)\|^2_{\mathcal{L}_2(\ell^2; X)} + 2 \langle u, b(u) \rangle_{L^2} \le K_1(1+\|u\|^2_X), \quad \quad u \in X\,.\]
\end{itemize}
Then for any $X$-valued $\xi$, there exists a unique global solution $u = \left(u(t)\right)_{t \ge 0}$ to
\[u(t) = \xi + \int_0^t \sigma(u(s))\,dW(s) + \int_0^t b(u(s))\,ds\,.\]
\end{theorem}

\begin{theorem}
\label{thm6.12}
Let the assumptions $({\bf H1}) - ({\bf H2})$ hold. Then for every $u_0 \in \rV$ there exists a unique global solution $u_n = \left(u_n(t)\right)_{t \ge 0}$ to
\begin{equation}
\label{eq:6.26}
\begin{cases}
u_n(t) = \int_0^t \left[ - \A_n  u_n(s) - B_n(u_n(s)) - g_n(u_n(s)) + f_n(u_n(s))\right]ds \\
 \hspace{1.5truecm} + \int_0^t G_n(u_n(s))\,dW(s),\\
u_n(0) = P_n u_0\,.
\end{cases}
\end{equation}
\end{theorem}

\begin{proof}
The proof is a direct application of Theorem~\ref{thm6.11}. Using Lemmas~\ref{lemma6.4} - \ref{lemma6.8}, we can show that condition $(i)$ of Theorem~\ref{thm6.11} is satisfied. In Lemma~\ref{lemma6.10} we proved that condition $(ii)$ is satisfied. Thus we have existence of the unique global solution $u_n = \left(u_n(t) \right)_{t \ge 0}$ to \eqref{eq:6.26}.
\end{proof}

\section{Existence of solution}
\label{s:existence of sol}
\subsection{A'priori estimates}
\label{s:6.1}
In this subsection we will obtain certain a'priori estimates for the solution $u_n$ of \eqref{eq:6.26}. We will use these a'priori estimates in Lemma~\ref{lemma6.15} to prove the tightness of measures on the space $\mathcal{Z}_T$, defined in \eqref{eq:3.1}. We will also establish certain higher order estimates which will be required to prove the convergence of non-linear terms in later sections.

Let us fix $T > 0$. For any $R > 0$, define the stopping time
\begin{equation}
\label{eq:6.27}
\tau_R^n := \inf \{t \in [0,T] : \|u_n(t)\|_{\rV} \ge R\}\,,
\end{equation}
where $u_n$ is the solution of \eqref{eq:6.26}. By the definition of a martingale solution one knows that for every $n \ge 1$, $\tau_R^n \nearrow \infty$ as $R \nearrow \infty$.

\begin{lemma}
\label{lemma6.13}
Let $u_n$ be the solution of \eqref{eq:6.26}. For all $\rho >0$ there exist positive constants $C_1(\rho), C_2(\rho)$ such that if $\|u_0\|_\rV \le \rho$, then
\begin{align}
\label{eq:6.28}
\sup_{n \in \N}\E \left(\sup_{t \in [0,T]}\|u_n(t\wedge \tau_R^n)\|^2_\rV \right) \le C_1(\rho)\,, \\
\label{eq:6.29}
\sup_{n \in \N}\E \int_0^{T \wedge \tau_R^n} |u_n(t)|^2_{\rD(\A )}\,dt \le C_2(\rho)\,.
\end{align}
Moreover, for every $\delta > 0$ there exists a constant $C(\delta) > 0$ such that if $|u_0|_\rH \le \delta$, then
\begin{equation}
\label{eq:6.30}
\sup_{n \in \N}\E \int_0^{T \wedge \tau_R^n} \|\uns\|^4_{L^4}\,ds \le C_3(\delta)\,.
\end{equation}
\end{lemma}

\begin{proof}
Let $u_n(t)$, $t \ge 0$ be the solution of \eqref{eq:6.26} then applying the It\^o formula to $\phi(\xi) = |\xi|_{\rm{H}}^2$ and the process $u_n(t)$, we get
\begin{align}
\label{eq:6.31}
 |u_n(\st)|_{\rm{H}}^2  & = |P_n u_0|_{\rH}^2 + 2 \int_0^\st \langle  u_n(s), - \A_n  u_n(s) - B_n(u_n(s)) \rangle_{\rH}\,ds  \nonumber \\
&~ - 2\int_0^\st \langle u_n(s), g_n(u_n(s)) - f_n(u_n(s)) \rangle_{\rm{H}}\,ds \nonumber \\
&~ + \int_0^\st \|G_n(s,u_n(s))\|^2_{\mathcal{L}_2(\ell^2;\rH)}\,ds \nonumber \\
&~  + 2 \int_0^\st  \left\langle  u_n(s), G_n(s,u_n(s)) \, dW_s \right\rangle_{\rm{H}} \,.
\end{align}
Using Lemma~\ref{lemma2.1}, the Young's inequality, assumption $(\textbf{H1})$, boundedness of $P_n$ in $\rH$, we get
\begin{align}
\label{eq:6.32}
&|u_n(\st)|^2_\rH  \le |u_0|^2_\rH - 2 \int_0^\st |\nabla \uns |^2_{L^2}\,ds + C_f \int_0^\st |\uns|^2_\rH\,ds   \nonumber \\
&\; + \int_0^\st |\uns|^2_\rH\,ds - 2 \int_0^\st \langle\uns, g(|\uns|^2)\uns\rangle_\rH\,ds  \nonumber \\
&\; + \frac14 \int_0^\st |\nabla \uns|^2_{L^2}\,ds   + 2 \int_0^\st \langle u_n(s), G(s, u_n(s))\,dW_s \rangle_{\rm{H}} \,.
\end{align}
Using the \eqref{eq:6.13}$_2$, we get
\begin{align*}
& |u_n(\st)|^2_\rH   \le |u_0|^2_\rH - 2 \int_0^\st |\nabla \uns|^2_{L^2} \,ds - 2 \int_0^\st \|\uns\|^4_{L^4}\,ds   \nonumber \\
&\; +2 C_N \int_0^\st |\uns|^2_\rH\,ds  + \int_0^\st \left(C_f|\uns|^2_\rH + |b_f|_{L^1}\right)ds + \int_0^\st |\uns|^2_\rH\,ds  \nonumber\\
&\; + \frac14 \int_0^\st |\nabla \uns |^2_{L^2}\,ds + 2 \int_0^\st \langle u_n(s), G(s,u_n(s))\,dW_s \rangle_{\rm{H}} \,. \nonumber
\end{align*}
On rearranging we get
\begin{align}
\label{eq:6.33}
&|u_n(\st)|^2_\rH  + \frac74 \int_0^\st |\nabla \uns |^2_{L^2}\,ds + 2 \int_0^\st \|\uns\|^4_{L^4}\,ds \nonumber \\
 & \le |u_0|^2_\rH + T |b_f|_{L^1} + C_{f,N} \int_0^\st |\uns|^2_\rH\,ds  \nonumber \\
 & \quad + 2 \int_0^\st \langle u_n(s), G(s,u_n(s))\,dW_s \rangle_{\rm{H}}\,.
\end{align}
Now since
\[\mu_n(t) = \int_0^\st \langle u_n(s), G(s, u_n(s))\,dW_s \rangle_{\rm{H}}\,, \quad t \in [0,T]\]
is a $\mathbb{F}$-martingale, as by Lemma~\ref{lemma2.1} and \eqref{eq:6.27} we have the following inequalities
\begin{align*}
\E & \int_0^{t \wedge \tau_R^n} \big |\langle u_n(s), G(s,u_n(s))\rangle_{\rH} \big |^2\,ds \\
&\le \E \int_0^\st|u_n(s)|_{\rH}^2 \|G(s,u_n(s))\|^2_{\mathcal{L}_2(\ell^2; \rH)}\,ds \\
& \le \frac14 \E \int_0^\st |u_n(s)|^2_{\rm{H}}|\nabla u_n(s)|^2_{L^2}\,ds < \infty,
\end{align*}
where to establish the last inequality we have used the equivalences of norm from Proposition~\ref{prop6.9}. Thus, $\E [ \mu_n(t)] = 0$.

Hence applying Lemma~\ref{lemma2.2} for the following three processes$\colon$
\[
\begin{split}
&X(t) = |u_n(\st)|^2_\rH\,, \\
&Y(t) = \frac74 \int_0^\st |\nabla \uns |^2_{L^2}\,ds + 2 \int_0^\st \|\uns\|^4_{L^4}\,ds \\
&\mbox{and}\\
&I(t) = 2\mu_n(t) = 2\int_0^\st \langle u_n(s), G(s, u_n(s))\,dW_s \rangle_{\rm{H}}\,,
\end{split}
\]
we obtain from \eqref{eq:6.33}, \eqref{eq:2.18} is satisfied for $\alpha = 1$ , $Z = |u_0|^2_\rH + T|b_f|_{L^1}$ and $\phi(r) = C_{f,N}$. Since $\E(I(t)) = 0$, \eqref{eq:2.19} is satisfied and hence all inequalities for the parameters (see \eqref{eq:2.17}) are trivially satisfied. Thus, if $|u_0|_\rH \le \delta$, we have
\begin{align}
\label{eq:6.34}
\sup_{n \in \N} \E \Big[& |u_n(\st)|^2_\rH + \frac74 \int_0^\st |\nabla \uns|^2_{L^2}\,ds \nonumber \\
& \quad + 2 \int_0^\st \|\uns\|^4_{L^4}\,ds \Big] \le C_T(\delta)\,.
\end{align}
In particular,
\begin{equation}
\label{eq:6.35}
\sup_{n \in \N} \left(\sup_{t \in [0,T]}\E \, |u_n(t\wedge \tau_R^n)|^2_\rH \right) \le C_{T}(\delta)\,.
\end{equation}
Hence, using \eqref{eq:6.34} and \eqref{eq:6.35} we infer that
\begin{equation}
\label{eq:6.36}
\sup_{n \in \N}\E \int_0^{T \wedge \tau_R^n} \|\uns\|^4_{L^4}\,ds \le \widetilde{C}_{T} (|u_0|^2_\rH) =: C_3(\delta)\,.
\end{equation}

Since we are interested in the estimates involving $\rV$ norm of $u$. We apply the It\^o formula to $\phi(\xi) = |\nabla \xi|_{L^2}^2$ and the process $u_n(t)$, obtaining
\begin{align}
\label{eq:6.37}
& |\nabla u_n(\st)|_{L^2}^2 = |\nabla (P_n u_0)|_{L^2}^2 + 2 \int_0^\st \big(\big(  u_n(s), - \A_n  u_n(s) \big)\big) ds  \nonumber \\
&~ - 2 \int_0^\st \big(\big( u_n(s), B_n(u_n(s)) \big)\big)ds - 2\int_0^\st \big(\big(u_n(s), g_n(u_n(s)) \big)\big) ds  \nonumber \\
&~ + 2\int_0^\st \big(\big( u_n(s), f_n(u_n(s)) \big)\big) ds + \int_0^\st \|\nabla(G_n(s,u_n(s)))\|^2_{\mathcal{L}_2(\ell^2;\rH)}ds \nonumber \\
&~ + 2 \int_0^\st  \big(\big(  u_n(s), G_n(s,u_n(s)) \, dW_s \big)\big) \,,
\end{align}
where $((\cdot, \cdot))$ is as defined in \eqref{eq:2.1}.\\
Using Lemma~\ref{lemma2.1}, assumptions $(\textbf{H1}) - (\textbf{H2})$, boundedness of $P_n$ in $\rH$, estimates \eqref{eq:6.8}, \eqref{eq:6.13}, the Cauchy-Schwarz and the Young's inequality, we get
\begin{align*}
&|\nabla u_n(\st)|^2_{L^2}  \le |\nabla u_0|^2_{L^2} - 2 \int_0^\st |\A \,u_n(s)|^2_{L^2}\, ds + \int_0^\st |\A \,u_n(s)|^2_{L^2}    \nonumber \\
&\; + \int_0^\st \big | |u_n(s)| \cdot |\nabla u_n(s)| \big |^2_{L^2}\,ds  +  2 C_N \int_0^\st |\nabla u_n(s)|^2_{L^2}\,ds  \nonumber \\
&\; - 4 \int_0^\st \big||u_n(s)| \cdot |\nabla u_n(s)|\big|^2_{L^2}\,ds  + 2\int_0^\st |\A \,u_n(s)|_{L^2}|f(u_n(s))|_{L^2}\,ds \nonumber \\
&\; + \frac12 \int_0^\st |\A \,u_n(s)|^2_{L^2}\,ds + C_{T, \sigma} \int_0^\st |\nabla u_n(s)|^2_{L^2}\,ds  \nonumber \\
&\; + 2 \int_0^\st ((  u_n(s), G(s,u_n(s))\,dW_s )) \nonumber
\end{align*}
i.e.,
\begin{align*}
&|\nabla u_n(\st)|^2_{L^2}  \le |\nabla u_0|^2_{L^2} - \frac12 \int_0^\st |\A \,u_n(s)|^2_{L^2}\, ds   \nonumber \\
&\; - 3 \int_0^\st \big | |u_n(s)| \cdot |\nabla u_n(s)| \big |^2_{L^2}\,ds + C_{T,\sigma,N} \int_0^\st |\nabla u_n(s)|^2_{L^2}\,ds   \nonumber \\
&\; + \frac14 \int_0^\st |\A \,u_n(s)|^2_{L^2}\,ds + \int_0^\st \left(C_f|\uns|^2_\rH + |b_f|_{L^1}\right)ds \nonumber \\
&\; + 2 \int_0^\st (( u_n(s), G(s, u_n(s))\,dW_s ))\,.
\end{align*}
On rearranging, we have
\begin{align}
\label{eq:6.38}
&|\nabla u_n(\st)|^2_{L^2} + 3 \int_0^\st \big | |u_n(s)| \cdot |\nabla u_n(s)| \big |^2_{L^2}\,ds    \nonumber \\
& \quad+  \frac14 \int_0^\st |\A \,u_n(s)|^2_{L^2}\, ds \nonumber \\
&\; \le |\nabla u_0|^2_{L^2} + T|b_f|_{L^1} + C_{T,\sigma,N} \int_0^\st |\nabla u_n(s)|^2_{L^2}\,ds  \nonumber \\
& \quad  + C_{f} \int_0^\st |u_n(s)|^2_{\rH}\,ds + 2 \int_0^\st (( u_n(s), G(s, u_n(s))\,dW_s ))\,.
\end{align}
Now since the process
\[\mu_n(t) = \int_0^\st (( u_n(s), G(s, u_n(s))\,dW_s )) \,, \quad t \in [0,T]\]
is a $\mathbb{F}$-martingale, as by Lemma~\ref{lemma2.1} and \eqref{eq:6.27} we have the following inequalities
\begin{align*}
& \E \int_0^{t \wedge \tau_R^n} \big |(( u_n(s), G(s,u_n(s)))) \big |^2\,ds \\
&\; \le \E \int_0^\st|\A \,u_n(s)|_{L^2}^2 \|G(s,u_n(s))\|^2_{\mathcal{L}_2(\ell^2; \rH)}\,ds \\
&\; \le \frac14 \E \int_0^\st |\A \,u_n(s)|^2_{L^2}|\nabla u_n(s)|^2_{L^2}\,ds < \infty,
\end{align*}
where to establish the last inequality we have used the equivalences of norm from Proposition~\ref{prop6.9}. Thus, $\E [ \mu_n(t)] = 0$.

Again as before, by applying Lemma~\ref{lemma2.2} to processes
\[
\begin{split}
& X(t) = |\nabla u_n(\st)|^2_{L^2}\,,\\
& Y(t) = \frac14 \int_0^\st |\A \,u_n(s)|^2_{L^2}\, ds + 3 \int_0^\st \big | |u_n(s)| \cdot |\nabla u_n(s)| \big |^2_{L^2}\,ds \\
&\mbox{and} \\
&I(t) = 2 \mu_n(t) = 2 \int_0^\st (( u_n(s), G(s, u_n(s))\,dW_s )), \quad t \ge 0\,,
\end{split}\]
the inequalities \eqref{eq:2.18} and \eqref{eq:2.19} are satisfied. Thus, from \eqref{eq:6.35} and \eqref{eq:6.37}, if $\|u_0\|_\rV \le \rho$, then
\begin{align}
\label{eq:6.39}
\sup_{n \in \N} \mathbb{E} \Big[ |\nabla u_n(\st)|_{L^2}^2 + \frac14 \int_0^\st |\A \,u_n(s)|^2_{L^2}\, ds \nonumber \\
 \qquad \qquad + 3 \int_0^\st \big | |u_n(s)| \cdot |\nabla u_n(s)| \big |^2_{L^2}\,ds \Big]  \leq C_T(\rho)\,.
\end{align}
In particular,
\begin{equation}
\label{eq:6.40}
\sup_{n \in \N} \left(\sup_{t \in [0,T]} \E\,|\nabla u_n(\st)|_{L^2}^2 \right) \le C_{T}(\rho)\,.
\end{equation}
From \eqref{eq:6.40} and \eqref{eq:6.39}, we have the following estimate
\begin{equation}
\label{eq:6.41}
\sup_{n \in \N} \E\, \int_0^{T \wedge \tau_R^n} |\A \,u_n(t)|^2_{L^2}\,dt \le C_{T}(\rho)\,.
\end{equation}
Note that $|u|^2_{\rm{D}(\A )} := |u|^2_{L^2} + |\A \,u|^2_{L^2}$. Thus from \eqref{eq:6.35} and \eqref{eq:6.41} we can infer \eqref{eq:6.29}. On combining \eqref{eq:6.35} and \eqref{eq:6.40}, we get
\begin{equation}
\label{eq:6.42}
\sup_{n \in \N} \left(\sup_{t \in [0,T]}\E\, \|u_n(\st)\|^2_{\rV} \right) \le C_{T}(\rho)\,.
\end{equation}

Using the Burkholder-Davis-Gundy inequality, the definition of $\langle\cdot, \cdot\rangle_\rV$ and the Young's inequality, for every $ \varepsilon > 0$ we obtain
\begin{align}
\label{eq:6.43}
\E & \sup_{t \in [0,T]}  \int_0^\st \langle u_n(s), G(s, u_n(s))\,dW_s\rangle_{\rV} \nonumber \\
& \le \E \left[ \int_0^{T\wedge \tau_R^n} \big |\langle u_n(s), G(s,u_n(s))\rangle_{\rH} + \big(\big(u_n(s), G(s,u_n(s))\big)\big) \big |^2\,ds \right]^{1/2} \nonumber \\
& \le \E \left[ \int_0^{T\wedge \tau_R^n} \|G(s,u_n(s))\|^2_{\mathcal{L}_2(\ell^2;\rH)} |u_n(s)|^2_{\rm{D}(\A )}\,ds \right]^{1/2} \nonumber \\
& \le \E \left[\frac14 \int_0^{T\wedge \tau_R^n} |\nabla u_n(s)|^2_{L^2}|u_n(s)|^2_{\rm{D}(\A )} \,ds \right]^{1/2}\nonumber \\
& \le \frac12  \E \left[ \sup_{s \in [0,T]} |\nabla u_n(s \wedge \tau_R^n)|^2_{L^2} \int_0^{T\wedge \tau_R^n} |u_n(s)|^2_{\rm{D}(\A )}\,ds \right]^{1/2} \nonumber \\
& \le \varepsilon\, \E \sup_{s \in [0,T]} |\nabla u_n(s \wedge \tau_R^n)|^2_{L^2} + C_\varepsilon \E \int_0^{T\wedge \tau_R^n} |u_n(s)|^2_{\rm{D}(\A )}\,ds\,.
\end{align}
On combining \eqref{eq:6.33} and \eqref{eq:6.38}, then using \eqref{eq:6.29}, \eqref{eq:6.35}, \eqref{eq:6.40}, \eqref{eq:6.43} and Lemma~\ref{lemma2.2}, we can infer \eqref{eq:6.28}.
\end{proof}

In the next lemma we will use the estimates from Lemma~\ref{lemma6.13} to establish higher order estimates.

\begin{lemma}
\label{lemma6.14}
Let $\tau_R^n$ be as defined in \eqref{eq:6.27}. For all $\rho > 0$ and $p \in [1,3]$ there exist positive constants $C_1(p, \rho)$ , $C_2(p, \rho)$ such that if $\|u_0\|_\rV \le \rho$, then
\begin{align}
\label{eq:6.44}
&\sup_{n \in \N}\mathbb{E}\left(\sup_{t \in [0,T]} \|u_n(\st)\|_{\rm{V}}^{2p}\right)  \le C_1(p, \rho)\,,\\
\label{eq:6.45}
&\sup_{n \in \N} \E \int_0^{T \wedge \tau_R^n} \|\uns\|_{\rm{V}}^{2(p-1)} |\A \,u_n(s)|_{L^2}^2 ~ds \le C_2(p, \rho)\,.
\end{align}
\end{lemma}

\begin{proof}
Let $p \in [1, \infty)$. Then by using the It\^o formula for $\xi(t) = \|u_n(t)\|_{\rm{V}}^2$, $\phi(x) = x^p,$ equations \eqref{eq:6.31}, \eqref{eq:6.37} and the definition of $\|\cdot\|_\rV$, we obtain
\begin{align}
\label{eq:6.46}
 & \|u_n(\st)\|_{\rm{V}}^{2p} \nonumber \\
 &  = \|u_n(0)\|_{\rm{V}}^{2p} - 2 p \int_0^\st \|\uns\|_{\rm{V}}^{2(p-1)} \left(|\nabla \uns|^2_{L^2} + |\A \,u_n(s)|_{L^2}^2 \right)\,ds  \nonumber \\
&~ - 2p \int_0^\st \|\uns\|_{\rm{V}}^{2(p-1)} \langle u_n(s), B_n(u_n(s))\rangle_{\rm{V}} \,ds    \nonumber \\
&~ -2p \int_0^\st \|u_n(s)\|^{2(p-1)}_{\rV} \langle u_n(s) , g_n(u_n(s)) \rangle_{\rV}\,ds \nonumber \\
&~ -2p \int_0^\st \|u_n(s)\|^{2(p-1)}_{\rV} \langle u_n(s) , f_n(u_n(s)) \rangle_{\rV}\,ds \nonumber \\
&~  + p \int_0^\st \|u_n(s)\|^{2(p-1)}_{\rV}\|G_n(s,u_n(s))\|^2_{\mathcal{L}_2(\ell^2;\rV)}\,ds \nonumber \\
&~ + 2p(p-1) \int_0^\st \|\uns\|_{\rm{V}}^{2(p-2)} \langle \uns, G_n(s,\uns) \rangle_{\rm{V}}^2\,ds \nonumber \\
&~ + 2p  \int_0^\st \|\uns\|_{\rm{V}}^{2(p-1)} \langle \uns, G_n(s,\uns) dW_s\rangle_{\rm{V}}\,.
\end{align}
Using Lemma~\ref{lemma2.1}, boundedness of $P_n$ in $\rm{V}$ and assumption $(\textbf{H1})$, we can simplify \eqref{eq:6.46}
\begin{align*}
&\|u_n(\st)\|_{\rm{V}}^{2p} \nonumber \\
&\, \leq \|u_n(0)\|_{\rm{V}}^{2p} - 2 p \int_0^\st \|\uns\|_{\rm{V}}^{2(p-1)} \left(|\nabla \uns|^2_{L^2} + |\A \,u_n(s)|_{L^2}^2 \right) \,ds  \\
&~~~ + 2p \int_0^\st \|\uns\|_{\rm{V}}^{2(p-1)} \left[ \frac12 |\A  \uns|^2_{L^2} + \frac12 \big | |\uns|\cdot|\nabla \uns |\big|^2_{L^2} \right]\,ds \\
&~~~ - 2p \int_0^\st \|u_n(s)\|^{2(p-1)}_{\rV} \langle u_n(s), g(|\uns|^2) \uns \rangle_{\rV}\,ds \\
&~~~ + 2p \int_0^\st \|\uns\|^{2(p-1)}_\rV \left(C_f |\uns|^2_\rH + |b_f|_{L^1} + \frac14 |\A \,\uns|^2_{L^2} \right)\,ds \\
&~~~ + p \int_0^\st \|\uns\|^{2(p-1)}_{\rV} \left[\frac12 |\A \,\uns|^2_{L^2} + C_{T,\sigma} |\nabla \uns|^2_{L^2} \right]\,ds \\
&~~~ + 2p(p-1) \int_0^\st \|\uns\|^{2(p-2)}\|G(s,\uns)\|^2_{\mathcal{L}_2(\ell^2;\rV)}\|\uns\|^2_\rV \,ds \\
&~~~  + 2p  \int_0^\st \|\uns\|_{\rm{V}}^{2(p-1)} \langle \uns, G(s, \uns)\,dW_s\rangle_{\rm{V}}\,.
\end{align*}
On rearranging, we get
\begin{align*}
& \|u_n(\st)\|_{\rm{V}}^{2p} \nonumber \\
&\, \leq \|u_n(0)\|_{\rm{V}}^{2p} - \frac{p}{2} \int_0^\st \|\uns\|_{\rm{V}}^{2(p-1)} |\A \,\uns|_{L^2}^2 \,ds  \\
&~~~ -2p \int_0^\st \|\uns\|^{2(p-1)}|\nabla \uns|^2_{L^2}\,ds \nonumber\\
&~~~ + p \int_0^\st \|\uns\|_{\rm{V}}^{2(p-1)} \big| |\uns|\cdot|\nabla \uns |\big |^2_{L^2}\,ds  \\
&~~~ - 2p \int_0^\st \|\uns\|_\rV^{2(p-1)} \big||\sqrt{g(|\uns|^2)}|\cdot|\uns| \big|^2_{L^2}\,ds \\
&~~~ + C_N \cdot 2p \int_0^\st \|u_n(s)\|^{2(p-1)}_{\rV}|\nabla \uns|^2_{L^2}\,ds   \\
&~~~ - 4p \int_0^\st \|\uns\|_{\rm{V}}^{2(p-1)} \big | |\uns|\cdot|\nabla \uns |\big |^2_{L^2}\,ds \\
&~~~ + 2p\,C_f\int_0^\st \|\uns\|^{2(p-1)}_\rV|\uns|^2_\rH\,ds + 2p|b_f|_{L^1}\int_0^\st \|\uns\|^{2(p-1)}_\rV\,ds \\
&~~~ + p\,C_{T,\sigma} \int_0^\st \|\uns\|^{2(p-1)}_{\rV}|\nabla \uns|^2_{L^2}\,ds \\
&~~~ + 2p  \int_0^\st \|\uns\|_{\rm{V}}^{2(p-1)} \langle \uns, G(s,\uns)\,dW_s\rangle_{\rm{V}} \\
&~~~ + 2p(p-1) \int_0^\st \|\uns\|^{2(p-1)}\left[\frac14 |\A \,\uns|^2_{L^2} + C_{T, \sigma} |\nabla \uns|^2_{L^2} \right] \,ds
\end{align*}
which on further simplification yields
\begin{align}
\label{eq:6.47}
& \|u_n(\st)\|_{\rm{V}}^{2p}  + \frac{p(3-p)}{2} \int_0^\st \|\uns\|_{\rm{V}}^{2(p-1)} |\A \,\uns|_{L^2}^2 \,ds  \nonumber \\
&~ + 3 p \int_0^\st \|\uns\|_{\rm{V}}^{2(p-1)} \big | |\uns|\cdot|\nabla \uns |\big |^2_{L^2}\,ds \nonumber \\
&~ + 2p \int_0^\st \|\uns\|_\rV^{2(p-1)} \big||\sqrt{g(|\uns|^2)}|\cdot|\uns| \big|^2_{L^2}\,ds \nonumber \\
&~~ \leq \|u_n(0)\|_{\rm{V}}^{2p} + C_{T,\sigma,N,p} \int_0^\st \|u_n(s)\|^{2(p-1)}_{\rV}|\nabla \uns|^2_{L^2}\,ds  \nonumber\\
&\quad + C_{f} \int_0^\st \|u_n(s)\|^{2(p-1)}_{\rV}|\uns|^2_{\rH}\,ds + 2p|b_f|_{L^1}\int_0^\st \|\uns\|^{2(p-1)}_\rV\,ds \nonumber \\
&\quad + 2p  \int_0^\st \|\uns\|_{\rm{V}}^{2(p-1)} \langle \uns, G(s,\uns)\,dW_s\rangle_{\rm{V}}\,.
\end{align}
As before we will show that
\[\mu_n(\st) = \int_0^\st \|\uns\|^{2(p-1)}_\rV \langle  u_n(s), G(s,u_n(s)) \, dW_s\rangle_{\rm{V}}\\, \quad t \in [0,T]\]
is a $\mathbb{F}$-martingale. By Lemma~\ref{lemma2.1} and \eqref{eq:6.27} we have the following inequalities
\begin{align*}
\E\int_0^\st & \|\uns\|_\rV^{4(p-1)}\big|\langle \uns, G(s, \uns) \rangle_\rV\big|^2\,ds \\
&  \le \E\int_0^\st\|\uns\|_\rV^{4(p-1)}|\uns|^2_{\rm{D}(\A )}\|G(s,\uns)\|^2_{\mathcal{L}_2(\ell^2;\rH)}\,ds \\
& \le \frac14 \E \int_0^\st\|\uns\|_\rV^{4(p-1)}|\uns|^2_{\rm{D}(\A )}|\nabla \uns|^2_{L^2}\,ds < \infty\,,
\end{align*}
where the finiteness of the integral follows from Proposition~\ref{prop6.9}. Hence, $\E [ \mu_n(t)] = 0$.

Since $|\uns|_\rH \le \|\uns\|_\rV$ and $|\nabla \uns|_{L^2} \le \|\uns\|_\rV$ on applying the modified version of the Gronwall Lemma (Lemma~\ref{lemma2.2}) for
\[
\begin{split}
&X(t) = \|u_n(\st)\|_{\rm{V}}^{2p}\,, \\
&I(t) = 2p  \int_0^\st \|\uns\|_{\rm{V}}^{2(p-1)} \langle \uns, G(s,\uns)\,dW_s\rangle_{\rm{V}}
\end{split}\]
and
\begin{align*}
Y(t) & = \frac{p(3-p)}{2} \int_0^\st \|\uns\|_{\rm{V}}^{2(p-1)} |\A \,\uns|_{L^2}^2 \,ds \\
&+ 3 p \int_0^\st \|\uns\|_{\rm{V}}^{2(p-1)} \big | |\uns|\cdot|\nabla \uns |\big |^2_{L^2}\,ds \\
&~~~ + 2p \int_0^\st \|\uns\|_\rV^{2(p-1)} \big||\sqrt{g(|u(s)|^2)}|\cdot|u(s)| \big|^2_{L^2}\,ds,
\end{align*}
we have
\begin{equation}
\label{eq:6.48}
\sup_{n \in \N}\left(\sup_{t \in [0,T ]} \E\, \|u_n(\st)\|^{2p}_\rV\right) \le C_{T,p}\|u_0\|^{2p}_{\rV}, \quad p \in [1,3]\,.
\end{equation}
Using \eqref{eq:6.48} in \eqref{eq:6.46}, we also obtain
\begin{align}
\label{eq:6.49}
& \sup_{n \in \N} \E \int_0^\st \|\uns\|^{2(p-1)}_\rV |\A \,\uns|^2_{L^2}\,ds \nonumber \\
& \quad \le C_{T,p} \|u_0\|^{2p}_\rV := C_2(p, \rho), \quad p \in [1,3]\,.
\end{align}

Now we are left to show the estimate \eqref{eq:6.44}. Using the Burkholder- Davis- Gundy inequality, Lemma~\ref{lemma2.1}, we get
\begin{align*}
\E \sup_{t \in [0,T]} & \int_0^\st \|\uns\|^{2(p-1)}_\rV \langle u_n(s), G(s,u_n(s))dW_s\rangle_{\rV} \nonumber \\
& \le \E \left[ \int_0^{T \wedge \tau_R^n} \|\uns\|^{4(p-1)}_\rV \big |\langle u_n(s), G(s,u_n(s)) \rangle_{\rV} \big |^2\,ds \right]^{1/2} \nonumber \\
& \le \E \left[ \int_0^{T \wedge \tau_R^n} \|\uns\|^{4(p-1)}_\rV \|G(s,u_n(s))\|^2_{\mathcal{L}_2(\ell^2;\rH)} |u_n(s)|^2_{\rm{D}(\A )}\,ds \right]^{1/2} \nonumber \\
& \le \frac14 \E \left[ \int_0^{T \wedge \tau_R^n} \|\uns\|^{4(p-1)}_\rV |\nabla u_n(s)|^2_{L^2}|u_n(s)|^2_{\rm{D}(\A )} \,ds \right]^{1/2}.\nonumber
\end{align*}
Using the definition of $\rV$-norm, the H\"older inequality and the Young's inequality, for $\varepsilon > 0$ we obtain
\begin{align}
\label{eq:6.50}
& \E \sup_{t \in [0,T]}  \int_0^\st \|\uns\|^{2(p-1)}_\rV \langle u_n(s), G(s,u_n(s))dW_s\rangle_{\rV} \nonumber \\
&\; \le \frac14 \E \left[ \int_0^{T \wedge \tau_R^n} \|\uns\|^{2p}_\rV \|u_n(s)\|^{2(p-1)}_{\rV}|u_n(s)|^2_{\rm{D}(\A )} \,ds \right]^{1/2}\nonumber \\
&\; \le \frac14  \E \left[ \sup_{t \in [0,T]} \|u_n(\st)\|^{2p}_{\rV} \int_0^{T \wedge \tau_R^n} \|\uns\|^{2(p-1)}_\rV |u_n(s)|^2_{\rm{D}(\A )}\,ds \right]^{1/2} \nonumber \\
&\; \le \varepsilon\, \E \sup_{t \in [0,{T }]} \|u_n(\st)\|^{2p}_{\rV} + C_\varepsilon \E \int_0^{T \wedge \tau_R^n} \|\uns\|^{2(p-1)} |u_n(s)|^2_{\rm{D}(\A )}\,ds\,.
\end{align}
Thus from \eqref{eq:6.47} and using \eqref{eq:6.49}, \eqref{eq:6.50} and Lemma~\ref{lemma2.2}, we have
\begin{align*}
&\E \left(\sup_{t \in [0,T]} \|u_n(\st)\|^{2p}_{\rV} \right) \\
&\quad  \le C_{T,p} \left(\E \,\|u_0\|^{2p}_\rV \right) + \varepsilon\, \E \left(\sup_{t \in [0,T]} \|u_n(\st)\|^{2p}_\rV \right) + C_{T,p, \varepsilon}\,.
\end{align*}
Choosing $\varepsilon$ small enough we get
\begin{align}
\label{eq:6.51}
\sup_{n \in \N}\E \left(\sup_{t \in [0,T]} \|u_n(\st)\|^{2p}_\rV \right) \le C_{T,p}\left(\|u_0\|^{2p}_{\rV}\right) := C_1(p, \rho), \quad p \in [1,3]\,.
\end{align}
\end{proof}

\subsection{Tightness of measures}
For each $n \in \mathbb{N}$, the solution $u_n$ of the truncated equation \eqref{eq:6.26} defines a measure $\emph{law}(u_n)$ on $(\mathcal{Z}_{T}, \mathcal{T})$, where $\mathcal{Z}_T$ was defined in \eqref{eq:3.1} as
\[\mathcal{Z}_T = \ccal([0,T]; \mathrm{U}^\prime) \cap L^2_w(0,T; \mathrm{D}(\A)) \cap L^2(0,T; \rH_{loc}) \cap \ccal([0,T]; \mathrm{V}_w).\]
In this subsection we will prove that this sequence of measures defined on $\mathcal{Z}_{T}$ is tight.

\begin{lemma}
\label{lemma6.15}
The set of measures $\{\emph{law}(u_n), n \in \N\}$ is tight on $(\mathcal{Z}_{T}, \mathcal{T})$.
\end{lemma}

\begin{proof}
We recall the definition of the stopping time, $\tau_R^n$
\[\tau_R^n := \inf \{ t \in [0,T] : \|u_n(t)\|_\rV  \ge R\}.\]
We will use Corollary~\ref{cor3.6} to prove the tightness of measures. According to estimates \eqref{eq:6.28} and \eqref{eq:6.29}, conditions $(a)$ and $(b)$ are satisfied. Thus it is sufficient to prove that the sequence $(u_n)_{n \in \mathbb{N}}$ satisfies the Aldous condition $[\textbf{A}]$ in $\rH$. By \eqref{eq:6.26}, for $t \in [0,T \wedge \tau_R^n]$ we have
\begin{align*}
\un &= u_n(0) - \int_0^t \A _n \uns\,ds - \int_0^t B_n (\uns)\,ds - \int_0^t g_n(\uns)\,ds \\
&~~~ + \int_0^t f_n(\uns)\,ds  + \int_0^t G_n(s, \uns)\,dW(s)\\
&:= J_1^n + J_2^n(t) + J_3^n(t) +J_4^n(t) + J_5^n(t) + J_6^n(t),\quad t \in [0,T \wedge \tau_R^n]\,.
\end{align*}
Let $s , t \in [0,T], s< t$ and $\theta:= t - s$. First we will establish estimates for each term of the above equality.
\noindent \textbf{Ad.} $J_2^n$. Since $\A  : \rV \to \rV^\prime$, then by the H\"older inequality and \eqref{eq:6.29}, we have the following inequalities
\begin{align}
\label{eq:6.52}
& \Eb \left[ |J_2^n(\st) - J_2^n(s \wedge \tau_R^n)|_{\rH} \right] = \Eb \left| \int_{s \wedge \tau_R^n}^{\st} \A _n \unr\,dr \right|_{\rH}
 \nonumber \\
& \qquad  \leq c\Eb \int_{s \wedge \tau_R^n}^{\st} |\A \unr|_{\rH}\,dr \leq c\theta^{\frac12} \left(\E \int_{s \wedge \tau_R^n}^{\st} |\unr|^2_{\rD(\A)}\,dr \right)^{\frac12} \nonumber \\
&\qquad \leq c (C_2(R))^\frac12 \cdot \theta^{1/2} := c_2 \cdot \theta^{1/2}\,.
\end{align}
\noindent \textbf{Ad.} $J^n_3$. $B: \rD(\A) \times \rV \rightarrow \rH$ is bilinear and continuous and $P_n : \rH \to \rH$ is bounded then by Lemma~\ref{lemma6.5}, the Cauchy-Schwarz inequality and \eqref{eq:6.28} we have
\begin{align}
\label{eq:6.53}
& \Eb  \left[ | J^n_3(\st)  - J^n_3(s \wedge \tau_R^n)|_{\rH} \right]  = \Eb \left|\int_{s \wedge \tau_R^n}^{\st} P_n B(\unr)\,dr \right|_{\rH} \nonumber \\
 & \; \leq  \Eb \int_{s \wedge \tau_R^n}^{\st} |P_n B(\unr, \unr)|_{\rH}\,dr \nonumber \\
& \; \leq  \Eb \int_{s \wedge \tau_R^n}^{\st} \|B\|\cdot |\unr|_{\rm{D}(\A )}\|\unr\|_\rV \, dr
 \nonumber \\
& \; \leq  C \Eb \left( \left[ \sup_{t \in [0,T ]} \|u_n(\st)\|^2_{\rm{V}} \right]^{1/2} \cdot \theta^{1/2} \left[\int_0^\st|\unr|^2_{\rm{D}(\A )}\,dr \right]^{1/2} \right) \nonumber \\
& \; \leq  (C_1(R))^{1/2} (C_2(R))^{1/2} \cdot \theta^{1/2} := c_3 \cdot \theta^{1/2}\,.
\end{align}
\noindent \textbf{Ad.} $J^n_4$. Since $H^1 \hookrightarrow L^6$ then by the definition of $g$ and estimate \eqref{eq:6.44} (for $p = 2$), we have
\begin{align}
\label{eq:6.54}
\Eb & \left[|J^n_4(\st) - J^n_4(s \wedge \tau_R^n)|_{\rH} \right]  = \Eb \left| \int_{s \wedge \tau_R^n}^{\st} g_n(\unr)\,dr \right|_{\rH}  \nonumber \\
&\leq \Eb \int_{s \wedge \tau_R^n}^{\st} |P_n(\Pi g(|\unr|^2)\unr)|_{\rH}\,dr\leq  \Eb \int_{s \wedge \tau_R^n}^{\st} |g(|\unr|^2)\unr|_{L^2}\,dr \nn \\
& \leq  \Eb \int_{s \wedge \tau_R^n}^{\st} \left(\int_{\R^3}|u_n(r,x)|^6\,dx\right)^{1/2}\,dr = \Eb  \int_{s \wedge \tau_R^n}^{\st} \|\unr\|_{L^6}^3 ds \nonumber
\\
& \leq C \, \E \int_{s \wedge \tau_R^n}^{\st} \|\unr\|^3_\rV\,dr \le C T^{1/2} \left[\E \left(\sup_{t \in [0,T]} \|u_n(\st)\|^4_{\rV}\right)\right]^{3/4} \theta^{1/2} \nn \\
& \leq C \cdot (C_1(2,R))^{3/4} \cdot \theta^{1/2} := c_4 \cdot \theta^{1/2}\,.
\end{align}
\noindent \textbf{Ad.} $J^n_5$. Using the assumption $\mathbf{H1}$, \eqref{eq:6.35} and the Cauchy-Schwarz inequality, we obtain the following inequalities
\begin{align}
\label{eq:6.55}
\Eb & \left[ | J^n_5(\st) - J^n_5(s \wedge \tau_R^n)|_{\rH} \right]  = \Eb \left|  \int_{s \wedge \tau_R^n}^{\st} P_n (\Pi f(\unr)\,dr) \right|_{\rH} \nn \\
& \leq \Eb \int_{s \wedge \tau_R^n}^{\st} |f(\unr)|_{\rH}\,dr \leq  \Eb \left(\int_{s \wedge \tau_R^n}^{\st} |f(\unr)|^2_{\rH}\,dr \right)^{1/2} \cdot \theta^{1/2}  \nn \\
& \leq  \, \left(C_f \int_{s \wedge \tau_R^n}^{\st} |\unr|^2_\rH \,dr + T|b_f|_{L^1} \right)^{1/2} \theta^{\frac12} \nn\\
& \leq\,  \left(C_f+|b_f|_{L^1}\right)^{1/2} T^{1/2} \left(C(R)\right)^{1/2} \theta^{1/2} := c_5 \cdot \theta^{1/2}\,.
\end{align}
\noindent \textbf{Ad.} $J^n_6$. Using the It\^o isometry, Lemma~\ref{lemma2.1} and \eqref{eq:6.28}, we obtain the following
\begin{align}
\label{eq:6.56}
\Eb & \left[ | J^n_6(\st) - J^n_6(s \wedge \tau_R^n)|^2_{\rH} \right] = \Eb \left| \int_{s \wedge \tau_R^n}^{\st} G_n(r, \unr)\,dW(r) \right|^2_{\rH}   \nonumber \\
& = \Eb \int_{s \wedge \tau_R^n}^{\st} \|P_n G(r, \unr)\|^2_{\mathcal{L}_2(\ell^2;\rH)}\,dr \leq \frac{1}{4} \E \int_{s \wedge \tau_R^n}^{\st} |\nabla \unr|^2_{L^2}\,dr  \nonumber \\
& \leq \frac{1}{4} \Eb \left(\sup_{t \in [0,T]}\|u_n(\st)\|_\rV^2 \right) \theta  \leq \frac{1}{4} C_1(R) \cdot \theta := c_6 \cdot \theta\,.
\end{align}

\noindent Let us fix $\kappa > 0$ and $\varepsilon > 0$. By the Chebyshev's inequality and estimates \eqref{eq:6.52} - \eqref{eq:6.55}, we obtain
\begin{align*}
& \mathbb{P}( \{|J_i^n(\st) - J^n_i(s \wedge \tau_R^n)|_{\rH} \geq \kappa \}) \\
& \quad  \leq \dfrac{1}{\kappa} \mathbb{E} \left[|J_i^n(\st) - J^n_i(s \wedge \tau_R^n)|_{\rH} \right] \leq \dfrac{c_i \theta^{1/2}}{\kappa}; \quad n \in \mathbb{N},
\end{align*}
where $i = 2, \dots, 5$. Let $\delta_i = \dfrac{\kappa^2}{c_i^2} \varepsilon^2$. Then
\[ \sup_{n \in \mathbb{N}} \sup_{0 \leq \theta \leq \delta_i} \mathbb{P}( \{|J_i^n(\st) - J^n_i(s \wedge \tau_R^n)|_{\rH} \geq \kappa \}) \leq \varepsilon,~~~ i = 2 \dots 5\,.\]
By the Chebyshev inequality and \eqref{eq:6.56}, we have
\begin{align*}
&\mathbb{P}( \{|J_6^n(\st) - J^n_6(s \wedge \tau_R^n)|_{\rH} \geq \kappa \}) \\
& \quad \leq \dfrac{1}{\kappa^2} \mathbb{E} \left[|J_6^n(\st) - J^n_6(s \wedge \tau_R^n)|_{\rH}^2 \right] \leq \dfrac{c_6 \theta}{\kappa^2}; \quad n \in \mathbb{N}
\end{align*}
Let $\delta_6 = \dfrac{\kappa^2}{c_6} \varepsilon$. Then
\[ \sup_{n \in \mathbb{N}} \sup_{0 \leq \theta \leq \delta_6} \mathbb{P}( \{|J_6^n(\st) - J^n_6(s \wedge \tau_R^n)|_{\rH} \geq \kappa \}) \leq \varepsilon\,.\]
Since $[\textbf{A}]$ holds for each term $J_i^n,~i= 1,2, \dots, 6$; we infer that it holds also for $(u_n)_{n \in \N}$. Thus, the proof of lemma can be concluded by invoking Corollary~\ref{cor3.6}.
\end{proof}

Now we will state the main theorem of this section.

\begin{theorem}
\label{thm6.16}
Let assumptions $(\mathbf{H1})$ and $(\mathbf{H2})$ be satisfied. Then there exists a martingale solution $(\widehat{\Omega}, \widehat{\mathcal{F}}, \widehat{\mathbb{F}}, \widehat{\mathbb{P}}, \widehat{u})$ of problem \eqref{eq:2.13} such that
\begin{equation}
\label{eq:6.57}
\hat{\E} \left[\sup_{t \in [0,T]}\|\hu(t)\|^2_{\rm{V}} + \int_0^T |\hu(t)|^2_{\rm{D}(\A )}\,dt \right] < \infty\,.
\end{equation}
\end{theorem}

In the following subsection we will prove Theorem~\ref{thm6.16} in several steps.

\subsection{Proof of Theorem~\ref{thm6.16}}
\label{s:6.3}

By Lemma~\ref{lemma6.15} the set of measures $\{\emph{law}(u_n), n \in \N \}$ is tight on the space $(\mathcal{Z}_{T}, \mathcal{T})$ defined by \eqref{eq:3.1}. Hence by Corollary~\ref{cor3.8} there exists a subsequence $(n_k)_{k \in \N}$, a probability space $(\widetilde{\Omega}, \widetilde{\mathcal{F}}, \widetilde{\mathbb{P}})$ and, on this space, $\mathcal{Z}_{T}$-valued random variables $\widetilde{u}, \widetilde{u}_{n_k}, k \ge 1$ such that
\begin{equation}
\label{eq:6.58}
\widetilde{u}_{n_k}\,\mbox{ has the same law as }\, u_{n_k}\,\mbox{ and }\, \widetilde{u}_{n_k} \to \widetilde{u} \,\mbox{ in }\,\mathcal{Z}_{T},\quad \tp\mbox{-a.s.}
\end{equation}
$\tunk \to \tu \,\mbox{ in }\,\mathcal{Z}_{T}$ $\tp\mbox{-a.s.}$ precisely means that
\begin{align*}
\tu_{n_k} &\to \tu \, \mbox{ in } \ccal([0,T]; {\rm U}^\prime),\\
\tu_{n_k} &\rightharpoonup \tu \, \mbox{ in } L^2(0,T; \rm{D}(\A)),\\
\tu_{n_k} &\to \tu \, \mbox{ in } L^2(0,T; \rH_{loc}),\\
\tu_{n_k} &\to \tu \, \mbox{ in } \ccal([0,T]; \rV_w).
\end{align*}
Let us denote the subsequence $(\widetilde{u}_{n_k})$ again by $(\tun)_{n \in \N}$.

By Theorem~B.\ref{thmb.1}, $\ccal([0,T]; \rH_n)$ is a Borel subset of $\ccal([0,T];{\rm U}^\prime) \cap L^2(0,T; \rH_{loc})$. Since $u_n \in \ccal([0,T]; \rm{H}_n)$, $\mathbb{P}$-a.s., and $\tun$, $u_n$ have the same laws on $\mathcal{Z}_T$, thus
\begin{equation}
\label{eq:equal_laws}
\emph{law}(\tun)\left(\ccal([0,T]; \rm{H}_n)\right) = 1, \quad n \in \N\,.
\end{equation}

Since $\ccal([0,T]; \rV) \cap \zcal_T$ and $L^2(0,T; \rm{D}(\A )) \cap \mathcal{Z}_T$ are Borel subsets of $\zcal_T$ (Theorem~B.\ref{thmb.1}) and $\tun$ and $u_n$ have the same laws on $\zcal_T$; from \eqref{eq:6.44} and \eqref{eq:6.29}, we have for $p \in [1,3]$
\begin{align}
\label{eq:6.59}
&\sup_{n \in \N} \widetilde{\E} \left(\sup_{0 \le s \le T} \|\tuns\|^{2p}_{\rm{V}} \right) \le C_1(p)\,,\\
\label{eq:6.60}
&\sup_{n \in \N} \widetilde{\E} \left[ \int_0^T |\tuns|^2_{\rm{D}(\A )}\,ds \right] \le C_2(\|u_0\|^2_\rV)\,.
\end{align}
Also, $\ccal([0,T]; \rH_n)$ is continuously embedded in $L^4(0,T; L^4)$ and $\tun$, $u_n$ have same law $\mu$ on $\ccal([0,T]; \rH_n)$, therefore we have
\begin{align*}
& \widetilde{\E} \int_0^T \|\tuns\|^4_{L^4}\,ds  = \int_{\widetilde{\Omega}} \left[ \int_0^T \|\widetilde{u}_n(s,\omega)\|^4_{L^4}\,ds\right] \,d\widetilde{\mathbb{P}}(\omega) \\
& \quad = \int_{L^4(0,T; L^4)}\left[\int_0^T\|y\|_{L^4}^4\,ds\right]\,d\mu(y) = \int_{\ccal([0,T]; \rH_n)}\left[\int_0^T\|y\|_{L^4}^4\,ds\right]\,d\mu(y) \\
& \quad = \int_{L^4(0,T; L^4)}\left[\int_0^T\|y\|_{L^4}^4\,ds\right]\,d\mu(y) = \int_{\Omega}\left[ \int_0^T \|{u}_n(s,\omega)\|^4_{L^4}\,ds \right]\,d{\mathbb{P}}(\omega)\\
& \quad = \E \int_0^T \|u_n(s)\|^4_{L^4}\,ds \,.
\end{align*}
Thus, by estimate \eqref{eq:6.30} we infer
\begin{equation}
\label{eq:L4estimate}
\sup_{n \in N}\widetilde{\E}\int_0^T\|\tuns\|^4_{L^4}\,ds \le C_3(|u_0|^2_{\rH})\,.
\end{equation}

By inequality \eqref{eq:6.60} we infer that the sequence $(\tun)$ contains a subsequence, still denoted by $(\tun)$ convergent weakly in $L^2([0,T] \times \tom; \rm{D}(\A ))$. Since by \eqref{eq:6.58} $\tp$-a.s $\tun \to \widetilde{u}$ in $\mathcal{Z}_{T}$, we conclude that $\widetilde{u} \in L^2([0,T] \times \tom; \rm{D}(\A ))$, i.e.
\begin{equation}
\label{eq:6.61}
\widetilde{\E} \left[\int_0^T |\tus|^2_{\rm{D}(\A )}\,ds \right] < \infty\,.
\end{equation}
Similarly by inequality \eqref{eq:6.59} for $p = 1$ we can choose a subsequence of $(\tun)$ convergent weak star in the space $L^2(\tom; L^\infty(0,T; \rm{V}))$ and, using \eqref{eq:6.58}, we infer that
\begin{equation}
\label{eq:6.62}
\widetilde{\E} \left(\sup_{0 \le s \le T} \|\tus\|^2_{\rm{V}} \right) < \infty\,.
\end{equation}

For each $n \geq 1$, let us consider a process $\widetilde{M}_n$ with trajectories in $\ccal([0,T]; \rH_n)$ in particular in $\ccal([0,T];\rm{H})$ defined by
\begin{align}
\label{eq:6.63}
\widetilde{M}_n(t) & = \widetilde{u}_n(t) - P_n \widetilde{u}(0) + \int_0^t \A \tuns\,ds + \int_0^t B_n(\tuns)\,ds  \nonumber \\
&~~~ + \int_0^t g_n(\tuns) \, ds - \int_0^t f_n(\tuns) \, ds, \quad \quad \quad  t \in [0,T]\,.
\end{align}

\begin{lemma}
\label{lemma6.17}
$\widetilde{M}_n$ is a square integrable martingale with respect to the filtration $\widetilde{\mathbb{F}}_n = (\widetilde{\mathcal{F}}_{n,t})$, where $\widetilde{\mathcal{F}}_{n,t} = \sigma\{\tuns, s \leq t\}$, with the quadratic variation
\begin{equation}
\label{eq:6.64}
 \langle \langle \widetilde{M}_n \rangle \rangle_t = \int_0^t \|G_n(s, \tuns)\|^2_{\mathcal{L}_2(\ell^2; \rH)}\, ds\,.
\end{equation}
\end{lemma}

\begin{proof}
Indeed since $\tun$ and $u_n$ have the same laws, for all $s, t \in [0,T]$, $s \leq t$, then for all bounded continuous functions $h$ on $\ccal([0,s]; \rV_w)$, and all $\psi, \zeta \in \rV_\gamma$ $(\gamma > \frac{d}{2})$, we have

\begin{equation}
\label{eq:6.65}
\widetilde{\Eb} \left[ \langle \widetilde{M}_n(t) - \widetilde{M}_n(s) , \psi \rangle h(\widetilde{u}_{n|[0,s]}) \right] = 0
\end{equation}
and
\begin{align}
\label{eq:6.66}
\widetilde{\Eb} \Big[ &\Big( \langle \widetilde{M}_n(t), \psi \rangle \langle \widetilde{M}_n(t), \zeta \rangle - \langle \widetilde{M}_n(s), \psi \rangle \langle \widetilde{M}_n(s), \zeta \rangle \nonumber \\
& - \int_s^t \langle \left(G(\sigma, \tunsi)\right)^\ast P_n \psi , \left(G(\sigma, \tunsi)\right)^\ast P_n \zeta \rangle_{\ell^2}\, d \sigma \Big) \cdot h(\widetilde{u}_{n|[0,s]}) \Big] = 0\,.
\end{align}
\end{proof}

\begin{lemma}
\label{lemma_cont_martingale}
Let us define a process $\widetilde{M}$ for $t \in [0,T]$ by
\begin{align}
\label{eq:6.67}
\widetilde{M}(t) & = \widetilde{u}(t) - \widetilde{u}(0) + \int_0^t \A  \tus\,ds + \int_0^t B(\tus)\,ds  \nonumber\\
&~~~ + \int_0^t \Pi(g(|\tus|^2)\tus) \,ds - \int_0^t \Pi\,f(\tus)\,ds\,.
\end{align}
Then $\widetilde{M}$ is a $\rH$-valued continuous process.
\end{lemma}

\begin{proof}
Since $\widetilde{u} \in \ccal([0,T]; \rV)$ we just need to show that the remaining terms on the r.h.s. of \eqref{eq:6.67} are $\rH$-valued a.s. and well-defined.

Using the Cauchy-Schwarz inequality repeatedly and \eqref{eq:6.61} we have the following inequalities
\begin{align*}
\widetilde{\E} \int_0^T \left| \A  \tus \right|_\rH\,ds \le T^{1/2} \left(\widetilde{\E} \int_0^T |\tus|^2_{\rm{D}(\A )}\,ds \right)^{1/2} < \infty\,.
\end{align*}
Since $H^{k,p}(\R^d) \hookrightarrow L^\infty(\R^d)$ for every $k > d/p$, hence there exists a $C > 0$ such that
$\|u\|_{L^\infty} \le C \,\|u\|_{H^{2,2}}$ for every $u \in H^{2,2}(\R^3)$. Thus by the Cauchy-Schwarz inequality, \eqref{eq:6.61} and \eqref{eq:6.62} we obtain the following estimate

\begin{align*}
\widetilde{\E} & \int_0^T  \left| B (\tus) \right|_\rH\,ds \le T^{1/2} \widetilde{\E} \left(\int_0^T |\tus \cdot \nabla \tus|^2_{L^2}\,ds \right)^{1/2}  \\
& \le T^{1/2}\widetilde{\E} \left(\int_0^T \|\tus\|^2_{L^\infty} |\nabla \tus|_{L^2}^2\,ds \right)^{1/2} \\
& \le T^{1/2}\,C\, \widetilde{\E} \left( \int_0^T |\tus|^2_{\rm{D}(\A )}\|\tus\|^2_{\rV}\,ds \right)^{1/2}\\
& \le T^{1/2}\,C\, \left[\widetilde{\E} \sup_{s \in [0,T]}\|\tus\|^2_\rV  \right]^{1/2} \left[\widetilde{\E} \int_0^T |\tus|^2_{\rm{D}(\A )}\,ds \right]^{1/2} < \infty\,.
\end{align*}
We know that for $d= 3$, $H^{1,2} \hookrightarrow L^6$, thus using \eqref{eq:g_bounded}, \eqref{eq:6.58} and \eqref{eq:6.59}, we get
\begin{align*}
\widetilde{\E} \int_0^T & \left| \Pi g(|\tus|^2)\tus \right|_\rH\,ds \le \widetilde{\E} \int_0^T |g(|\tus|^2)\tus|_{L^2}\,ds  \\
&  \le \widetilde{\E} \int_0^T \|\tus\|_{L^6}^3\,ds  \le C\, \widetilde{\E} \int_0^T \|\tus\|^3_\rV\,ds \\
& \le C \left(\widetilde{\E} \sup_{s \in [0,T]}\|\tus\|^4_\rV \right)^{3/4} T < \infty\,.
\end{align*}
Using the assumptions $(\textbf{H1})$ and \eqref{eq:6.62} we can show that
\begin{align*}
& \widetilde{\E} \int_0^T  \left| \Pi f(\tus) \right|_\rH\,ds \le \widetilde{\E} \int_0^T |f(\tus)|_{L^2}\,ds   \\
&\; \le T^{1/2} \,\widetilde{\E} \left( \int_0^T |f(\tus)|^2_{L^2}\,ds\right)^{1/2} \le T^{1/2}\left(\widetilde{\E} \int_0^T \left(C_f|\tus|^2_\rH + |b_f|_{L^1}\right)\,ds\right)^{1/2} < \infty\,.
\end{align*}
This concludes the proof of the lemma.
\end{proof}

\begin{lemma}
\label{lemma6.18}
Let $\gamma > \frac32$, $u \in L^2(0,T; \rH) \cap L^4(0,T; L^4)$ and $(u_n)_{n \in \N}$ be a bounded sequence in $L^2(0,T; \rH) \cap L^4(0,T; L^4)$ such that $u_n \to u$ in $L^2(0,T; \rH_{loc})$. Then for all $r,t \in [0,T]$ and all $\psi \in \rV_\gamma \colon$

\begin{equation}
\label{eq:6.68}
\lim_{n \to \infty} \int_r^t \langle g(|\uns|^2)\uns , \psi \rangle \,ds = \int_r^t \langle g(|u(s)|^2) u(s), \psi\rangle \,ds\,.
\end{equation}
Here $\langle \cdot, \cdot \rangle$ denotes the duality pairing between $\rV_\gamma$ and $\rV_{\gamma}^\prime$.
\end{lemma}

\begin{proof} We will prove the lemma in two steps.

\noindent \textit{Step I}\\
Let us fix $\gamma > \frac{3}{2}$ and $r,t \in [0,T]$. Assume first that $\psi \in \mathcal{V}$. Then there exists a $R > 0$ such that $\text{supp}(\psi)$ is a compact subset of $\ocal_R$. There exists a constant $C \ge 0$ such that
\begin{align}
\label{eq:6.69}
& |\langle g(|u|^2)u , \psi \rangle|  = \left|\int_{\ocal_R} g(|u(x)|^2)u(x) \psi(x) \,dx\right| \nn\\
&\;  \le |g(|u|^2)|_{L^2(\ocal_R)}|u|^2_{L^2(\ocal_R)} \|\psi\|_{L^\infty(\ocal_R)} \le \big||u|^2 \big |_{L^2(\ocal_R)} |u|_{L^2(\ocal_R)}\|\psi\|_{L^\infty} \nn \\
&\;  \le C \|u\|^2_{L^4}|u|_{L^2(\ocal_R)}\|\psi\|_{\rV_\gamma}, \quad \quad u \in \rH \cap L^4,
\end{align}
where we used \eqref{eq:2.7} to establish the last inequality. We have
\[g(|u_n|^2)u_n - g(|u|^2)u = g(|u_n|^2)(u_n - u) + \left[g(|u_n|^2) - g(|u|^2) \right]u\,.\]
Thus using the estimate \eqref{eq:6.69}, the H\"older inequality, \eqref{eq:g_lipschitz} and the Cauchy-Schwarz inequality, we obtain
\begin{align*}
&\left| \int_r^t \langle g(|\uns|^2)\uns , \psi \rangle \,ds - \int_r^t \langle g_n(|u(s)|^2)u(s), \psi \rangle\,ds \right| \\
& \quad \le \left| \int_r^t \langle g(|\uns|^2)(\uns - u(s)), \psi \rangle\,ds \right| \\
& \qquad + \left| \int_r^t \left\langle \left(g(|\uns|^2) - g(|u(s)|^2)\right)u(s) , \psi \right\rangle \,ds \right| \\
& \quad \le C \int_r^t \|u_n(s)\|^2_{L^4} |u_n(s) - u(s)|_{L^2(\ocal_R)} \|\psi\|_{\rV_\gamma} \,ds \\
& \qquad + \lipzc \int_r^t \left| \langle |u_n(s) - u(s)| \left(|u_n(s)| + |u(s)| \right) u(s) , \psi \rangle \right|\,ds \\
& \quad \le C\|\psi\|_{\rV_\gamma} \int_r^t \|\uns\|^2_{L^4} |u_n(s) - u(s)|_{L^2(\ocal_R)} \,ds \\
& \qquad + \lipzc C\,\|\psi\|_{\rV_\gamma} \int_r^t |\uns - u(s)|_{L^2(\ocal_R)}\big | u(s) \left[ |u_n(s)| + |u(s)| \right] \big |_{L^2(\ocal_R)}\,ds.\\
& \quad \le C \|\psi\|_{\rV_\gamma} \Bigg[|u_n|_{L^4(0,T; L^4)}^2 |u_n - u|_{L^2(0,T; L^2({\ocal_R}))} \\
& \qquad   + \lipzc |u_n - u|_{L^2(0,T; L^2({\ocal_R}))} \left[ \int_r^t \big |u(s) \left[|u_n(s)| + |u(s)| \right] \big |^2_{L^2(\ocal_R)}\,ds \right]^{1/2} \Bigg]\\
& \quad \le C \left[ |u_n|_{L^4(0,T; L^4)}^2 + \lipzc |u|_{L^4(0,T; L^4)}^2 \left( |u_n|_{L^4(0,T; L^4)}^2 + |u|_{L^4(0,T; L^4)}^2 \right)^{1/2} \right] \\
& \qquad \times |u_n - u|_{L^2(0,T; L^2({\ocal_R}))}\|\psi\|_{\rV_\gamma}\,.
\end{align*}
Since $u_n \to u$ in $L^2(0,T; \rH_{loc})$ we infer that \eqref{eq:6.68} holds for every $\psi \in \mathcal{V}$.

\noindent\textit{Step II}\\
Let $\psi \in \rV_\gamma$ and $\varepsilon > 0$. Then there exists a $\psi_\varepsilon \in \mathcal{V}$ such that $\|\psi_\varepsilon - \psi\|_{\rV_\gamma} < \varepsilon.$ Hence, we get
\begin{align}
\label{eq:6.70}
& \left| \langle g(|u_n|^2)u_n - g(|u|^2)u , \psi \rangle \right| \nonumber \\
& \quad \le \left| \langle g(|u_n|^2)u_n - g(|u|^2)u , \psi_\varepsilon \rangle \right| + \left| \langle g(|u_n|^2)u_n - g(|u|^2)u , \psi - \psi_\varepsilon \rangle \right| \nonumber \\
& \quad \le \left| \langle g(|u_n|^2)u_n - g(|u|^2)u , \psi_\varepsilon \rangle \right| \nn \\
& \qquad  + \left[\|g(|u_n|^2)u_n\|_{\rV^\prime_\gamma} + \|g(|u|^2)u\|_{\rV^\prime_\gamma} \right]\|\psi - \psi_\varepsilon\|_{\rV_\gamma}\,.
\end{align}
Since $ \mathcal{V}$ is dense in $\rV_\gamma$, \eqref{eq:6.69} holds for all $\psi \in \rV_\gamma$. In particular, there exists a constant $C > 0$ such that
\begin{equation}
\label{eq:6.71}
\|g(|u|^2)u \|_{\rV^\prime_\gamma} \le C \|u\|^2_{L^4}|u|_{\rH}, \quad \quad u \in \rH \cap L^4\,.
\end{equation}
Using \eqref{eq:6.70}, \eqref{eq:6.71} and the Cauchy-Schwarz inequality we have following inequalities
\begin{align*}
& \left| \int_r^t \langle g(|u_n(s)|^2) u_n(s) - g(|u(s)|^2)u(s) , \psi \rangle \, d s \right| \\
&\le \varepsilon\,C \int_r^t \left(\|u_n(s)\|^2_{L^4}|u_n(s)|_{\rH} + \|u(s)\|^2_{L^4}|u(s)|_{\rH} \right)\, ds\\
& \quad~~ + \left| \int_r^t \langle g(|u_n(s)|^2)u_n(s) - g(|u(s)|^2)u(s) , \psi_\varepsilon \rangle\,ds \right| \\
&  \quad \le  \varepsilon C \left[ \|u_n\|_{L^4(0,T; L^4)}^2\|u_n\|_{L^2(0,T; \rH)} + \|u\|_{L^4(0,T; L^4)}^2\|u\|_{L^2(0,T; \rH)}\right] \\
& \quad~~ + \left| \int_r^t \langle g(|u_n(s)|^2)u_n(s) - g(|u(s)|^2)u(s) , \psi_\varepsilon \rangle\,ds \right|\,.
\end{align*}

Hence by Step I and the assumptions on $u, u_n$ there exists a $M > 0$ such that
\[\limsup_{n \to \infty} \left| \int_r^t \langle g(|u_n(s)|^2) u_n(s) - g(|u(s)|^2)u(s) , \psi \rangle \, d s \right| \le M \varepsilon. \]
Since $\varepsilon > 0$ is arbitrary we conclude the proof.
\end{proof}

\begin{corollary}
\label{cor6.19}
Let $\gamma > \frac{3}{2}$, $u \in L^2(0,T; \rH) \cap L^4(0,T; L^4)$ and $(u_n)_{n \in \N}$ be a bounded sequence in $L^2(0,T; \rH) \cap L^4(0,T; L^4)$ such that $u_n \to u$ in $L^2(0,T; \rH_{loc})$. Then for all $r,t \in [0,T]$ and all $\psi \in \rm{V}_\gamma$
\begin{equation}
\label{eq:6.72}
\lim_{n \to \infty} \int_r^t \langle g(|u_n(s)|^2)u_n(s) , P_n\psi \rangle \,ds = \int_r^t \langle g(|u(s)|^2) u(s), \psi\rangle \,ds\,.
\end{equation}
Here $\langle \cdot, \cdot \rangle$ denotes the dual pairing between $\rm{V}_\gamma$ and $\rm{V}_\gamma^\prime$.
\end{corollary}

\begin{proof}
Let us fix $\gamma > \frac{3}{2}$ and take $r, t \in [0,T]$ and $\psi \in \rm{V}_\gamma$. We have
\begin{align*}
& \int_r^t \langle g(|\uns|^2)\uns , P_n \psi \rangle \,ds \\
& \quad = \int_r^t \langle g(|\uns|^2)\uns , P_n \psi - \psi \rangle \,ds + \int_r^t \langle g(|\uns|^2)\uns , \psi \rangle \,ds \\
 & \quad := I_1(n) + I_2(n)\,.
\end{align*}
We will consider each of these integrals individually. Using the estimate from \eqref{eq:6.71}, we have
\begin{align*}
|I_1(n)| & \le \int_r^t \|g(|\uns|^2)\uns\|_{\rV_\gamma^\prime}\|P_n \psi - \psi\|_{\rV_\gamma}\,ds \\
&\le \|P_n \psi - \psi \|_{\rV_\gamma}\int_r^t \|u_n(s)\|_{L^4}^2|u_n(s)|_{\rH}\,ds\,.
\end{align*}
Since the sequence $(u_n)_{n \in \N}$ is bounded in $L^2(0,T; \rH) \cap L^4(0,T; L^4)$ and $P_n \psi \to \psi$ in $V_\gamma$, we infer that $\lim_{n \to \infty} I_1(n) = 0$. By Lemma~\ref{lemma6.18} we infer that
\[\lim_{n \to \infty} I_2(n) = \int_r^t \langle g(|u(s)|^2)u(s), \psi \rangle \,ds\,.\]
\end{proof}

\begin{lemma}
\label{lemma6.20}
For all $s, t \in [0,T]$ such that $s \leq t$ and $\gamma > 3/2 \colon$
\begin{itemize}
\item[(a)] $\lim_{n \rightarrow \infty}\langle \widetilde{u}_n(t), P_n \psi\rangle = \langle\widetilde{u}(t), \psi \rangle,~~ \widetilde{\mathbb{P}}$-a.s., $ \psi \in \rV,$
\item[(b)] $\lim_{n \rightarrow \infty} \int_s^t \langle \A  \tunsi, P_n \psi \rangle_\rH\,ds = \int_s^t \langle \A  \tusi, \psi \rangle_\rH\,d \sigma,~~ \widetilde{\mathbb{P}}$-a.s., $\psi \in \rH$,
\item[(c)]
 $\lim_{n \rightarrow \infty} \int_s^t \langle B(\tunsi), P_n \psi \rangle\,d\sigma = \int_s^t\langle B(\tusi), \psi \rangle\,d \sigma,~~ \widetilde{\mathbb{P}}$-a.s., $\psi \in \rV_\gamma$,
\item[(d)] $\lim_{n \rightarrow \infty} \int_s^t  \langle g(|\tunsi|^2) \tunsi, P_n \psi \rangle\,d \sigma = \int_s^t \langle g(|\tusi|^2) \tusi, \psi \rangle\,d \sigma,~ \widetilde{\mathbb{P}}$-a.s., $\psi \in \rV_\gamma$,
\item[(e)] $\lim_{n \rightarrow \infty} \int_s^t \langle f(\tunsi), P_n \psi \rangle\,d \sigma = \int_s^t \langle f(\tusi), \psi \rangle\,d \sigma,~~ \widetilde{\mathbb{P}}$-a.s., $\psi \in \rV_\gamma$,
\end{itemize}
where $\langle \cdot, \cdot \rangle$ denotes the duality pairing between appropriate spaces.
\end{lemma}

\begin{proof}
Let us fix $s, t \in [0,T]$, $s \leq t$ and $\gamma > \frac{3}{2}$. By \eqref{eq:6.58} we know that $\mathbb{P}$-a.s.
\begin{equation}
\label{eq:6.73}
\widetilde{u}_n \rightarrow \widetilde{u}~in~\ccal([0,T];{\rm U}^\prime) \cap L^2_w(0,T; {\rm D}(\A )) \cap L^2(0,T; \rH_{loc}) \cap \ccal([0,T]; V_w).
\end{equation}

Let $\psi \in \rV$. Since $\widetilde{u}_n \rightarrow \widetilde{u}$ in $\ccal([0,T]; \rV_w)$ $\widetilde{\mathbb{P}}$-a.s., from \eqref{eq:6.59} $\widetilde{u}_n$ is uniformly  bounded in $\ccal([0,T]; \rV_w)$ and $P_n \psi \rightarrow \psi$ in $\rm{V}$, thus
\begin{align*}
& \lim_{n \to \infty} \langle\widetilde{u}_n(t), P_n \psi\rangle - \langle\widetilde{u}(t), \psi\rangle  \\
&\; = \lim_{n \to \infty}\langle\widetilde{u}_n(t) - \widetilde{u}(t), \psi\rangle + \lim_{n \to \infty}\langle\widetilde{u}_n(t), P_n \psi - \psi\rangle = 0 \quad \widetilde{\mathbb{P}}\text{-a.s.}
\end{align*}
Hence we infer that assertion $(a)$ holds.

Let $\psi \in \rH$. Since by \eqref{eq:6.73} $\widetilde{u}_n \rightarrow \widetilde{u}$ in $L^2_w(0,T; \rm{D}(\A ) )$ $\widetilde{\mathbb{P}}$-a.s., from \eqref{eq:6.60} $\tun$ is uniformly bounded in $L_w^2(0,T; \rm{D}(\A ))$ and $P_n \psi \rightarrow \psi$ in $\rm{H}$. Thus, we have, $\widetilde{\mathbb{P}}$-a.s.,
\begin{align*}
&\lim_{n \to \infty} \int_s^t \langle \A  \tunsi, P_n \psi \rangle_{\rm{H}} ~ d \sigma - \int_s^t \langle \A  \tusi, \psi\rangle_{\rm{H}}~d \sigma \\
& = \lim_{n \to \infty} \int_s^t \langle \A  \tunsi - \A  \tusi, \psi\rangle_{\rm{H}}~d \sigma + \lim_{n \to \infty} \int_s^t \langle \A  \tunsi, P_n \psi - \psi\rangle_{\rm{H}}~d \sigma \\
&= 0\,.
\end{align*}
Hence, we have shown that assertion $(b)$ is true.

Assertion $(c)$ follows directly for every $\psi \in \rV_\gamma$ from \cite[Lemma~B.1]{[BM13]} and a modification of Corollary~\ref{cor6.19}.

By \eqref{eq:6.73} $\widetilde{u}_n \rightarrow \widetilde{u}$ in $L^2(0,T; \rH_{loc})$. From Lemma~\ref{lemma6.13}, \eqref{eq:6.58} and \eqref{eq:L4estimate} the sequence ($\widetilde{u}_n$) is bounded in $L^2(0,T; \rH) \cap L^4(0,T; L^4)$ and $\widetilde{u} \in L^2(0,T; \rH) \cap L^4(0,T; L^4)$. Thus, using Corollary~\ref{cor6.19} we infer that $(d)$ holds for every $\psi \in \rV_\gamma$.

Now we are left to deal with $(e)$. Let $\psi \in \rV_\gamma$,
\begin{align*}
& \int_s^t \langle f(\tunsi) , P_n \psi \rangle \,ds  - \int_0^t \langle f(\widetilde{u}(\sigma), \psi \rangle\,d\sigma \\
& \quad = \int_0^t \langle f(\tunsi) - f(\widetilde{u}(\sigma)), \psi \rangle\,d \sigma + \int_0^t \langle f(\widetilde{u}_n(\sigma)), P_n \psi - \psi \rangle\,d\sigma\,.
\end{align*}
Since $\rV_\gamma \hookrightarrow \rH$, by the Cauchy-Schwarz inequality we have
\begin{align*}
& \int_s^t \langle f(\tunsi) , P_n \psi \rangle \,ds  - \int_0^t \langle f(\widetilde{u}(\sigma), \psi \rangle\,d\sigma \\
& \quad \le \int_0^t \langle f(\tunsi) - f(\widetilde{u}(\sigma)), \psi \rangle \,d\sigma + \int_s^t \|f(\uns)\|_{\rV_\gamma^\prime} \|P_n \psi - \psi\|_{\rm{V}_\gamma} \,d\sigma \\
& \quad \le \int_0^t \langle f(\tunsi) - f(\widetilde{u}(\sigma)), \psi \rangle\,d\sigma + \|P_n \psi - \psi\|_{\rm{V}_\gamma} \int_s^t \left(C_f\|\tunsi\|_\rH + |b_f|_{L^1}\right)\,d\sigma \\
& \quad := I_1(n) + I_2(n)\,.
\end{align*}
Since $\widetilde{u}_n \to  \widetilde{u}$ in $L^2(0,T; \rH_{loc})$ and $\widetilde{u}_n$ is a bounded sequence in $L^2(0,T; \rH)$. $I_1(n)$ can be shown to converge to zero as $n \to \infty$ following the methodology of Lemma~\ref{lemma6.18} and Corollary~\ref{cor6.19}. Since $P_n \psi \to \psi$ in $\rm{V}_\gamma$, $I_2(n) \to 0$ as $n \to \infty$. This completes the proof of Lemma~\ref{lemma6.20}.
\end{proof}

The proofs of Lemmas~\ref{lemma6.21}, \ref{lemma6.22} and \ref{lemma6.23} follow the similar methodology as that of Lemmas 5.6 - 5.8 \cite{[BM13]} and Lemmas~5.9 - 5.11 \cite{[BD16]}.

Let $h$ be the bounded continuous function on $\ccal([0,T]; \rV_w)$.

\begin{lemma}
\label{lemma6.21}
Let $\gamma > \frac32$. For all $s, t \in [0,T]$, such that $s \leq t$ and all $\psi \in \rm{V}_\gamma$
\begin{align}
\label{eq:convg_martingale}
 &\lim_{n \rightarrow \infty} \widetilde{\Eb} \left[ \langle\widetilde{M}_n(t) - \widetilde{M}_n(s), \psi\rangle h(\widetilde{u}_{n|[0,s]})\right]\nn \\
 &\; = \widetilde{\Eb} \left[ \langle\widetilde{M}(t) - \widetilde{M}(s), \psi\rangle h(\widetilde{u}_{|[0,s]}) \right]\,.
 \end{align}
\end{lemma}

\begin{proof}
Let us fix $s, t \in [0,T], s \leq t$ and $\psi \in \rm{V}_\gamma$. By Eq. \eqref{eq:6.63}, we have
\begin{align*}
&\langle\widetilde{M}_n(t) - \widetilde{M}_n(s), \psi\rangle = \langle \widetilde{u}_n(t) - \widetilde{u}_n(s), P_n \psi \rangle + \int_s^t \langle \A  \tunsi, P_n \psi \rangle \,d \sigma\\
&~~~~  + \int_s^t \langle B(\tunsi), P_n \psi \rangle \,d \sigma + \int_s^t \langle g(|\tunsi|^2)\,\tunsi, P_n \psi \rangle \,d \sigma \\
&~~~~ - \int_s^t \langle f(\tunsi), P_n \psi \rangle \,d \sigma\,.
\end{align*}
By Lemma \ref{lemma6.20}, we infer that
\begin{equation}
\label{eq:6.74}
\lim_{n \rightarrow \infty} \langle \widetilde{M}_n(t) - \widetilde{M}_n(s), \psi \rangle = \langle \widetilde{M}(t) - \widetilde{M}(s), \psi \rangle, \quad \widetilde{\mathbb{P}}\text{-a.s.}
\end{equation}
In order to prove \eqref{eq:convg_martingale} we first observe that since $\tun \to \tu$ in $\mathcal{Z}_{T}$, in particular in $\ccal([0,T]; \rV_w)$ and $h$ is a bounded continuous function on $\ccal([0,T]; \rV_w)$, we get
\begin{equation}
\label{eq:6.75}
\lim_{n \to \infty} h(\widetilde{u}_{n|[0,s]}) = h(\tu_{|[0,s]}),
\end{equation}
and
\begin{equation}
\label{eq:h_bounded}
 \sup_{n \in \N} \|h(\widetilde{u}_{n|[0,s]})\|_{L^\infty} < \infty\,.
 \end{equation}
Let us define a sequence of $\R$-valued random variables $\colon$
\[ f_n(\omega):= \left[ \langle \widetilde{M}_n(t, \omega), \psi \rangle - \langle \widetilde{M}_n(s, \omega), \psi \rangle \right]h(\widetilde{u}_{n|[0,s]}), ~~~~ \omega \in \widetilde{\Omega}\,.\]
We will prove that the functions $\{f_n\}_{n \in \mathbb{N}}$ are uniformly integrable in order to apply the Vitali's convergence theorem. We claim that
\begin{equation}
\label{eq:6.76}
\sup_{n \in \N} \widetilde{\Eb}[|f_n|^2] < \infty\,.
\end{equation}
Since, $\rH \hookrightarrow \rV_\gamma^\prime$ then by the Cauchy-Schwarz inequality, for each $n \in \mathbb{N}$ we have
\begin{equation}
\label{eq:6.77}
\widetilde{\Eb} [|f_n|^2] \leq 2c \|h\circ \widetilde{u}_n\|^2_{L^{\infty}}|\psi|^2_{\rV_\gamma} \widetilde{\Eb}\left[|\widetilde{M}_n(t)|^2_H + |\widetilde{M}_n(s)|^2_H \right]\,.
\end{equation}
Since, $\widetilde{M}_n$ is a continuous martingale with quadratic variation defined in \eqref{eq:6.64}, by the\\ Burkholder-Davis-Gundy inequality we obtain
\begin{equation}
\label{eq:6.78}
\widetilde{\Eb} \left[ \sup_{t \in [0,T]} |\widetilde{M}_n(t)|_{\rm{H}}^2 \right] \leq c \widetilde{\Eb} \left[  \int_0^T \|G_n(\sigma, \tunsi)\|_{\mathcal{L}_2(\ell^2; \rH)}^2\,d \sigma \right]\,.
\end{equation}
Since, $P_n \colon \rH \to \rH$ is a contraction and by Lemma~\ref{lemma2.1}, \eqref{eq:6.44} for $p = 1$, we have
\begin{align}
\label{eq:6.79}
\widetilde{\Eb} & \left[\int_0^T \|G_n(\sigma, \tunsi)\|_{\mathcal{L}_2(\ell^2; \rH)}^2\, d \sigma  \right] \leq \widetilde{\Eb} \left[ \int_0^T \|G(\sigma, \tunsi)\|_{\mathcal{L}_2(\ell^2; \rH)}^2\,d \sigma\right] \nonumber \\
& \leq \widetilde{\Eb} \left[ \int_0^T \frac14 |\nabla \tunsi |^2_{L^2} \,d \sigma \right] \leq \widetilde{\Eb}\left[ \sup_{\sigma \in [0,T]}\|\tunsi\|_{\rm{V}}^2\right]T < \infty\,.
\end{align}
Then by \eqref{eq:6.77} and \eqref{eq:6.79} we see that \eqref{eq:6.76} holds. Since the sequence $\{f_n\}_{n \in \mathbb{N}}$ is uniformly integrable and by \eqref{eq:6.74} it is $\widetilde{\mathbb{P}}$-a.s. point-wise convergent, then application of the Vitali's convergence theorem completes the proof of the Lemma.
\end{proof}

\begin{remark}
\label{rem_strong}
Using Burkholder-Davis-Gundy inequality we have proved a stronger claim \eqref{eq:6.78} than what we needed.
\end{remark}

From Lemma~\ref{lemma6.17} and Lemma~\ref{lemma6.21} we have the following corollary.
\begin{corollary}
\label{cor_martingale_6}
For all $s,t \in [0,T]$ such that $s \leq t \colon$
\[\E \left(\widetilde{M}(t) - \widetilde{M}(s) \big|\widetilde{\mathcal{F}}_t \right) = 0\,.\]
\end{corollary}

\begin{lemma}
\label{lemma6.22}
For all $s,t \in [0,T]$ such that $s \leq t$ and all $\psi, \zeta \in \rm{V}_\gamma$
\begin{align*}
\lim_{n \rightarrow \infty} & \widetilde{\Eb} \Big[ \Big( \langle \widetilde{M}_n(t), \psi \rangle \langle \widetilde{M}_n(t), \zeta \rangle - \langle \widetilde{M}_n(s), \psi \rangle \langle \widetilde{M}_n(s), \zeta \rangle \Big) h(\widetilde{u}_{n|[0,s]}) \Big] \\
& = \widetilde{\Eb} \Big[\Big( \langle \widetilde{M}(t), \psi \rangle \langle \widetilde{M}(t), \zeta \rangle - \langle \widetilde{M}(s), \psi \rangle \langle \widetilde{M}(s), \zeta \rangle \Big) h(\widetilde{u}_{|[0,s]}) \Big],
\end{align*}
where $\langle \cdot, \cdot \rangle$ denotes the appropriate duality pairing.
\end{lemma}

\begin{proof}
Let us fix $s, t \in [0,T]$ such that $s \leq t$ and $\psi, \zeta \in \rV_\gamma$ and define $\R$-valued random variables $f_n$ and $f$ for $\omega \in \widetilde{\Omega}$ by
\begin{align*}
&f_n(\omega) := \Big( \langle \widetilde{M}_n(t, \omega), \psi \rangle \langle \widetilde{M}_n(t, \omega), \zeta \rangle - \langle \widetilde{M}_n(s, \omega), \psi \rangle \langle \widetilde{M}_n(s, \omega), \zeta \rangle \Big) h(\widetilde{u}_{n|[0,s]}(\omega)),\\
&f(\omega) := \Big( \langle \widetilde{M}(t, \omega), \psi \rangle \langle \widetilde{M}(t, \omega), \zeta \rangle - \langle \widetilde{M}(s, \omega), \psi \rangle \langle \widetilde{M}(s, \omega), \zeta \rangle \Big) h(\widetilde{u}_{|[0,s]}(\omega))\,.
\end{align*}
By Lemma \ref{lemma6.20} or more precisely by \eqref{eq:6.74} and \eqref{eq:6.75} we infer that $\lim_{n \rightarrow \infty} f_n(\omega) = f(\omega)$, for $\widetilde{\mathbb{P}}$ almost all $\omega \in \tom$.\\
We will prove that the functions $\{f_n\}_{n \in \mathbb{N}}$ are uniformly integrable. We claim that for some $r > 1$,
\begin{equation}
\label{eq:6.80}
\sup_{n \in \N} \widetilde{\Eb} \left[|f_n|^r\right] < \infty\,.
\end{equation}
For each $n \in \mathbb{N}$ as before we have
\begin{equation}
\label{eq:6.81}
\widetilde{\Eb} \left[ |f_n|^r\right] \leq C \|h \circ \widetilde{u}_n\|_{L^{\infty}}^r \|\psi\|^r_{\rm{V}_\gamma} \|\zeta\|^r_{\rm{V}_\gamma} \widetilde{\Eb} \left[|\widetilde{M}_n(t)|_\rH^{2r} + |\widetilde{M}_n(s)|_\rH^{2r} \right]\,.
\end{equation}
Since, $\widetilde{M}_n$ is a continuous martingale with quadratic variation defined in \eqref{eq:6.64}, by the Burkholder-Davis-Gundy inequality we obtain
\begin{equation}
\label{eq:6.82}
\widetilde{\Eb} \left[ \sup_{t \in [0,T]} |\widetilde{M}_n(t)|_\rH^{2r} \right] \leq c \widetilde{\Eb} \left[  \int_0^T \|G_n(\sigma, \tunsi)\|_{\mathcal{L}_2(\ell^2; \rH)}^2\,d \sigma \right]^r\,.
\end{equation}
Since, $P_n \colon \rH \to \rH$ is a contraction and by Lemma~\ref{lemma2.1} we have
\begin{align}
\label{eq:6.83}
\widetilde{\Eb} & \left[  \int_0^T \|G_n(\sigma, \tunsi)\|_{\mathcal{L}_2(\ell^2; \rH)}^2\,d\sigma \right]^r  \leq \widetilde{\E}\left[  \int_0^T \|G(\sigma, \tunsi)\|_{\mathcal{L}_2(\ell^2; \rH)}^2\,d\sigma \right]^r \nonumber \\
& \leq \widetilde{\Eb} \left[ \int_0^T \frac14 |\nabla \tunsi|_{L^2}^2\,d \sigma \right]^r \leq C_r T^{r-1}  \widetilde{\Eb}\left[\int_0^T \|\tunsi\|^{2r}_\rV \, d \sigma\right]\nonumber \\
& \leq C_r T^r \widetilde{\E} \left[\sup_{\sigma \in [0,T]} \|\tunsi\|^{2r}_\rV \right]\,.
\end{align}
Thus if $r \in [1,3]$ then by \eqref{eq:6.44}, \eqref{eq:h_bounded} and \eqref{eq:6.81} - \eqref{eq:6.83} we infer that \eqref{eq:6.80} holds. Hence by application of the Vitali's convergence theorem
\begin{equation}
\label{eq:6.84}
\lim_{n \rightarrow \infty} \widetilde{\Eb}[f_n] = \widetilde{\Eb}[f]\,.
\end{equation}
\end{proof}

We will be using the following notations in following lemmata. $\rV^\prime(\ocal_R)$ is the dual space to $\rV(\ocal_R)$, where
\[\rV({\mathcal{O}_R}) := \, \mbox{the closure of}\, \mathcal{V}({\mathcal{O}_R})\, \mbox{in}\, H^1({\mathcal{O}_R}, \R^3),\]
where
\[\mathcal{V}({\mathcal{O}_R}) := \{u \in C_0^\infty(\R^3;\R^3) \colon {\rm div}\,u = 0,\, {\rm supp}\,u \subset \mathcal{O}_R\}.\]
We recall that $\rH_{\mathcal{O}_R}$ is the space of restrictions to the subset ${\mathcal{O}_R}$ of elements of the space $\rH$ i.e.,
\[\rH_{\mathcal{O}_R} := \left\{u_{|{\mathcal{O}_R}} \colon u \in \rH \right\},\]
with the scalar product defined by
\[\langle u , \v \rangle_{\rH_{\mathcal{O}_R}} :=  \int_{\mathcal{O}_R} u(x)\,\v(x)\,dx, \quad u, \v \in \rH_{\mathcal{O}_R}\,.\]

\begin{lemma}
\label{lemma6.23}
The map $G \colon \rH_{\mathcal{O}_R} \to \mathcal{L}_2(\ell^2; \rV^\prime({\mathcal{O}_R}))$ given by \eqref{eq:2.12} is well defined and there exists some constant $C_R > 0$ such that
\begin{equation}
\label{eq:6.85}
\|G(u)\|_{\mathcal{L}_2(\ell^2;\rV^\prime(\ocal_R))} \le C_R\|u\|_{\rH_{\mathcal{O}_R}}, \quad u \in \rH\,.
\end{equation}
Moreover, for every $\psi \in \vcal$ the mapping $\rH \ni u \mapsto \langle G(u), \psi \rangle \in \ell^2$ is continuous, if in the space $\rH$ we consider the Fr\'echet topology inherited from the space $L^2_{loc}(\R^3,\R^3)$.
\end{lemma}

\begin{proof}
Let $\sigma=({\sigma}^1, ..., {\sigma}^{d}) : \overline{\ocal} \to \rd $ and fix $R > 0$. Let $u \in \mathcal{V}(\ocal_R)$. Then
\begin{equation}  \label{E:differentiation}
\sum_{j=1}^{d} \frac{\partial }{\partial {x}_{j}} ({\sigma}^{j} u)
 = \sum_{j=1}^{d} \biggl(  \frac{\partial {\sigma}^{j} }{\partial {x}_{j}} u
    + {\sigma}^{j} \frac{\partial u }{\partial {x}_{j}} \biggr)
 = (\divv \sigma ) u + \sum_{j=1}^{d} {\sigma}^{j} \frac{\partial u }{\partial {x}_{j}}\,.
\end{equation}
Let $\v \in \mathcal{V}(\ocal_R)$. Since, $\v$ on the boundary $\partial {\ocal }_{R}$ is equal to zero, thus using the integration by parts formula, we obtain for $\v \in \mathcal{V}(\ocal_R)$
\begin{align*}
 \int_{\ocal_R} \bigl( \sum_{j=1}^{d} {\sigma}^{j} \frac{\partial u }{\partial {x}_{j}}  \bigr) \v \, dx
 & = \sum_{j=1}^{d} \int_{\ocal_R} \frac{\partial }{\partial {x}_{j}} ({\sigma}^{j} u) \v \, dx  - \int_{\ocal_R} (\divv \sigma) u\, \v \, dx   \\
 & = - \sum_{j=1}^{d} \int_{\ocal_R} ({\sigma}^{j} u) \frac{\partial \v}{\partial {x}_{j}}  \, dx  - \int_{\ocal_R} (\divv \sigma) u\, \v \, dx\,.
\end{align*}
Using the H\"{o}lder inequality, we obtain
\begin{align}
\label{E:Holder_est_O_R}
& \bigl| \int_{{\ocal }_{R}} \bigl( \sum_{j=1}^{d} {\sigma}^{j} \frac{\partial u }{\partial {x}_{j}}  \bigr) \v \, dx  \bigr| \nn \\
  &\quad \le \| \sigma {\| }_{{L}^{\infty }} {| u | }_{{\rH}_{{\ocal }_{R}}} {\| \v \| }_{\rV ({\ocal }_{R})}
 + \| \divv \sigma {\| }_{{L}^{\infty }} {| u | }_{{\rH}_{{\ocal }_{R}}} {\| \v \| }_{\rV({\ocal }_{R})}.
\end{align}
Therefore, if we define a linear functional ${\hat{B}}_{R}$ by
\[  {\hat{B}}_{R} \v :=\int_{{\ocal}_{R}} \bigl( \sum_{j=1}^{d} {\sigma}^{j} \frac{\partial u }{\partial {x}_{j}} \bigr) \v \, dx,
 \qquad \v \in \mathcal{V}({\ocal }_{R})\,,\]
we infer that it is bounded in the norm of the space $\rV({\ocal }_{R})$. Thus it can be uniquely extended to a linear bounded functional (denoted also by ${\hat{B}}_{R}$) on $\rV({\ocal }_{R})$. Moreover, by estimate \eqref{E:Holder_est_O_R} we have the following inequality
\[ \|\hat{B}_R \|_{\rV^\prime(\ocal_R)} \le \bigl(  \| \sigma \|_{L^\infty} + \| \divv \sigma \|_{L^\infty} \bigr) | u |_{\rH_{\ocal_{R}}} \]
or equivalently
\begin{equation}  \label{E:hat_B_R_est}
\| (\sigma \cdot \nabla ) u {\| }_{\rV^{\prime }({\ocal }_{R})}
\le \bigl( \|  \sigma {\| }_{{L}^{\infty }} + \| \divv \sigma {\| }_{{L}^{\infty }} \bigr) \cdot {| u | }_{{\rH}_{{\ocal }_{R}}}\,.
\end{equation}
Since by equality \eqref{eq:2.12}, $G(u)(e_j) = \Pi\left[(\sigma_j \cdot \nabla) u\right]$, where $\left\{e_j\right\}_{j=1}^\infty$ is an orthonormal basis of $\ell^2$, we get by estimate \eqref{E:hat_B_R_est}
\begin{align*}
& \|G(u)\|_{{\mathcal{L}_2}(\ell^2;\rV^\prime(\ocal_R))} = \left[\sum_{j=1}^\infty \|G(u)(e_j)\|^2_{\rV^\prime(\ocal_R)}\right]^{1/2} \nn \\
& \quad \le \bigl( \|  \sigma {\| }_{\ell^2} + \| \divv \sigma {\| }_{\ell^2} \bigr) \cdot {| u | }_{{\rH}_{{\ocal }_{R}}}\,.
\end{align*}
Therefore, using the assumption $\mathbf{(H2)}$, $G(u) \in \mathcal{L}_2(\ell^2,\rV^{\prime }({\ocal }_{R}))$ and
\[
\| G(u)  {\| }_{\mathcal{L}_2 (\ell^2 ,\rV^{\prime }({\ocal }_{R}))}  \le {C}_{R} \cdot  {|u|}_{{\rH}_{{\ocal }_{R}}}\,.
\]
By estimate \eqref{eq:6.85} and the continuity of the embedding $\mathcal{L}_2(\ell^2, \rV^{\prime }({\ocal }_{R})) \hookrightarrow \mathcal{L}(\ell^2, \rV^{\prime }({\ocal }_{R}))$, we obtain
\[ \|G(u)y{\|}_{\rV^{\prime }({\ocal }_{R})} \le C(R) {|u|}_{{\rH}_{{\ocal }_{R}}}^{} \|y\|_{\ell^2} , \qquad u \in \rH , \quad y \in \ell^2
\]
for some constant $C(R)>0$. Thus, for any $\psi  \in \rV({\ocal }_{R})$
\begin{equation}
\label{eq:G_estimate}
|(G(u)y)\psi | \le C(R) {|u|}_{{\rH}_{{\ocal }_{R}}}^{} \|y\|_{\ell^2} \|\psi\|_{\rV({\ocal }_{R})} , \qquad u \in \rH , \quad y \in \ell^2\,.
\end{equation}
Now we identify ${}_{\rV^\prime}\langle G(\cdot), \psi \rangle_{\rV}$ with the mapping $\psi^{\ast \ast}G \colon \rH \to (\ell^2)^\prime$ defined by
\[({\psi }^{\ast \ast }G(u))y := (G(u)y)\psi\,\, \in\, \R \,, \qquad u \in \rH\,, \quad y \in \ell^2\,. \]
Thus, from the inequality \eqref{eq:G_estimate}, we infer that
\begin{equation} \label{E:psi**_est}
  \|{\psi }^{\ast \ast }G(u){\|}_{\ell^2} \le C(R) \|\psi\|_\rV {|u|}_{{\rH}_{{\ocal }_{R}}}^{}.
\end{equation}
Therefore, if we fix $\psi \in \vcal $ then, there exists ${R}_{0}>0$
such that $\mbox{supp}\, \psi $ is a compact subset of ${\ocal }_{{R}_{0}}$.
Since $G$ is linear,  estimate (\ref{E:psi**_est}) with $R:={R}_{0}$ yields that the mapping
\[
     {L}^{2}_{loc}(\R^3 , \R^3 ) \supset \rH \ni u \mapsto {\psi }^{\ast \ast }G(u) \in \ell^2
\]
is continuous in the Fr\'{e}chet topology inherited on the space $\rH$ from the space ${L}^{2}_{loc}(\R^3 , \R^3 )$, concluding the proof of the lemma.
\end{proof}

\begin{lemma}[Convergence of quadratic variations]
\label{lemma6.24} For any $s, t \in [0, T]$ and $\psi, \zeta \in \rm{V}_\gamma$, we have
\begin{align*}
\lim_{n \rightarrow \infty} & \widetilde{\Eb}\left[ \left( \int_s^t \langle  \left(G(\sigma, \tunsi) \right)^\ast P_n \psi , \left(G(\sigma,\tunsi) \right)^\ast P_n \zeta \rangle\,d \sigma \right) \cdot h(\widetilde{u}_{n|[0,s]})\right]\\
& = \widetilde{\Eb} \left[ \left( \int_s^t \langle  \left(G(\sigma, \tusi) \right)^\ast\,\psi , \left(G (\sigma, \tusi\right)^\ast\, \zeta \rangle \,d \sigma \right) \cdot h(\widetilde{u}_{|[0,s]})\right],
\end{align*}
where $\langle \cdot, \cdot \rangle$ is the inner product in $\ell^2$.
\end{lemma}

\begin{proof}
Let us fix $\psi, \zeta \in \rV_\gamma$ and let us denote for $\omega \in \widetilde{\Omega}$
\[ f_n(\omega) := \left(\int_s^t \langle \left(G(\sigma,\tunsi)\right)^\ast P_n \psi,\left (G(\sigma,\tunsi)\right)^\ast P_n \zeta  \rangle_{\ell^2}\,d \sigma \right) \cdot h(\widetilde{u}_{n|[0,s]})\,.\]
We will prove that the functions are uniformly integrable and convergent $\widetilde{\mathbb{P}}$-a.s. We start by proving that for some $r > 1$,
\begin{equation}
\label{eq:6.89}
\sup_{n \in \N} \widetilde{\Eb}[|f_n|^r] < \infty\,.
\end{equation}
Since, $\mathcal{L}_2(\ell^2;H)$ is continuously embbeded in $\mathcal{L}(\ell^2; H)$, then by \eqref{eq:2.15} there exists some $c > 0$ such that
\begin{align*}
|\left(G(\sigma,\widetilde{u}_n(\sigma, \omega))\right)^\ast P_n \psi|_{\ell^2} &\leq \|G(\sigma,\widetilde{u}_n(\sigma, \omega))\|_{\lcal(\ell^2; \rH)}|P_n \psi|_\rH \\
& \le \frac{c}{2} |\nabla \widetilde{u}_n(\sigma, \omega)|_{L^2}|\psi|_{\rH}\,,
\end{align*}
and thus
\begin{align*}
&\E\,|f_n|^r = \E\,\left| \left( \int_s^t \langle \left(G(\sigma,\tunsi)\right)^\ast P_n \psi,\left (G(\sigma,\tunsi)\right)^\ast P_n \zeta \rangle_{\ell^2} d \sigma \right) \cdot h(\widetilde{u}_{n|[0,s]})\right|^r \\
&\, \leq \|h \circ \widetilde{u}_n\|^{r}_{{L}^\infty} \E\,\left(\int_s^t |\left(G(\sigma,\widetilde{u}_n(\sigma, \omega))\right)^\ast P_n \psi|_{\ell^2}  |\left(G(\sigma,\widetilde{u}_n(\sigma, \omega))\right)^\ast P_n \zeta|_{\ell^2}  \, d\sigma \right)^r \\
&\, \leq \frac{c^{2r}}{4^r} \|h \circ \widetilde{u}_n\|_{{L}^\infty} ^{r}|\psi|_{\rm{H}}^r |\zeta|_{\rm{H}}^r \E\,\left(\int_s^t \|\widetilde{u}_n(\sigma, \omega)\|_{\rm{V}}^2 \, d \sigma \right)^r\,.
\end{align*}
Using the H\"older inequality, we get
\begin{align*}
& \E\,\left(\int_s^t \|\widetilde{u}_n(\sigma, \omega)\|_\rV^2 d \sigma\right)^r \leq (t-s)^{r-1} \E\, \int_s^t \|\widetilde{u}_n(\sigma, \omega)\|_{\rm{V}}^{2r}\, d \sigma \\
&\; \leq T^r \E \left(\sup_{\sigma \in [0,T]} \|\widetilde{u}_n(\sigma, \omega)\|_{\rm{V}}^{2r}\right).
\end{align*}
Thus
\[\E\,|f_n|^r \leq \widetilde{C} \E\left(\sup_{ \sigma \in [0,T]}\|\widetilde{u}_n(\sigma, \omega)\|_{\rm{V}}^{2r}\right) \]
for some $\widetilde{C} > 0$. Hence by \eqref{eq:6.59} for some $r \in (1,3]$
\[ \sup_{n \in \N} \widetilde{\Eb} |f_n|^r \leq \widetilde{C} \sup_{n \in \N} \widetilde{\Eb}\left[\sup_{\sigma \in [0,T]}\|\widetilde{u}_n(\sigma, \omega)\|_{\rm{V}}^{2r} \right] \leq \widetilde{C} C_1(r) < \infty,\]
inferring \eqref{eq:6.89}.

\noindent \textbf{Pointwise convergence $\colon$} Next, we have to prove the following pointwise convergence for a fix $\omega \in \tom$, i.e. we will show that for a fix $\omega \in \tom$
\begin{align}
\label{eq:6.90}
\lim_{n \rightarrow \infty} &\int_s^t \left\langle  \left(G(\sigma, \tunsi) \right)^\ast P_n \psi , \left(G(\sigma,\tunsi) \right)^\ast P_n \zeta \right\rangle_{\ell^2}d \sigma \nonumber \\
 & = \int_s^t \left\langle  \left(G(\sigma, \tusi) \right)^\ast\,\psi , \left(G (\sigma, \tusi\right)^\ast\, \zeta \right\rangle_{\ell^2} d \sigma\,.
\end{align}
Let us fix $\omega \in \tOmega $ such that
\begin{itemize}
\item[(i)] $\tun (\cdot ,\omega) \to \tu (\cdot ,\omega)$ in ${L}^{2}(0,T, {H}_{loc})$,
\item[(ii)]  $\tu (\cdot ,\omega )\in {L}^{2}(0,T;H)$ and the sequence $(\tun (\cdot ,\omega ){)}_{n\in \N}$ is bounded in $\ccal([0,T]; \rV)$.
\end{itemize}
Notice that, in order to prove \eqref{eq:6.90} it is sufficient to prove that
\begin{equation} \label{E:L^2(s,t;Y)_conv}
  {G(\cdot, \tun (\cdot , \omega  ))}^\ast \Pn \psi  \to {G(\cdot, \tu (\cdot , \omega  ))}^\ast  \psi
 \quad \mbox{in} \quad {L}^{2}(s,t;\ell^2)\,.
\end{equation}
We have
\begin{align}
&\int_{s}^{t}  \bigl\| {G(\sigma, \tun (\sigma , \omega  ))}^\ast \Pn \psi
 - {G(\sigma, \tu (\sigma , \omega ))}^\ast  \psi  {\bigr\| }_{\ell^2}^{2}  \, d\sigma  \nonumber \\
&\le \int_{s}^{t}  \biggl( \bigl\| {G(\sigma, \tun (\sigma , \omega  ))}^\ast ( \Pn \psi  -\psi ) {\bigr\| }_{\ell^2} \nn \\
& \qquad \quad + \bigl\|  {G(\sigma, \tun (\sigma , \omega ))}^\ast \psi - {G(\sigma, \tu (\sigma , \omega  ))}^\ast  \psi  {\bigr\| }_{\ell^2}^{} {\biggr) }^{2}  \, d\sigma  \nonumber \\
&\le 2 \int_{s}^{t}   \bigl\| {G(\sigma, \tun (\sigma , \omega  ))}^\ast {\bigl\|}_{\mathcal{L} (\rH,\ell^2)}^{2}
 \cdot |\Pn \psi  -\psi|_\rH^2 \, d \sigma \nonumber \\
& \quad  +  2 \int_{s}^{t}
\bigl\|  {G(\sigma, \tun (\sigma , \omega  ))}^\ast \psi
- {G(\sigma, \tu (\sigma , \omega ))}^\ast  \psi  {\bigr\| }_{\ell^2}^{2}   \, d\sigma \nn \\
& =: 2 \left\{ {I}_{1}(n) + {I}_{2}(n)\right\}    \label{E:2_(I1+I2)}\,.
\end{align}
Let us consider the term ${I}_{1}(n)$. Since $\psi \in \rV_\gamma $, we have
\[\lim_{n\to \infty } | \Pn \psi - \psi |_\rH = 0\,.\]
By Lemma~\ref{lemma2.1}, the continuity of the embedding $\mathcal{L}_2 (\ell^2, \rH) \hookrightarrow \mathcal{L} (\ell^2, \rH)$ and (ii),
we infer that
\begin{align*}
\int_{s}^{t}
  \bigl\| {G(\sigma,\tun (\sigma , \omega  ))}^\ast {\bigl\|}_{\mathcal{L}(\rH, \ell^2)}^{2} \, d \sigma
  &\le C  \int_{s}^{t}  {|\nabla \tun(\sigma, \omega)|}_{L^2}^{2}\, d \sigma \\
&  \le \widetilde{C} T \sup_{n \in \N} \|\tun (\omega )\|_{\ccal([0,T]; \rV)} \le K
\end{align*}
for some constant $K>0$. Thus
\[
\lim_{n \to \infty } {I}_{1}(n)=
\lim_{n \to \infty }  \int_{s}^{t}   \bigl\| {G(\sigma, \tun (\sigma , \omega  ))}^\ast {\bigl\|}_{\mathcal{L} (\rH,\ell^2)}^{2}
 \cdot |\Pn \psi  -\psi|_\rH^2 \, d \sigma  = 0\,.
\]

Let us move to the  term ${I}_{2}(n)$ in \eqref{E:2_(I1+I2)}. We will prove that for every $\psi \in \rV_\gamma $ the term ${I}_{2}(n)$ tends to zero  as $n \to \infty $.
 Assume first that $\psi \in \mathcal{V} $. Then there exists $R>0$ such that
$\mbox{supp } \psi $ is a compact subset of ${\ocal }_{R}$.
Since $\tun (\cdot ,\omega ) \to \tu (\cdot ,\omega )$ in ${L}^{2}(0,T;{H}_{loc})$, then in particular
\[\lim_{n\to \infty } {q}_{T,R} \bigl( \tun (\cdot ,\omega ) - \tu (\cdot ,\omega ) \bigr) =0\,,\]
where ${q}_{T,R}$ is the seminorm defined by \eqref{E:seminorms}. In other words,
$\tun (\cdot ,\omega ) \to \tu (\cdot ,\omega )$ in ${L}^{2}(0,T;$ ${H}_{{\ocal }_{R}})$. Therefore, there exists a subsequence $(\tunk (\cdot ,\omega ){)}_{k} $ such that
\[ \tunk (\sigma ,\omega ) \to \tu (\sigma ,\omega )
   \quad \mbox{ in } {H}_{{\ocal }_{R}} \mbox{ for almost all } \sigma \in [0,T] \mbox{ as } k\to \infty  \,.\]
Hence by Lemma~\ref{lemma6.23} as $k \to \infty$
\[  G \bigl(\sigma, \tunk (\sigma ,\omega ){\bigr) }^\ast \psi \to G \bigl(\sigma, \tu (\sigma ,\omega ) {\bigr) }^\ast \psi
   \mbox{ in } \ell^2 \mbox{ for almost all } \sigma \in [0,T]\,. \]
In conclusion, by the Vitali's convergence theorem
\[      \lim_{k\to \infty } \int_{s}^{t} \|G \bigl( \tunk (\sigma ,\omega ){\bigr) }^\ast \psi - G \bigl( \tu (\sigma ,\omega ) {\bigr) }^\ast \psi\|_{\ell^2}^{2} \, d\sigma =0 \qquad \mbox{for } \psi \in \mathcal{V} \,. \]
Repeating the above reasoning for all subsequences, we infer that from every subsequence of the sequence $\bigl( G \bigl( \sigma, \tun (\sigma ,\omega ){\bigr) }^\ast \psi {\bigr) }_{n}$ we can choose the subsequence convergent in ${L}^{2}(s,t; \ell^2)$ to the same limit. Thus, the whole sequence
$\bigl( G \bigl( \sigma, \tun (\sigma ,\omega ){\bigr) }^\ast \psi {\bigr) }_{n}$ is convergent to $G \bigl(\sigma, \tu (\sigma ,\omega ){\bigr) }^\ast \psi  $ in ${L}^{2}(s,t; \ell^2)$. At the same time
\[ \lim_{n\to\infty } {I}_{2}(n) =0 \qquad \mbox{for every } \psi \in \mathcal{V}\,.\]
If $\psi \in \rV_\gamma$ then for every $\varepsilon > 0$ we can find ${\psi }_{\varepsilon } \in \mathcal{V} $ such that $\|\psi - {\psi }_{\varepsilon }\|_{V_\gamma} < \varepsilon $.
By the continuity of embeddings $\mathcal{L}_2 (\ell^2, \rH) \hookrightarrow \mathcal{L} (\ell^2, \rH) \hookrightarrow \mathcal{L}(\ell^2, \rV_\gamma^\prime)$, Lemma~\ref{lemma2.1} and (ii), we obtain
\begin{align*}
&\int_{s}^{t} \bigl\|  {G(\sigma, \tun (\sigma , \omega ))}^\ast \psi
- {G(\sigma, \tu (\sigma ,\omega ))}^\ast  \psi  {\bigr\| }_{\ell^2}^{2}   \, d\sigma  \\
&\le 2  \int_{s}^{t} \!  \bigl\| [ {G(\sigma, \tun (\sigma , \omega ))}^\ast
- {G(\sigma, \tu (\sigma ,\omega ))}^\ast ] ( \psi - {\psi }_{\varepsilon } ) {\bigr\| }_{\ell^2}^{2}   \,d\sigma \\
&\qquad  + 2  \int_{s}^{t}  \bigl\| [ {G(\sigma, \tun (\sigma , \omega ))}^\ast - {G(\tu (\sigma ,\omega ))}^\ast ]  {\psi }_{\varepsilon }  {\bigr\| }_{\ell^2}^{2} \,d\sigma \\
&\le 4 \int_{s}^{t}    \left[\|G(\sigma, \tun (\sigma , \omega ))\|_{\mathcal{L} (\ell^2, \rV^{\prime }_\gamma)}^{2}
+ \|G(\sigma, \tu (\sigma , \omega ))\|_{\mathcal{L} (\ell^2, \rV^{\prime }_\gamma)}^{2} \right] \|\psi - {\psi }_{\varepsilon } \|_{\rV_\gamma}^{2}\,d\sigma \\
&\qquad + 2\int_{s}^{t}   \bigl\| [ {G(\sigma, \tun (\sigma , \omega ))}^\ast
- {G(\sigma, \tu (\sigma ,\omega ))}^\ast ]  {\psi }_{\varepsilon }  {\bigr\| }_{\ell^2}^{2}   \, d\sigma \\
& \le c\, T \bigl( \|\tun (\cdot , \omega )\|_{\ccal(0,T; \rV)}^{2}
  + \|\tu (\cdot , \omega )\|_{\ccal(0,T; \rV)}^{2}\bigr) \cdot {\varepsilon }^{2} \\
  & \qquad + 2\int_{s}^{t}   \bigl\| [ {G(\sigma, \tun (\sigma , \omega ))}^\ast
- {G(\sigma, \tu (\sigma ,\omega ))}^\ast ]  {\psi }_{\eps }  {\bigr\| }_{\ell^2}^{2}   \, d\sigma  \\
&\le C  {\varepsilon }^{2} + 2\int_{s}^{t}   \bigl\| [ {G(\sigma, \tun (\sigma , \omega ))}^\ast
- {G(\sigma, \tu (\sigma ,\omega ))}^\ast ]  {\psi }_{\varepsilon }  {\bigr\| }_{\ell^2}^{2}   \, d\sigma ,
\end{align*}
for some positive constants $c$ and $C$. Passing to the upper limit as $n \to \infty $, we infer that
\[
 \limsup_{n \to \infty } \int_{s}^{t}   \bigl\|  {G(\sigma, \tun (\sigma , \omega ))}^\ast \psi
- {G(\sigma, \tu (\sigma ,\omega ))}^\ast  \psi  {\bigr\| }_{\ell^2}^{2}   \, d\sigma
\le C  {\varepsilon }^{2} .
\]
In conclusion, we proved that
\[
 \lim_{n \to \infty } \int_{s}^{t}   \bigl\|  {G(\sigma, \tun (\sigma , \omega ))}^\ast \psi
- {G(\sigma, \tu (\sigma ,\omega ))}^\ast  \psi  {\bigr\| }_{\ell^2}^{2}   \, d\sigma   = 0
\]
which completes the proof of (\ref{E:L^2(s,t;Y)_conv}). Thus, by \eqref{eq:6.89}, convergence \eqref{eq:6.90} and the Vitali's convergence theorem, we conclude the proof of Lemma~\ref{lemma6.24}.
\end{proof}

By Lemma \ref{lemma6.21} we can pass to the limit in \eqref{eq:6.65}. By Lemmas \ref{lemma6.22} and \ref{lemma6.24} we can pass to the limit in \eqref{eq:6.66} as well. After passing to the limits we infer that for all $\psi, \zeta \in \rm{V}_\gamma$ and all bounded continuous functions $h$ on $\ccal([0,T]; \rV_w)$:
\begin{equation}
\label{eq:6.91}
\widetilde{\Eb} \left[ \langle \widetilde{M}(t) - \widetilde{M}(s) , \psi \rangle h(\widetilde{u}_{|[0,s]}) \right] = 0\,,
\end{equation}
and
\begin{align}
\label{eq:6.92}
\widetilde{\Eb} \Big[ &\Big( \langle \widetilde{M}(t), \psi \rangle \langle \widetilde{M}(t), \zeta \rangle - \langle \widetilde{M}(s), \psi \rangle \langle \widetilde{M}(s), \zeta \rangle \nonumber \\
  & - \int_s^t \langle \left(G(r, \widetilde{u}(r))\right)^\ast\psi , \left(G(r, \widetilde{u}(r))\right)^\ast \zeta \rangle_{\ell^2}\,dr \Big) \cdot h(\widetilde{u}_{|[0,s]}) \Big] = 0\,,
\end{align}
where $\langle \cdot, \cdot \rangle$ is the dual pairing between $\rV_\gamma^\prime$ and $\rV_\gamma$.

From Lemma~\ref{lemma6.17}, Lemma~\ref{lemma6.22} and Lemma~\ref{lemma6.24}, we infer the following corollary.
\begin{corollary}
\label{cor_quadratic_6}
For $t \in [0,T]$
\[\langle \langle \widetilde{M}\rangle \rangle_t = \int_0^t \|G(s,\tus)\|_{\mathcal{L}_2(\ell^2; \rH)}\,ds\,,\quad \quad t \in [0,T]\,.\]
\end{corollary}

\noindent \textbf{Continuation of the proof of Theorem~\ref{thm6.16}.}
Now we apply the idea analogous to that used by Da Prato and Zabczyk, see \cite[Section~8.3]{[DZ92]}. By Lemma~\ref{lemma_cont_martingale}, and Corollary~\ref{cor_martingale_6}, we infer that $\widetilde{M}(t), t \in [0, T]$ is an $\rH$-valued continuous square integrable martingale with respect to the filtration $\widetilde{\mathbb{F}} = (\widetilde{\mathcal{F}_t})$. Moreover, by Corollary~\ref{cor_quadratic_6} the quadratic variation of $\widetilde{M}$ is given by
\begin{equation}
\label{eq:6.93}
\langle \langle \widetilde{M}\rangle \rangle_t = \int_0^t \|G (s, \tus)\|_{\mathcal{L}_2(\ell^2;\rH)}\,ds\,.
\end{equation}
Therefore by the martingale representation theorem, there exist
\begin{itemize}
\item a stochastic basis $(\widetilde{\widetilde{\Omega}}, \widetilde{\widetilde{\mathcal{F}}}, \widetilde{\widetilde{\mathbb{F}}}, \widetilde{\widetilde{\mathbb{P}}})$,
\item a cylindrical Wiener process $\widetilde{\widetilde{W}}(t)$,
\item and a progressively measurable process $\widetilde{\tu}(t)$ such that for all $t \in [0,T]$ and all $\v \in \mathcal{V} $
\begin{align*}
& \langle \widetilde{\tu}(t) , \v \rangle  - \langle \widetilde{\widetilde{u}}(0) , \v \rangle + \int_0^t \langle \A  \widetilde{\tu}(s), \v \rangle\,ds + \int_0^t \langle B(\widetilde{\tu}(s), \widetilde{\tu}(s)), \v \rangle \, ds \\
& \quad=   \int_0^t \langle f(\widetilde{\tu}(s)) - g(|\widetilde{\tu}(s)|^2) \widetilde{\tu}(s), \v \rangle \,ds + \left \langle \int_0^t G(s, \widetilde{\tu}(s))\,d\widetilde{\widetilde{W}}(s), \v \right \rangle.
\end{align*}
\end{itemize}
\noindent Thus, the conditions from Definition~\ref{defn3.9} hold with $(\widehat{\Omega}, \widehat{\mathcal{F}}, \widehat{\mathbb{F}}, \widehat{\mathbb{P}}) = (\widetilde{\widetilde{\Omega}}, \widetilde{\widetilde{\mathcal{F}}}, \widetilde{\widetilde{\mathbb{F}}}, \widetilde{\widetilde{\mathbb{P}}})$, $\widehat{W} = \widetilde{\widetilde{W}}$ and $\widehat{u} = \widetilde{\widetilde{u}}$.
The proof of Theorem~\ref{thm6.16} is thus complete.

\subsection{Uniqueness and strong solutions}
\label{s:6.4}

In this subsection we will show that the solutions of \eqref{eq:2.13} are pathwise unique and that the martingale solution of \eqref{eq:2.13} is the strong solution. Let us recall the definition of pathwise unique solutions.

\begin{definition}
\label{defn6.25}
Let $(\Omega, \mathcal{F}, \mathbb{F}, \mathbb{P}, W, u^i), i = 1,2$ be the martingale solutions of \eqref{eq:2.13} with $u^i(0) = u_0, i= 1,2$. Then we say that the solutions are {\bf pathwise unique} if $\mathbb{P}$-a.s. for all $t \in [0,T],$ $u^1(t) = u^2(t)$.
\end{definition}

\begin{theorem}
\label{thm6.26}
Assume that the assumptions $({\bf H1})$ and $({\bf H2})$ are satisfied. If $u_1, u_2$ are two solutions of \eqref{eq:2.13} defined on the same filtered probability space $(\widehat{\Omega}, \widehat{\mathcal{F}}, \widehat{\mathbb{F}}, \hp)$ then $\hp$-a.s. for all $t \in [0,T]$, $u_1(t) = u_2(t)$.
\end{theorem}

This theorem has been proved in \cite[Theorem~3.7]{[RZ09]}.

\begin{theorem}
\label{thm6.27}
Assume that assumptions $(\mathbf{H1})$ and $(\mathbf{H2})$ are satisfied. Then there exists a path-wise unique strong solution $u \in \ccal([0,T]; \rV) \cap L^2(0,T; \rm{D}(\A ))$ of \eqref{eq:2.13} such that
\[ \sup_{t \in [0,T]} \|u(t)\|^2_\rV + \int_0^T |u(t)|^2_{\rm{D}(\A )}\,dt < \infty\,. \]
\end{theorem}

\begin{proof}
Since by Theorem~\ref{thm6.16} there exists a martingale solution and by Theorem~\ref{thm6.26} it is pathwise unique, the existence of strong solution follows from \cite[Theorem~2]{[Ondrejat04]}. Moreover, it can be shown that $u \in \ccal([0,T];\rV) \cap L^2(0,T; \rD(\A))$ by following the proof of \cite[Lemma~3.6]{[BD16]}.
\end{proof}

\section{Invariant measures}
\label{s:5}

In this section, we consider time homogeneous damped tamed NSEs, i.e. the coefficients $f, \sigma$ are independent of $t$ and furthermore $f \in \rH$ is not dependent on $u$. The time homogeneous damped tamed NSEs in abstract form are given by
\begin{equation}
\label{eq:5.1}
\begin{split}
&du(t) = \left[- \A _{\alpha}u(t) - B(u(t)) - \Pi[g(|u(t)|^2)u(t)] + \Pi f \right]dt \\
& \hspace{1.5truecm} + \sum_{j=1}^\infty G_j(u(t))\,dW_t^j\,,\\
&u(0) = u_0 \in \rV ,
\end{split}
\end{equation}
where $\A _\alpha = \alpha I - \nu \Delta$ for some $\alpha \in \R$ and $\nu > 0$ is the viscosity. The operator $B$ and the cylindrical Wiener process $W = \left(W_j\right)_{j=1}^\infty$ on $\ell^2$ is same as defined in Section~\ref{s:2} and $G_j$ are as defined in \eqref{eq:2.11}.

Let $ \mathcal{B}_b(\rV)$ denote the set of all bounded and Borel measurable functions on $\rV$. For any $\varphi \in \bcal_b(\rV)$, $t \ge 0$, we define a function $T_t \varphi \colon \rV \to \R$ by
\begin{equation}
\label{eq:5.2}
T_t \varphi(\v) := \E \left(\varphi(u(t;\v))\right), \quad \v \in \rV\,.
\end{equation}
It follows from Theorem~\ref{thm_damped_exist} and Ondrejat \cite{[Ondrejat05]} (see also \cite{[BF17]}) that $T_t \varphi \in \bcal_b(\rV)$ and $\left\{T_t\right\}_{t \ge 0}$ is a semigroup on $\bcal_b(\rV)$. Also since this unique solution to \eqref{eq:5.1} has $a.e.$ path in $\ccal([0,T]; \rV)$, it is also a Markov semigroup (see \cite[Theorem~27]{[Ondrejat05]}). Moreover, $\left\{T_t \right\}_{t \ge 0}$ is a Feller semigroup, i.e. $T_t$ maps $C_b(\rV)$ into itself.

Next we state the main result of this section, regarding the existence of an invariant measure.

\begin{theorem}
\label{thm5.4}
Let for every $\alpha > 0$, the assumptions $(\mathbf{H1})^\prime - (\mathbf{H3})^\prime$ be satisfied. Then there exists an invariant measure $\mu \in \mathcal{P}(\rV)$ of the semigroup $(T_t)_{t \ge 0}$ defined by \eqref{eq:5.2}, such that for any $t \ge 0$ and $\varphi \in {\rm SC}_{b}(\rV_w)$
\[\int_\rV T_t \varphi(u) \mu(du) = \int_\rV \varphi(u) \mu (du).\]
\end{theorem}

If $T_t$ is sequentially weakly Feller Markov semigroup then for every $\varphi \in {\rm SC}_b(\rV_w)$, $T_t \varphi \in {\rm SC}_b(\rV_w) \subset B_b(\rV)$ (see \cite{[BF17],[MS99]} for the definitions and inclusion of the spaces); therefore the integral on the l.h.s. in Theorem~\ref{thm5.4} makes sense.

Next we list the assumptions that we make on the coefficients $f$ and $\sigma$ along with a coercivity type assumption, see \cite{[Pardoux79]}.

\begin{trivlist}
\item{(\textbf{H1})$^\prime$} The function $f:\R^3 \to \R^3$ is time independent and $\rH$-valued.
\item{(\textbf{H2})$^\prime$} A  measurable  function $\sigma :\R^3 \to \ell^2 $ of $C^1$ class with respect to the $x$-variable and for all $x \in \R^3$ there exists a constant $C_\sigma > 0$ such that
\[\|\partial_{x^j} \sigma(x)\|_{\ell^2} \le C_{\sigma}, \quad j = 1,2,3\]
and, for all $x \in \R^3$,
\begin{equation*}
\|\sigma(x)\|^2_{\ell^2} \le \frac{1}{4}\,.
\end{equation*}
\item{(\textbf{H3})$^\prime$} there exists a $\delta > 0$ such that
\[ 2 \nu |\nabla u |^2_{L^2} - \|G(u)\|^2_{\mathcal{L}_2(\ell^2; \rH)} \ge 2\delta |\nabla u|^2_{L^2}, \quad u \in \rV\,.\]
\end{trivlist}

The following theorem regarding the existence of a pathwise unique strong solution to the time homogeneous damped tamed NSEs \eqref{eq:5.1} can be proved by modifying the proofs of Theorem~\ref{thm6.16} and Theorem~\ref{thm6.26} to incorporate the extra linear term $\alpha\,u$.

\begin{theorem}
\label{thm_damped_exist}
Assume that assumptions $(\mathbf{H1})^\prime$ and $(\mathbf{H2})^\prime$ are satisfied. Then for every $u_0 \in \rV$, there exists a path-wise unique strong solution $u$ of \eqref{eq:5.1} for every $T > 0$ such that $u \in \ccal([0,T]; \rV) \cap L^2(0,T; \mathrm{D}(\A))$, $\mathbb{P}$-a.s.
\end{theorem}

For fixed initial value $u_0 = \v \in \rV$ we denote the (path-wise) unique solution of \eqref{eq:5.1}, whose existence is proved in Theorem~\ref{thm_damped_exist} by $u(t;\v)$.

\begin{definition}
\label{defn5.2}
We say that a family $\{T_t\}_{t \ge 0}$ is sequentially weakly Feller iff
\[T_t : {\rm SC}_b(\rV_w) \to {\rm SC}_b(\rV_w), \quad t \ge 0.\]
\end{definition}

For a metric space $\mathbb{U}$, we use $\mathcal{P}(\mathbb{U})$ to denote the family of all Borel probability measures on $\mathbb{U}$. We will use the following theorem from Maslowski-Seidler \cite{[MS99]} to prove the existence of invariant measures.

\begin{theorem}
\label{thm5.3}
Assume that
\begin{itemize}
\item[(i)] the semigroup $\{T_t\}_{t \ge 0}$, defined by \eqref{eq:5.2} is sequentially weakly Feller in $\rV$,
\item[(ii)] for any $\varepsilon > 0$ there exists $R > 0$ such that
\[ \sup_{ T \ge 1} \frac{1}{T} \int_0^T \mathbb{P}(\{\|u(t;u_0)\|_\rV > R\})\,dt < \varepsilon\,.\]
\end{itemize}
Then there exists at least one invariant measure for \eqref{eq:5.1}.
\end{theorem}

\subsection{Boundedness in probability}
\label{s:5.1}
\begin{lemma}
\label{lemma5.5}
Let $u_0 \in \rV$. Then, under the assumptions of Theorem~\ref{thm5.4}, for every $\varepsilon > 0$, there exists $R > 0$ such that
\begin{equation}
\label{eq:5.3}
 \sup_{ T \ge 1} \frac{1}{T} \int_0^T \mathbb{P}\left(\left\{\|u(t;u_0)\|_\rV > R\right\}\right)\,dt < \varepsilon\,.
\end{equation}
\end{lemma}

\begin{proof}
Using the It\^o formula for the function $\phi(\xi) = |\xi|^2_\rH$ and the process $u(t)$, we have
\begin{align}
\label{eq:5.4}
\frac12 |u(t)|_\rH^2 & = \frac12 |u_0|^2_\rH +  \int_0^t \langle -\A _\alpha u(s) - B(u(s)) - \Pi (g(|u(s)|^2)u(s)) , u(s) \rangle_\rH ds \nonumber \\
&~~ + \int_0^t \langle \Pi f , u(s)\rangle_\rH ds +  \int_0^t \sum_{j=1}^\infty \langle G_j(u(s))dW_s^j, u(s) \rangle_\rH \nn \\
&~~ + \frac12 \int_0^t \|G(u(s))\|^2_{\mathcal{L}_2(\ell^2;\rH)}\,ds\,.
\end{align}
Now we deal with each term on the r.h.s. of \eqref{eq:5.4} one by one. Firstly let us notice that we have
\begin{equation}
\label{eq:5.5}
\langle \A _\alpha u, u \rangle_\rH = \alpha |u|^2_{\rH} + \nu |\nabla u|^2_{L^2}\,,
\end{equation}
\begin{equation}
\label{eq:5.6}
\langle B(u), u \rangle_\rH  = 0\,,
\end{equation}
\begin{equation}
\label{eq:5.7}
\langle \Pi\left(g(|u|^2) u\right) , u \rangle_\rH = \big| |\sqrt{g(|u|^2)}|\cdot|u| \big|^2_{L^2}\,.
\end{equation}
Using the assumptions on $f$, for any $\beta > 0$ we obtain the following estimate
\begin{equation}
\label{eq:5.8}
\langle \Pi\,f, u \rangle_\rH \le \|f\|_{\rH}\|u\|_\rH \le \frac{1}{4 \beta}|f|^2_{\rH} + \beta \,|u|^2_\rH\,.
\end{equation}
Since $u$ is the solution of \eqref{eq:5.1} and satisfies the estimates \eqref{eq:6.25} (courtsey, Theorem~\ref{thm_damped_exist}), we can show that the process
\[M(t) = \int_0^t \langle u(s), \sum_{j=1}^\infty G_j(u(s))\,dW_s^j\rangle_\rH,\]
is a martingale. Thus, taking expectation in \eqref{eq:5.4} and using the estimates \eqref{eq:5.5} -- \eqref{eq:5.8}, we infer
\begin{align*}
\frac12 \E\, |u(t)|^2_\rH & \le  \frac12 |u_0|^2_\rH - \alpha\, \E \int_0^t |u(s)|^2_{\rm{H}}\,ds - \nu\, \E \int_0^t |\nabla u(s)|^2_{L^2}\,ds  \nonumber \\
&~~ -  \E \int_0^t \big||\sqrt{g(|u(s)|^2)}|\cdot|u(s)| \big|^2_{L^2}\,ds + \frac{1}{4 \beta} \E \int_0^t |f|^2_{\rH}\,ds \\
& ~~ + \beta\, \E \int_0^t |u(s)|^2_{\rH}\,ds + \frac12 \int_0^t \|G(u(s))\|^2_{\mathcal{L}_2(\ell^2;\rH)}\,ds\,.
\end{align*}
On rearranging, we get
\begin{align}
\label{eq:5.9}
& \frac12 \E \,|u(t)|^2_\rH + \frac12 \E \int_0^t \left(2 \nu |\nabla u(s)|^2_{L^2} - \|G(u(s))\|^2_{\mathcal{L}_2(\ell^2;\rH)}\right)\,ds \nn \\
&\quad + \E \int_0^t \big||\sqrt{g(|u(s)|^2)}|\cdot|u(s)| \big|^2_{L^2}\,ds \nonumber \\
&\le \frac12 |u_0|^2_\rH  + \frac{1}{4 \beta} T |f|^2_{\rH} 
+ (\beta - \alpha) \E \int_0^t |u(s)|^2_{\rH}\,ds\,.
\end{align}
Now using the assumption $(\textbf{H3})^\prime$ in \eqref{eq:5.9}, we obtain
\begin{align*}
&\frac12 \E \,|u(t)|^2_\rH + \delta\E \int_0^t |\nabla u(s)|^2_{L^2}\,ds +  (\alpha - \beta)\, \E \int_0^t | u(s)|^2_{\rH}\,ds \nn \\
 &\; \le \frac12 |u_0|^2_\rH  + \frac{1}{4 \beta} T |f|^2_{\rH}\,.
\end{align*}
Choosing $\beta = \frac{\alpha}{2}$ yields
\begin{align*}
&\frac12 \E \,|u(t)|^2_\rH + \delta \E \int_0^t |\nabla u(s)|^2_{L^2}\,ds +  \frac{\alpha}{2} \E \int_0^t | u(s)|^2_{\rH}\,ds  \\
& \; \le \frac12 |u_0|^2_\rH  + \frac{1}{2 \alpha} T |f|^2_{\rH} .
\end{align*}
Therefore for $\gamma = \min{\{\frac{\alpha}{2}, \delta\}} > 0$,
\[\frac12 \E \,|u(t)|^2_\rH  + \gamma \E \int_0^t \|u(s)\|^2_{\rV}\,ds \le \frac12 |u_0|^2_\rH  + \frac{1}{2 \alpha} T |f|^2_{\rH} .\]
Thus, for any $T > 0$, we infer that
\begin{equation}
\label{eq:5.10}
 \frac{1}{T} \int_0^T \E\,\|u(s)\|^2_\rV\,ds \le \frac{1}{2 \gamma T} |u_0|^2_\rH + \frac{1}{4 \gamma^2} |f|^2_{\rH}.
 \end{equation}
Using the Chebyshev inequality and inequality \eqref{eq:5.10}, we infer that for every $T \ge 0$
\begin{align*}
\frac{1}{T} \int_0^T \mathbb{P}(\{\|u(t,u_0)\|_\rV > R\})\,dt & \le \frac{1}{T R^2} \int_0^T \E\,\|u(t)\|^2_\rV\,dt \\
& \le \frac{1}{R^2} \left[ \frac{1}{2 \gamma T} |u_0|^2_\rH + \frac{1}{4 \gamma^2} |f|^2_{\rH}\right].
\end{align*}
Now for sufficiently large $R > 0$ depending on $\varepsilon, |u_0|_\rH$ and $|f|_{\rH}$, the assertion follows.

\end{proof}

\subsection{Sequentially weak Feller property}
\label{s:5.2}
We are left to verify the assumption $(i)$ of Theorem~\ref{thm5.3}, i.e. the Markov semigroup $\{T_t\}_{t \ge 0}$ is sequentially weakly Feller in $\rV$. In other words we want to show that for any $t > 0$ and any bounded and weakly continuous $\varphi : \rV \to \R$, if $\xi_n \to \xi$ weakly in $\rV$, then
\begin{equation}
\label{eq:6.6.11}
T_t \varphi(\xi_n) \to T_t\varphi(\xi)\;.
\end{equation}

The second named author in his PhD thesis proved that the martingale solutions of stochastic constrained Navier-Stokes equations continuously depend on the initial data \cite[Theorem~5.7.7]{[Dhariwal17]}. We have a similar result for time homogeneous damped tamed NSEs, which can be proved analogously, see also \cite[Theorem~4.11]{[BMO17]}.

\begin{theorem}
\label{thm5.6}
Assume that $(u_{0,n})_{n=1}^\infty$ is a $\rV$-valued sequence that is convergent weakly to $u_0 \in \rV$. Let
\[ \left(\Omega_n, \mathcal{F}_n, \mathbb{F}_n, \mathbb{P}_n, W_n, u_n\right)\]
be a martingale solution of problem \eqref{eq:5.1} on $[0, \infty)$ with the initial data ${u}_{0,n}$. Then for every $T > 0$ there exist
\begin{itemize}
\item a subsequence $({n}_{k}{)}_{k}$,
\item a stochastic basis $\bigl( \tOmega , \tfcal ,  {\widetilde{\mathbb{F}}} , \tp  \bigr) $,
\item an $\ell^2$-valued cylindrical Wiener process $\widetilde{W}(t) = \left(\widetilde{W}^j(t)\right)_{j=1}^\infty$,
\item and $\widetilde{\mathbb{F}}$-progressively measurable processes $\tu$,  $\big(\tunk\big)_{k \ge 1} $ (defined on this basis) with  laws supported in $ \mathcal{Z}_{T}$ such that
\begin{equation}   \label{eq:5.11}
  \tunk \mbox{ has the same law as } \unk \mbox{ on } \mathcal{Z}_{T}
    \mbox{ and } \tunk \to \tu \mbox{ in } \mathcal{Z}_{T},
    \quad  \tp\mbox{-a.s.}
\end{equation}
and  the  system
\[\bigl( \tOmega , \tfcal ,  {\widetilde{\mathbb{F}}}, \tp , \tW,\tu  \bigr)\]
is a martingale solution to problem \eqref{eq:5.1} on the interval $[0,T]$ with the initial data $u_0$. In particular, for all $t \in [0,T]$ and all $\mathrm{v} \in \mathcal{V} $
\begin{align*}
 &\ilsk{\tu(t)}{\v}{} +  \int_{0}^{t} \dual{\A _\alpha \tu(s)}{\v}{} \, ds
+ \int_{0}^{t} \dual{B(\tu(s))}{\v}{} \, ds   \nonumber \\
& \quad + \int_{0}^{t} \dual{g(|\tu(s)|^2) \tu(s)}{\v}{} \, ds \nn \\
  &  = \ilsk{\tu(0)}{\v}{\rV} + \int_{0}^{t} \dual{f}{\v}{} \, ds
 + \Dual{\int_{0}^{t} \sum_{j=1}^\infty G_j(s,\tu(s))\,dW^j(s)}{\v}{},  \quad \widehat{\p }\mbox{-a.s.}
\end{align*}
Moreover,  the process $\tu $ satisfies  the following inequality
\begin{equation}  \label{eq:5.12}
  \widetilde{\mathbb{E}}\,\left[ \,\sup_{ s\in [0,T]  } {\|\tu (s)\|_\rV }^{2}\, + \int_{0}^{T} |\tu(s)|^2_{\rm{D}(\A )}\, ds \,\right] < \infty\,.
\end{equation}
\end{itemize}
\end{theorem}

We will need the uniqueness in law of solutions of \eqref{eq:5.1}, which is defined next.

\begin{definition}
\label{defn5.7}
Let $(\Omega^i, \mathcal{F}^i, \mathbb{F}^i, \mathbb{P}^i, W^i, u^i), i = 1,2$ be martingale solutions of \eqref{eq:5.1} with $u^i(0) = u_0, i=1,2$. Then we say that the solutions are {\bf unique in law} if
\[law_{\mathbb{P}^1}(u^1) = law_{\mathbb{P}^2}(u^2)\,\mbox{ on }\, \ccal([0, \infty); \rV_w) \cap L^2([0, \infty); \rD(\A )),\]
where $law_{\mathbb{P}^i}(u^i), i = 1,2$ are by definition probability measures on \newline $\ccal([0, \infty); \rV_w) \cap L^2([0, \infty); \rD(\A ))$.
\end{definition}

\begin{lemma}
\label{lemma5.8}
Assume that assumptions $(\mathbf{H1})^\prime - (\mathbf{H3})^\prime$ are satisfied. Then the martingale solution of \eqref{eq:5.1} is unique in law.
\end{lemma}

The proof of the above lemma is the direct application of Theorems~2 and 11 of \cite{[Ondrejat04]} once we have proved the existence of a pathwise unique martingale solution of \eqref{eq:5.1}; which follows from Theorem~\ref{thm_damped_exist}.

\begin{lemma}
\label{lemma5.9}
The semigroup $\{T_t\}_{t \ge 0}$ is sequentially weakly Feller in $\rV$.
\end{lemma}

\begin{proof}
Let us choose and fix $ 0 < t \leq T, \xi \in \rV$ and  $\varphi \colon \rV \to \R$ be a bounded weakly continuous function. Need to show that $T_t \varphi$ is sequentially weakly Feller in $\rV$. For this aim let us choose a $\rV$-valued sequence $(\xi_n)$ weakly convergent to $\xi$ in $\rV$. Since the function $T_t \varphi \colon \rV \to \R$ is bounded, we only need to prove \eqref{eq:6.6.11}.

Let $u_n(\cdot) = u(\cdot; \xi_n)$ be a strong solution of \eqref{eq:5.1} on $[0,T]$ with the initial data $\xi_n$ and let $u(\cdot) = u(\cdot; \xi)$ be a strong solution of \eqref{eq:5.1} with the initial data $\xi$. We assume these processes are defined on the stochastic basis $(\Omega, \mathcal{F}, \mathbb{F}, \mathbb{P})$. By Theorem~\ref{thm5.6} there exist
\begin{itemize}
\item a subsequence $(n_k)_k$,
\item a stochastic basis $(\tom, \tf, \widetilde{\mathbb{F}}, \tp)$, where $\widetilde{\mathbb{F}} = \{\tf_s\}_{s \in [0,T]}$,
\item an $\ell^2$-valued cylindrical Wiener process $\widetilde{W}(t)$ on $(\tom, \tf, \widetilde{\mathbb{F}}, \tp)$,
\item and progressively measurable processes $\widetilde{u}(s), (\widetilde{u}_{n_k}(s))_{k \ge 1}$, $s \in [0,T]$ (defined on this basis) with laws supported in $\mathcal{Z}_T$, where
\[\zcal_T = \ccal([0,T]; \mathrm{U}^\prime) \cap L^2_w(0,T; \mathrm{D}(\A)) \cap L^2(0,T; \rH_{loc}) \cap \ccal([0,T]; \mathrm{V}_w),\]
such that
\begin{equation}
\label{eq:5.13}
\tu_{n_k} \mbox{ has the same law as } u_{n_k} \mbox{ on } \mathcal{Z}_T \mbox{ and } \tu_{n_k} \to \tu  \mbox{ in } \mathcal{Z}_T, \; \tp\mbox{-a.s.}
\end{equation}
and the system
\begin{equation}
\label{eq:5.14}
(\tom, \tf, \widetilde{\mathbb{F}}, \tp, \widetilde{W}, \tu)
\end{equation}
is a martingale solution to \eqref{eq:5.1} on the interval $[0,T]$ with the initial data $\xi$.
\end{itemize}
In particular, by \eqref{eq:5.13}, $\tp$-almost surely
\[ \tu_{n_k}(t) \to \tu(t)~\mbox{weakly in}~ \rV\,.\]
Since the function $\varphi : \rV \to \R$ is sequentially weakly continuous, we infer that $\tp$-a.s.,
\[ \varphi(\tu_{n_k}(t)) \to \varphi(\tu(t))~\mbox{in}~\R\,.\]
Since the function $\varphi$ is also bounded, by the Lebesgue dominated convergence theorem we infer that
\begin{equation}
\label{eq:5.15}
\lim_{k \to \infty} \widetilde{\E} \left[ \varphi(\tu_{n_k}(t)) \right] = \widetilde{\E} \left[ \varphi(\tu(t)) \right]\,.
\end{equation}
From the equality of laws of $\tu_{n_k}$ and $u_{n_k}, k \in \mathbb{N}$, on the space $\mathcal{Z}_T$ we infer that $\tu_{n_k}$ and $u_{n_k}$ have same laws on $\rV_w$ and so
\begin{equation}
\label{eq:5.16}
\widetilde{\E} \left[ \varphi(\tu_{n_k}(t)) \right] = \E \left[ \varphi(u_{n_k}(t)) \right] \,.
\end{equation}
On the other hand, the r.h.s. of \eqref{eq:5.16} is equal by \eqref{eq:5.2}, to $T_t\varphi(\xi_{n_k})$.

Since $u$, by assumption, is a martingale solution of \eqref{eq:5.1} with the initial data $\xi$ and from the above, $\tu$ is also a solution of \eqref{eq:5.1} with the initial data $\xi$. Thus, by Lemma~\ref{lemma5.8}, we infer that
\[\mbox{the processes } u \mbox{ and } \tu \mbox{ have same law on the space } \mathcal{Z}_T\,.\]
Hence
\begin{equation}
\label{eq:5.17}
\widetilde{\E} \left[ \varphi(\tu(t)) \right] = \E \left[ \varphi(u(t)) \right] \,.
\end{equation}
As before, the r.h.s. of \eqref{eq:5.17} is equal by \eqref{eq:5.2}, to $T_t\varphi(\xi)$. Thus, by \eqref{eq:5.15}, \eqref{eq:5.16} and \eqref{eq:5.17}, we infer
\[\lim_{k \to \infty} T_t \varphi(\xi_{n_k}) = T_t \varphi(\xi)\,.\]
Using the subsequence argument, we infer that the whole sequence $(T_t \varphi(\xi_n))_{n \in \mathbb{N}}$ is convergent and
\[\lim_{n \to \infty} T_t \varphi(\xi_n) = T_t \varphi(\xi)\,.\]
\end{proof}

\begin{proof}[Proof of Theorem \ref{thm5.4}]
The existence of an invariant measure is established by using Theorem~\ref{thm5.3}, Lemmas~\ref{lemma5.5} and \ref{lemma5.9}.  Hence the proof of Theorem~\ref{thm5.4} is complete.

\end{proof}

\begin{remark}\label{rem-referee} It has been suggested to the authors by the referee, that there should exist an invariant measure for the original non-damped, tamed 3-D Navier-Stokes equations
(\ref{eq:1.1}-\ref{eq:1.3}). We agree with the premises of this suggestion for which we are grateful, yet we have decided to postpone a detailed study of this issue till another publication. Note that the existence of a unique invariant measure for the non-damped tamed 3-D Navier-Stokes equations was established by R\"ockner and Zhang \cite{[RZ09]} in a bounded domain (3-D torus).
\end{remark}

\appendix
\section{Convergence of \texorpdfstring{$P_n$}{P_n}}
\label{s:a}

\begin{lemmaA}
Let $\gamma > \frac{3}{2}$ and $P_n \colon \rH \to \rH_n$ be the orthogonal projection as given by \eqref{eq:6.1} (for  more details see Section~\ref{s:6}). Then as $n \to \infty$
\begin{itemize}
\item[(i)] $P_n \psi \to \psi$ in $\rH$ for $\psi \in \rH$,
\item[(ii)] $P_n \psi \to \psi$ in $\rV$ for $\psi \in \rV$,
\item[(iii)] $P_n \psi \to \psi$ in $\rV_\gamma$ for $\psi \in \rV_\gamma\,.$
\end{itemize}
\end{lemmaA}

\begin{proof}
Let $\psi \in \rH$, then by \eqref{eq:6.1} and Parseval's equality we have
\begin{align*}
&|P_n \psi - \psi|_\rH^2 = \int_{\R^3}|\fcal(P_n\psi)(\xi) - \hat{\psi}(\xi)|^2\,d\xi \\
& \;= \int_{\R^3}|\ind_{B_n}(\xi)\hat{\psi}(\xi) - \hat{\psi}(\xi)|^2\,d\xi = \int_{|\xi| > n}|\hat{\psi}(\xi)|^2\,d\xi\,.
\end{align*}
Now since $\psi \in \rH$ using Lebesgue dominated convergence theorem it can be shown that
\[\lim_{n \to \infty} \int_{|\xi| > n}|\hat{\psi}(\xi)|^2\,d\xi = 0\,,\]
which infers (i).

Let $\psi \in \rV$, then by \eqref{eq:6.1} and the definition of $\rV$-norm we get
\begin{align*}
& \|P_n \psi - \psi\|_\rV^2  = \int_{\R^3}(1+|\xi|^2)\,\left|\fcal(P_n\psi)(\xi) - \hat{\psi}(\xi)\right|^2\,d\xi\\
 &\; = \int_{\R^3} (1+|\xi|^2)\, \left|\ind_{B_n}(\xi)\hat{\psi}(\xi) - \hat{\psi}(\xi)\right|^2\,d\xi  = \int_{|\xi| > n}(1+|\xi|^2)\,|\hat{\psi}(\xi)|^2\,d\xi\,.
\end{align*}
Again using the Lebesgue dominated convergence theorem for $\psi \in \rV$, we get
\[\lim_{n \to \infty} \int_{|\xi| > n}(1+|\xi|^2)\,|\hat{\psi}(\xi)|^2\,d\xi = 0\,,\]
thus proving (ii).

Let $\psi \in \rV_\gamma$, then by \eqref{eq:6.1} and the definition of $\rV_\gamma$-norm we get
\begin{align*}
& \|P_n \psi - \psi\|_{\rV_\gamma}^2 = \int_{\R^3}(1+|\xi|^2)^\gamma\,\left|\fcal(P_n\psi)(\xi) - \hat{\psi}(\xi)\right|^2\,d\xi \\
&\, = \int_{\R^3} (1+|\xi|^2)^\gamma \left|\ind_{B_n}(\xi)\hat{\psi}(\xi) - \hat{\psi}(\xi)\right|^2 d\xi = \int_{|\xi| > n}(1+|\xi|^2)^\gamma |\hat{\psi}(\xi)|^2 d\xi\,.
\end{align*}
Similarly as before it can be shown that for $\psi \in \rV_\gamma$,
\[\lim_{n \to \infty} \int_{|\xi| > n}(1+|\xi|^2)^\gamma\,|\hat{\psi}(\xi)|^2\,d\xi = 0\,,\]
which concludes the proof.
\end{proof}

\section{Kuratowski theorem}
\label{s:b}
The main objective of this appendix is to prove the following theorem (see \cite[Lemma~4.2]{[BMO17]}).

\begin{theoremB}
\label{thmb.1}
Let $T > 0$ and
\[ \mathcal{Z}_T = \mathcal{C}([0,T]; {\rm U}^\prime) \cap L^2(0,T; \rH_{loc}) \cap L^2_w(0,T; {\rm D}(\A )) \cap \ccal([0,T]; \rV_w)\,.\]
Then the sets $\ccal([0,T];\rV) \cap \mathcal{Z}_T$, $\ccal([0,T]; \rH_n) \cap \mathcal{Z}_T$ and $L^2(0,T; \rm{D}(\A )) \cap \mathcal{Z}_T$ are Borel subsets of $\mathcal{Z}_T$.
\end{theoremB}

The proof of the above theorem heavily relies on the following Kuratowski theorem \cite{[Kuratowski52]}.
\begin{theoremB}
\label{thmb.2}
Assume that $X_1, X_2$ are the Polish spaces with their Borel $\sigma$-fields denoted respectively by $\bcal(X_1), \bcal(X_2)$. If $\varphi \colon X_1 \to X_2$ is an injective Borel measurable map then for any $E_1 \in \bcal(X_1)$, $E_2 := \varphi(E_1) \in \bcal(X_2)$.
\end{theoremB}

We will also need following abstract results to prove Theorem~B.\ref{thmb.1}

\begin{lemmaB}
\label{lemmab.3}
Let $X_1, X_2$ and $Z$ be topological spaces such that $X_1$ is a Borel subset of $X_2$. Then $X_1 \cap Z$ is a Borel subset of $X_2 \cap Z$, where $X_2 \cap Z$ is a topological space too, with the topology given by
\begin{equation}\label{eqn-topology on intersection}
\tau(X_2 \cap Z) = \left\{ A \cap B : A \in \tau(X_2), B \in \tau(Z)\right\}\,.
\end{equation}
\end{lemmaB}

\begin{proof}
{Since} the Borel $\sigma$-filed on $X_2 \cap Z$ is the smallest $\sigma$-field generated by $\tau(X_2 \cap Z)$, i.e. $\bcal(X_2 \cap Z) = \sigma( \tau(X_2 \cap Z))$, in order to prove the lemma it is enough to show that $\forall\, Y \in \bcal(X_1)$
\begin{equation}
\label{eq:b.1}
 Y \cap Z \in \bcal(X_2 \cap Z)\,.
\end{equation}
Firstly, we show that \eqref{eq:b.1} holds for all $Y \in \tau(X_1)$. Since $X_1 \in \bcal(X_2)$, $X_1 \subset X_2$ and has trace topology from $X_2$, i.e $\forall \, Y \in \tau(X_1)$ there exists a $C \in \tau(X_2)$ such that
\[Y = C \cap X_1\,.\]
As $X_1 \in \bcal(X_2)$ there exists a countable collection $\left\{K_i\right\}_{i \in \N}$ of open subsets of $X_2$ such that
\[X_1 = \bigcup_{i \in \N} K_i\,.\]
Therefore,
\[Y \cap Z = C \cap X_1 \cap Z = C \cap \left(\bigcup_{i \in \N}K_i\right)\cap Z = \bigcup_{i \in \N} \left(C \cap K_i \right) \cap Z\,.\]
Since $C \in \tau(X_2)$, for every $i \in \N$, $C \cap K_i$ is open in $X_2$ and there exists a collection $\left\{B_j\right\}_{j\in \N} \in \tau(X_2)$ such that
\[\bigcup_{i \in \N} \left(C \cap K_i \right) = \bigcup_{j \in \N} B_j\,.\]
Thus
\[Y \cap Z = \bigcup_{j \in \N} \left(B_j \cap Z\right)\,,\]
and for every $j \in \N$, $B_j \cap Z \in \bcal(X_2 \cap Z)$. Since $\bcal(X_2 \cap Z)$ is a $\sigma$-field, the countable union also belongs to $\bcal(X_2 \cap Z)$, proving \eqref{eq:b.1} for every $Y \in \tau(X_1)$.\\
Secondly, we implement the method of good sets to prove \eqref{eq:b.1} for { a larger class of } subsets of $X_1$. Let
\[ \mathcal{G} = \left\{A \subset X_1 \colon A \cap Z \in \bcal(X_2 \cap Z) \right \}\,.\]
\textbf{Claim}$\colon$ $\mathcal{G}$ is a $\sigma$-field.
\begin{itemize}
\item[i)] $X_1 \in \mathcal{G}$ since $X_1 \subset X_1$ and $X_1 \in \tau(X_1)$ by the definition of topology.

\item[ii)] Let $A \in \mathcal{G}$. We want to show that $A^c := X_1 \setminus A \in \mathcal{G}$, i.e. $A^c \subset X_1$ and $A^c \cap Z \in \bcal(X_2 \cap Z)$. Since $A \in \mathcal{G}$, $A \subset X_1$ and $A \cap Z \in \bcal(X_2 \cap Z)$. Clearly $A^c = X_1\setminus A \subset X_1$. \\
Since $A \cap Z \in \bcal(X_2 \cap Z)$, then by the definition of $\sigma$-field
\[{}^c \left(A \cap Z\right) := \left(X_2 \cap Z\right)\setminus(A \cap Z) \in \bcal(X_2 \cap Z)\,.\]
We have the following set relations
\begin{align*}
{}^c(A \cap Z) &= {}^c A \cup {}^c Z = \left[ (X_2 \cap Z)\setminus A \right] \cup \left[ (X_2 \cap Z) \setminus Z \right] \\
& = \left[(X_2 \setminus A) \cap Z \right] \cup \emptyset = (X_2 \setminus A ) \cap Z \\
& = \left[A^c \cup \left(X_2 \setminus X_1 \right) \right] \cap Z \\
& = \left(A^c \cap Z \right) \cup \left[\left(X_2 \setminus X_1 \right) \cap Z\right] \\
& = \left(A^c \cap Z \right) \cup {}^cX_1\;.
\end{align*}
Now in the above identity ${}^c\left(A \cap Z \right)$, ${}^cX_1$ belongs to $\bcal(X_2 \cap Z)$ and hence $A^c \cap Z \in \bcal(X_2 \cap Z)$, inferring $A^c \in \mathcal{G}$.

\item[iii)] Let $\{A_i\}_{i \in\N} \in \mathcal{G}$. Then $A_i \subset X_1$ for every $i \in \N$ hence
\[\bigcup_{i \in \N} A_i \subset X_1\,.\]
Also, the following holds
\[\left( \bigcup_{i \in \N} A_i \right) \cap Z = \bigcup_{i \in \N} (A_i \cap Z)\,.\]
Since $A_i \in \mathcal{G}$, $A_i \cap Z \in \bcal(X_2 \cap Z)$ and as $\bcal(X_2 \cap Z)$ is a $\sigma$-field
\[\bigcup_{i \in \N} \left(A_i \cap Z \right) \in \bcal(X_2 \cap Z)\,.\]
\end{itemize}
From $i) - iii)$ we can infer that $\mathcal{G}$ is a $\sigma$-field. We have already shown that $\tau(X_1) \subset \mathcal{G}$ thus
\[\mathcal{B}(X_1) =\sigma(\tau(X_1)) \subset \mathcal{G}\,.\]
Therefore, we have shown that for every $Y \in \bcal(X_1)$, $Y \cap Z \in \bcal(X_2 \cap Z)$.
\end{proof}

\begin{lemmaB}
\label{lemmab.4}
Let $X_1, X_2, Y$ be topological spaces such that $X_1 \subset X_2$, $X_1$ has trace topology from $X_2$ and $X_1 \cap Y = X_2 \cap Y$ then
\[\tau(X_1 \cap Y) = \tau(X_2 \cap Y)\,.\]
\end{lemmaB}

\begin{proof}
The topologies of $X_1 \cap Y$ and $X_2 \cap Y$ denoted by $\tau(X_1 \cap Y)$ and $\tau(X_2 \cap Y)$ respectively {are} given by
\begin{align*}
\tau(X_1 \cap Y) = \mbox{generated by\,} \left\{A \cap B \colon A \in \tau(X_1), B \in \tau(Y) \right\}, \\
\tau(X_2 \cap Y) = \mbox{generated by\,} \left\{C \cap B \colon C \in \tau(X_2), B \in \tau(Y) \right\}.
\end{align*}
Since $X_1$ has a trace topology from $X_2$, for every $A \in \tau(X_1)$ there exists a $C \in \tau(X_2)$ such that $A = C \cap X_1$. Thus
\[\tau(X_1 \cap Y) = \mbox{generated by\,} \left\{C \cap X_1 \cap B \colon C \in \tau(X_2), B \in \tau(Y) \right\}.\]
Thus all we are left to show is $C \cap X_1 \cap B = C \cap B$ for every $C \in \tau(X_2)$ and $B \in \tau(Y).$ Since $X_1 \cap Y = X_2 \cap Y$, we have the following set relations
\begin{align*}
C \cap X_1 \cap B & = \left(C \cap X_1 \right) \cap \left(Y \cap B \right) = (C \cap X_1 \cap Y ) \cap B \\
& = (C \cap X_2 \cap Y) \cap B = (C \cap X_2) \cap (Y \cap B) = C \cap B\,.
\end{align*}
\end{proof}

We will need the following space$\colon$
\begin{align*}
L^2_{loc}([0,T] \times \R^3) = \Big\{& u \colon [0,T] \times \R^3 \to \R^3\\
&\, \mbox{s.t. } \int_0^T \int_{|x| \le R} |u(x,t)|^2 dx dt < \infty , \forall\, R > 0\Big\}.
\end{align*}
$L^2_{loc}([0,T] \times \R^3)$ is complete under the family of semi-norms
\[\rho_R := \left[\int_0^T \int_{|x| \le R} |u(x,t)|^2\,dx\,dt\right]^{1/2}\,. \]
In particular, it is a Frec{h}\'{e}t space with the metric
\[d(u,v) = \sum_{n \ge 1} \frac{1}{2^n} \frac{\rho_n(u-v)}{1+\rho_n(u-v)}\,.\]

\begin{remarkB}
\label{remb.5}
$L^2(0,T; \rH_{loc}) \subset L^2_{loc}([0,T] \times \R^3)$ and we can define a topology on $L^2(0,T; \rH_{loc})$ by restricting the metric $d$ to $L^2(0,T; \rH_{loc})$. Hence $L^2(0,T; \rH_{loc})$ {is a topological space with the} trace topology from $L^2_{loc}([0,T] \times \R^3)$.
\end{remarkB}

Let us define a new topological space$\colon$
\[\widetilde{\mathcal{Z}}_T := \ccal([0,T]; {\rm U}^\prime) \cap L^2_{loc}([0,T]\times \R^3) \cap L^2_w(0,T; {\rm D}(\A )) \cap \ccal([0,T]; \rV_w)\,.\]

Note that $\widetilde{\mathcal{Z}}_T$ and $\mathcal{Z}_T$ are same as a set. Because $L^2_{loc}([0,T] \times \R^3) \cap L^2_w(0,T; \rm{D}(\A ))$ and $L^2(0,T; \rH_{loc}) \cap L^2_w(0,T; \rm{D}(\A ))$ are same as a set. $L^2(0,T; \rH_{loc}) \subset L^2_{loc}([0,T] \times \R^3)$ and the only extra elements in $L^2_{loc}([0,T] \times \R^3)$ are the ones which are locally square integrable but have non-zero divergence. But the intersection of $L^2_{loc}([0,T]\times \R^3)$ with $L^2_w(0,T; \rm{D}(\A ))$ eliminates those elements as the divergence free condition is imposed by the second set.

By Remark~B.\ref{remb.5} and Lemma~B.\ref{lemmab.4} $\widetilde{\mathcal{Z}}_T$ and $\mathcal{Z}_T$ have the same topologies. Thus we will prove Theorem~B.\ref{thmb.1} for $\widetilde{\mathcal{Z}}_T$ instead of $\mathcal{Z}_T$.

\begin{proof}[Proof of Theorem~B.\ref{thmb.1}]
First of all $\ccal([0,T]; \rV) \subset \ccal([0,T]; {\rm U}^\prime) \cap L^2_{loc}([0,T] \times \R^3)$. Secondly, $\ccal([0,T]; \rV)$ and $\ccal([0,T]; {\rm U}^\prime) \cap L^2_{loc}([0,T] \times \R^3)$ are Polish spaces. And finally, since $\rV$ is continuously embedded in ${\rm U}^\prime$, the map
\[i \colon \ccal([0,T]; \rV) \to \ccal([0,T]; {\rm U}^\prime) \cap L^2_{loc}([0,T] \times \R^3)\,, \]
is continuous and hence Borel. Thus, by application of the Kuratowski theorem (see Theorem~B.\ref{thmb.2}), $\ccal([0,T]; \rV)$ is a Borel subset of $\ccal([0,T]; {\rm U}^\prime) \cap L^2_{loc}([0,T] \times \R^3)$. Therefore, by Lemma~B.\ref{lemmab.3} $\ccal([0,T]; \rV) \cap \widetilde{\mathcal{Z}}_T$ is a Borel subset of $\ccal([0,T]; {\rm U}^\prime) \cap L^2_{loc}([0,T] \times \R^3) \cap \widetilde{\mathcal{Z}}_T$ which is equal to $\widetilde{\mathcal{Z}}_T$. We can show in the same way in the case of $\ccal([0,T];\rH_n) \cap \mathcal{Z}_T$.

Similarly we can show that $L^2(0,T; \rm{D}(\A )) \cap \widetilde{\mathcal{Z}}_T$ is a Borel subset of $\widetilde{\mathcal{Z}}_T$. $L^2(0,T; \rm{D}(\A )) \hookrightarrow L^2_{loc}([0,T] \times \R^3)$ and both are Polish spaces thus by application of the Kuratowski theorem [Theorem~B.\ref{thmb.2}], $L^2(0,T; \rm{D}(\A ))$ is a Borel subset of $L^2_{loc}([0,T]\times \R^3)$. Finally, we can conclude the proof of the theorem by Lemma~B.\ref{lemmab.3}.
\end{proof}

\end{document}